\documentclass[alphabetic]{article}

\usepackage{enumerate, longtable}
\usepackage{amsmath, amscd, amsfonts, amsthm, amssymb,  latexsym, comment, stmaryrd, graphicx, dsfont}
\usepackage{xcolor}
\usepackage{fullpage}
\usepackage[
        colorlinks,
        linkcolor=red,  citecolor=blue,
        backref
]{hyperref}

\usepackage{amstext} 
\usepackage{array}
\usepackage{amsrefs}

\usepackage{afterpage}
\usepackage{babel}

\newtheorem{theorem}{Theorem}[section]
\newtheorem{thm}[theorem]{Theorem}
\newtheorem*{theorem*}{Theorem}
\newtheorem*{informal_theorem*}{Informal Theorem}

\newtheorem{lem}[theorem]{Lemma}
\newtheorem{claim}[theorem]{Claim}
\newtheorem{corollary}[theorem]{Corollary}
\newtheorem{prop}[theorem]{Proposition}

\theoremstyle{definition}

\newtheorem{definition}[theorem]{Definition}
\newtheorem{rem}[theorem]{Remark}

\numberwithin{equation}{section}

\newcommand{\algfp}{\overline{\field_p}}

\newcommand{\hayesrelation}[1]{R_{d_1,x^{d_2+1},\pi^{#1}}}

\newcommand{\shortintervalcharacters}{\widehat{R}_{d,1,\pi}}

\newcommand{\mon}{\mathrm{mon}}
\newcommand{\sym}{\mathrm{sym}}

\newcommand{\TR}{\mathrm{TR}}
\newcommand{\pal}{\mathrm{pal}}
\newcommand{\UTR}{\mathrm{UTR}}
\newcommand{\F}{\mathbb{F}}

\newcommand{\Z}{\mathbb{Z}}

\newcommand{\g}{\mathfrak{g}}

\newcommand{\tr}{\mathrm{tr}}

\newcommand{\img}{\mathrm{Im}}
\newcommand{\skpal}{\mathrm{skpal}}
\newcommand{\condprob}[2]{\Pr\left[
{#1} \Bigg| {#2}
\right]}

\newcommand{\rs}{\mathrm{r}}

\newcommand{\case}[1]{\vspace{0.3cm}\noindent\framebox{#1.}}
\newcommand{\maincase}[1]{{\setlength{\fboxrule}{1.5pt}\vspace{0.3cm}\noindent\framebox{#1.}}}
\newcommand{\Ol}{\mathcal{O}}
\newcommand{\field}{\mathbb{F}}
\newcommand{\resring}[1]{\Ol/\pi^{#1}\Ol}
\newcommand{\resringr}[1]{\mathcal{R}/\pi^{#1}\mathcal{R}}

\newcommand{\C}{\mathbb{C}}

\newcommand{\charc}{\mathrm{char}}
\newcommand{\Adj}{\mathrm{Adj}}
\newcommand{\rad}{\mathrm{rad}}

\newcommand{\Pee}{\mathcal{P}}

\makeatother

\begin{document}
\title{Traces of powers of random matrices over local fields}
\author{Noam Pirani\footnote{Email: \href{mailto:noampir9@gmail.com}{noampir9@gmail.com}.}}
\maketitle

\begin{abstract}

Let $M$ be chosen uniformly at random w.r.t. the Haar measure on the unitary group $U_n$, the unitary symplectic group $USp_{2n}$ or the orthogonal group $O_n$. Diaconis and Shashahani proved that the traces $\tr(M),\tr(M^2),\ldots,\tr(M^k)$ converge in distribution to independent normal random variables as $k$ is fixed and $n\to\infty$. Recently, Gorodetsky and Rodgers proved analogs for these results for matrices chosen from certain finite matrix groups. For example, let $M$ be chosen uniformly at random from $U_n(\field_q)$. They show that $\{\tr(M^i)\}_{i=1,p\nmid i}^{k}$ converge in distribution to independent uniform random variables in $\field_{q^2}$ as $k$ is fixed and $n\to\infty$. They also consider the case where we let $k\to\infty$ with $n$. 
   
    We prove analogs for these results over local fields. Let $\mathcal{F}$ be a local field with a ring of integers $\Ol$, a uniformizer $\pi$, and a residue field of odd characteristic. Let $\mathcal{K}/\mathcal{F}$ be an unramified extension of degree $2$ with a ring of integers $\mathcal{R}$. Let $M$ be chosen uniformly at random w.r.t. the Haar measure on the unitary group $U_n(\Ol)$, and fix $k$. We prove that the traces of powers $\{\tr(M^i)\}_{i=1,p\nmid i}^k$ converge to independent uniform random variables on $\mathcal{R}$, as $n\to\infty$. We also consider the case where $k$ may tend to infinity with $n$. We show that for some constant $c$ (coming from the mod $\pi$ distribution), the total variation distance from independent uniform random variables on $\mathcal{R}$ is $o(1)$ as $n\to\infty$, as long as $k<c\cdot n$. We also consider other matrix groups over local fields, namely $G=GL_n(\Ol), SL_n(\Ol), Sp_{2n}(\Ol), SO_n(\Ol)$, and prove similar results for them. Moreover, we consider traces of powers $M^{pi}$ and traces of negative powers, 
    and show that apart from certain necessary modular restrictions, they also equidistribute in the limit.

    The proofs rely on considering the lifting of a matrix from $G(\resring{k-1})$ to $G(\resring{k})$, and analyzing how the characteristic polynomial behaves under such lifts. We also rely on results on the distribution of conjugacy classes in finite matrix groups due to Fulman and the finite fields results of Gorodetsky-Rodgers.
\end{abstract}

\newpage

\tableofcontents

\newpage

\section{Introduction}
Diaconis and Shahshahani \cite{DS94} studied the joint distribution of traces of powers of a matrix over classical matrix groups, initiating a line of work in random matrix theory. Concretely, consider $U_n=\{g\in GL_n(\mathbb{C}):g\overline{g}^t=I_n\}$,  the unitary group, and let $Z=X+iY$ be a standard complex normal random variable, with $X$, $Y$ independent and $X,Y\sim N(0,\frac{1}{2})$ are real Gaussians. Let $M$ be a random matrix chosen uniformly w.r.t. the Haar measure on $U_n$. Diaconis and Shahshahani proved that for a fixed $k$ the vector 
$
(\tr(M), \tr(M^2),\ldots, \tr(M^k))
$
converges in distribution, as $n\to\infty$, to the vector
$
(\sqrt{1}Z_1,\sqrt{2}Z_2,\ldots,\sqrt{k}Z_k)
$ 
where $Z_j\sim Z$ are i.i.d random variables. They also prove similar results for the orthogonal and unitary symplectic groups. Their results were extended by various methods, ranging from representation theory (See e.g \cite{DE01, BR01, Sto05, BG06, Rai97}) to analytic techniques (See e.g \cite{PV04}) or combinatorial methods (See e.g \cite{HR03}).

Confirming a conjecture of Diaconis, Johansson \cite{Joh97} proved that the rate of convergence is super-exponential in $n$, that is 

$$
    \Pr[\mathrm{Re}(\tr(M^k))\le x,\mathrm{Im}(\tr(M^k))\le y]-\Pr[\mathrm{Re}(\sqrt{k}Z_k)\le x,\mathrm{Im}(\sqrt{k}Z_k)\le y]=O_k(e^{-c_k n\log n}).
$$ More recently, Johansson and Lambert \cite{JL21} proved that the convergence rate is also super-exponential for the joint moments. More precisely, set $\mathbf{T}=(\tr(M^i))_{i=1}^k$, $\mathbf{Z}=(Z_i)_{i=1}^k$. For a random variable $X$ on $\C^k$, define a probability measure by setting $\mu_X(A)=\Pr[X\in A]$ for every Borel measurable $A\subset \C^k$. For $Y$ another random variable on $\C^k$, define the total variation distance $d_{TV}(X,Y)=\sup_A |\mu_X(A)-\mu_Y(A)|$, where $A$ ranges over all Borel measurable sets. Johansson and Lambert proved that $d_{TV}(\mathbf T,\mathbf K)$ converges super-exponentially (in $n$) to zero as we take $n,k\to\infty$ (with a complicated but explicit error term), provided that $k<\frac{\sqrt{n}}{6.45(\log n)^{1/4}}$. 

Recently, Gorodetsky and Rodgers \cite{GR21} considered similar questions over finite fields. For example, for the unitary case, they define $U_n(\field_q)=\{g\in GL_n(\field_{q^2}):g\overline{g}^t=I_n\}$, where $\overline{g}$ is defined via the action of Frobenius on the entries of $g$. In this case, $\tr(M),\tr(M^2),\ldots,\tr(M^k)$ do not converge to independent random variables since we have 
$
\tr(M^p)=\tr(M)^p
$. Instead, they consider traces of powers that are coprime with $p$. They prove that as $n\to\infty$, the random variables $(\tr(M^{b_i}))_{i=1}^k$ for $p\nmid b_i$ converge in total variation distance to $k$ independent, uniform random variables taking values in $\field_q$, provided that $\max(b_i)/n$ is bounded by some constant depending on $q$. They also consider the distribution of traces of powers for the groups $GL_n(\field_q)$,  $SL_n(\field_q)$, $O^\pm_n(\field_q)$, $Sp_{2n}(\field_q)$. Gorodetsky and Kovaleva \cite{GK23} considered a single trace of power (or a single linear combination of traces of powers) for $GL_n(\field_q)$ and improved the range of equidistribution much further, showing that $\tr(g^k)$ converges (in total variation distance) to a uniform random variable when $n\to\infty$, as long as $\log(k)=o(n^2)$. They also prove that this range is sharp.

Random $p$-adic matrices were previously considered mainly in the context of the average size of their cokernels, motivated by applications in number theory. See for example 
\cite{EJV11, PR12, Wood17, VP21}. The work \cite{EJV11} is particularly related to our paper since it considers a piece of the characteristic polynomial of a Haar random matrix $M\in GL_n(\Z_p)$ (corresponding to the $(x-1)$-primary part modulo $p$) and computes its limiting distribution.

The present paper considers $p$-adic analogs of the above results. For the rest of the paper, let $p$ be an odd prime, $q$ a power of $p$. Let $\mathcal{F}/\mathbb{Q}_p$ be an unramified finite extension with a ring of integers $\Ol$, such that the residue field of $\Ol$ is  $\field_q$ and the uniformizer is $\pi$. Let $G$ be one of the compact $p$-adic matrix groups $U_n(\Ol)$, $GL_n(\Ol)$, $SL_n(\Ol)$, $SO_n(\Ol)$, $Sp_{2n}(\Ol)$. Let $A$ be a matrix from $G$, chosen uniformly at random w.r.t. the Haar measure. We examine the joint distribution of the traces of powers of $A$, $(\tr(A^i))_{i=1\atop{p\nmid i}}^r$, when $r<c_G\cdot n$ for some constant $c_G$ depending on $G$ and $q$. Under these conditions, we prove that the total variation distance from a uniform distribution is $o(1)$ as $n\to \infty$. We also examine the distribution of a single trace of a power $\tr(M^r)$, $p\nmid r$, for a matrix $M\in GL_n(\Ol)$. We obtain a range that exceeds the range obtained as a corollary from the joint distribution result, building on the results of \cite{GK23}.

Similarly to the finite field case, we note that when $p|r$, there are dependencies between $\tr(A^r)$ and the previous traces of powers. Let $\sigma\in \mathrm{Gal}(\mathcal{F}/\mathbb{Q}_p)$ be the Frobenius element $\left(
\frac{\pi}{p}
\right)$. Then, if $p^i||r$, $i\ge 1$, we have 
$$\tr(A^r)\equiv \sigma(\tr(A^{r/p}))\pmod{\pi^i}.$$
This is proved in Lemma \ref{lem_trace_dependencies} below, which extends a phenomenon first observed by Arnold \cite{Arn04} and later extended in \cite{MP10}. However, there are no further restrictions on the traces of powers of $p$-adic matrices. This motivates the following

\begin{definition}
A trace datum of length $d$ over $\Ol$ is a sequence of the form $(\frac{1}{|i|_p}a_i)_{i=1}^d$, for $a_i\in\Ol$. For a matrix $M\in GL_n(\Ol)$ we attach a trace datum of length $d$ by setting the sequence $(\tr(M^i)-1_{p|i}\sigma(\tr(A^{i/p})))_{i=1}^d$.
\end{definition}

To handle the distribution of traces of powers which can be both positive and negative, we define 

\begin{definition}
Let $d_1,d_2>0$ be two integers such that $d=d_1+d_2$. We define a $(d_1,d_2)$-trace datum to be a sequence of the form $(\frac{1}{|i|_p}a_i)_{0\neq i=-d_1}^{d_2}$. The length of a $(d_1,d_2)$-trace datum is $d$. For a matrix $M$ we attach a $(d_1,d_2)$ trace datum by concatenating the $d_1$ trace datum for $M^{-1}$ with the $d_2$ trace datum for $M$.
\end{definition}

Finally, we define the uniform trace data, $\UTR_d$, which is simply a sequence $(\frac{1}{|i|_p}a_i)_{i=1}^d$, for $a_i\in\Ol$ chosen uniformly at random w.r.t. the Haar measure. We define similarly a uniform $(d_1,d_2)$-trace data $\UTR_{d_1,d_2}$.

For a random variable $X$ taking values in $\Ol^l$, we define the probability measure $\mu_X$ on $\Ol^l$ by $\mu_X(A)=\Pr[X\in A]$, for every Borel measurable $A$. We define the total variation distance between two measures $\mu_1,\mu_2$ on $\Ol^l$ by setting $d_{TV}(\mu_1,\mu_2)=\sup_A |\mu_1(A)-\mu_2(A)|$ where $A$ ranges over all Borel measurable sets $A\subset \Ol^l$. Our results for the joint distribution of the traces of powers are as follows (in all cases, the implicit constant depends only on the group $G$ and on $q$):

\begin{theorem}\label{main_theorem_gl}
Let $d<n$, and let $\TR_{d_1,d_2}^{GL}$ be the random variable attaching to a matrix $M\in GL_n(\Ol)$ its $(d_1,d_2)$-trace datum of length $d$. Then, 

$$
d_{TV}(\mu_{\TR_{d_1,d_2}^{GL}},\mu_{\UTR_{d_1,d_2}})=O\left(
q^{d-\frac{n^2}{2d}}\left(1+\frac{1}{q-1}\right)^n\binom{n+d}{n}+q^{d-\frac{n^2}{d}+o(n)}
\right).
$$
In particular, there is some constant $c_q^{GL}$, such that for $d<c_q^{GL}\cdot n$, this total variation distance is $o(1)$ as $n\to\infty$. 
\end{theorem}

\begin{theorem}\label{main_theorem_sl}
Let $d<n$, and let $\TR_{d_1,d_2}^{SL}$ be the random variable attaching to a matrix $M\in SL_n(\Ol)$ its $(d_1,d_2)$-trace datum of length $d$. Then, 

$$
d_{TV}(\mu_{\TR_{d_1,d_2}^{SL}}, \mu_{\UTR_{d_1,d_2}})=O\left(
q^{d-\frac{n^2}{2d}}\left(1+\frac{1}{q-1}\right)^n\binom{n+d}{n}+q^{d-\frac{n^2}{d}+o(n)}
\right).
$$
In particular, there is some constant $c_q^{SL}$, such that for $d<c_q^{SL}\cdot n$, this total variation distance is $o(1)$ as $n\to\infty$.\end{theorem}

\begin{theorem}\label{main_theorem_sp}
Let $d<n$, and let $\TR_d^{Sp}$ be the random variable attaching to a matrix $M\in Sp_{2n}(\Ol)$ its trace datum of length $d$. Then, 

$$
d_{TV}(\mu_{\TR_d^{Sp}}, \mu_{\UTR_d})=O\left(
q^{d-\frac{n^2}{2d}+\frac{n}{2}}\left(1+\frac{1}{q^2-1}\right)^n\binom{n+2d-1}{n}+q^{2d-\frac{n^2}{2d}+o(n)}
\right).
$$
In particular, there is some constant $c_q^{Sp}$, such that for $d<c_q^{Sp}\cdot n$, this total variation distance is $o(1)$ as $n\to\infty$.
\end{theorem}

\begin{theorem}\label{main_theorem_so}
Let $d<\frac{n}{2}$, and let $\TR_d^{SO}$ be the random variable attaching to a matrix $M\in SO_{n}(\Ol)$ its trace datum of length $d$. Then, 

$$
d_{TV}(\mu_{\TR_d^{SO}}, \mu_{\UTR_d})=O\left(
q^{d-\frac{n^2}{8d}+\frac{n}{4}
}\left(1+\frac{1}{q^2-1}\right)^{n/2}\binom{n/2+2d-1}{n/2}+q^{2d+\frac{n}{2}-\frac{n^2}{4d}+o(n)}
\right).
$$
In particular, there is some constant $c_q^{SO}$, such that for $d<c_q^{SO}\cdot n$, this total variation distance is $o(1)$ as $n\to\infty$.
\end{theorem}

\begin{theorem}\label{main_theorem_un}
Let $\mathcal{K}/\mathcal{F}$ be an unramified quadratic extension, and let $\mathcal{R}$ be its ring of integers. Let $d<\frac{n}{2}$, and let $\TR_d^{U}$ be the random variable attaching to a matrix $M\in U_{n}(\Ol)$ its trace datum of length $d$. Let $\UTR_d$ be a uniform trace datum over $\mathcal{R}$. Then, 

$$
d_{TV}(\mu_{\TR_d^{U}}, \mu_{\UTR_d})=O\left(
q^{2d+\frac{n}{2}-\frac{n^2}{4d}}\left(1+\frac{1}{q-1}\right)^{n}\binom{n+2d-3}{n}
\right).
$$
In particular, there is some constant $c_q^{U}$, such that for $d<c_q^{U}\cdot n$, this total variation distance is $o(1)$ as $n\to\infty$.
\end{theorem}

\begin{rem}
    The sizes of the constants $c_q^G$ for the various $G$ come from the finite field results of Gorodetsky-Rodgers. 
    \begin{enumerate}
        \item For $GL_n$ and $SL_n$, one can see that the range of $d$ for which the bound is useful is growing with $q$, and can reach $d<\frac{1}{\sqrt{2}}n$ for $q$ large.
        \item For $Sp_{2n}$, we must have $d<n$. This stems from the fact that if $M\in Sp_{2n}$ has a characteristic polynomial $\charc(M)=x^{2n}+a_{2n-1}x^{2n-1}+\cdots+a_1x+1$, then $a_{2n-i}=a_i$. Using Newton's formula, this shows that $\TR_d^{Sp}$ is not close to $\UTR_d$ when $d>n$. Our results are meaningful in a range $d<(\frac{1}{2}-o_q(1))n$.
        \item For $SO_n$, for reasons similar to those in the symplectic group, $\TR_d^{SO}$ is not close to $\UTR_d$ when $d>n/2$. The actual range we get is $d<(\frac{1}{4}-o_q(1))n$.
        \item For $U_n$ the situation is again similar to the symplectic case and thus $\TR_d^{U}$ is not close to $\UTR_d$ when $d>n/2$. The actual range we get is $d<(\frac{1}{4}-o_q(1))n$. 
    \end{enumerate} 
\end{rem}

As an immediate consequence of the last Theorem we obtain a direct analog for one of the results of Diaconis and Shahshahani \cite[Theorem 1]{DS94}:

\begin{theorem}\label{thm_diaconis_shahshahani_analog}
Fix $d>0$, and let $B_1,\ldots,B_d$ be fixed balls in $\mathcal{R}$. Let $Z_i\sim \frac{1}{|i|_p}U_i$, where $U_i$ are independent uniform random variables on $\mathcal{R}$. Then, 

$$
\lim_{n\to\infty}\Pr_{M\in U_n(\Ol)}
\left[
\tr(M^i)-1_{p|i}\sigma(\tr(M^{i/p}))\in B_i,i=1,\ldots,d
\right]
=\prod_{i=1}^d \Pr[Z_i\in B_i].
$$
\end{theorem}

We have analogous results for each one of the groups we consider, appearing in the corresponding section of the paper. We also get an extended range for a single trace in $GL_n(\Ol)$.

\begin{theorem}\label{main_theorem_single_trace_gl}
Let $U$ be the uniform distribution on $\Ol$, and let $p\nmid r$ be an integer, which may depend on $n$, such that $\log r=o(n^2)$. Then,
$$
d_{TV}(\mu_{\tr(A^r)},\mu_U)\to_{n\to\infty} 0.
$$ 
\end{theorem}

We prove Theorems \ref{main_theorem_gl}, \ref{main_theorem_sl}, \ref{main_theorem_sp},
\ref{main_theorem_so} and \ref{main_theorem_un} using similar ideas. Let $G$ be one of the groups $GL_n$, $SL_n$, $USp_{2n}$, $U_n$, $SO_n$. The basic idea is to understand, given $A_0\in G(\resring{k-1})$, what is the structure of the characteristic polynomial $\charc(A)$ for $A\in G(\Ol/\pi^k\Ol)$ a lift of $A_0$. As it turns out, the structure is regulated by the image of a certain linear map $\partial\charc_{A_0}$. We identify $\resring{}$ with $\field_q$, and we let $\g(\field_q)$ be the Lie algebra associated with the matrix group $G(\field_q)$. We also denote by $\field_q[x]_{<n}$ the set of polynomials of degree $<n$ with coefficients in $\field_q$. Abusing notation slightly, the lifts $A$ of $A_0$ to $G(\resring{k})$ are of the form $A=A_0(I+\pi^{k-1}A_1)$, where $A_1\in \g(\field_q)$ and $I$ is the identity matrix. The function $\partial\charc_{A_0}$ is a linear map from $\g(\field_q)$ with an image in $\field_q[x]_{<n}$, such that 

\begin{equation}\label{eqn_charpoly_lift_intro}
\charc(A)=\charc(A_0)+\pi^{k-1}\partial\charc_{A_0}(A_1).    
\end{equation}

In all cases, we compute the image of $\partial\charc_{A_0}$ explicitly. This involves finding explicit representatives for conjugacy classes of elements in $G(\field_{q^l})$, for an extension $\field_{q^l}/\field_q$ in which $\charc(A_0)\pmod{\pi}$ splits completely. It turns out that the dimension of the image is equal to $\deg(\min(A_0)\pmod{\pi})$ in the $GL_n$ case (and is determined explicitly by the degree of the minimal polynomial in the other cases).

After studying this map, we estimate the probability that $\charc(A)$ lands in some short interval (or some Hayes residue class). Using Lemma \ref{lem_connection_traces_charpoly} below (essentially Newton's formulae), this is equivalent to understanding the joint distribution of the first $k$ traces of powers of $M$, that is the distribution of the vector $(\tr(M^i))_{i=1}^k$. This is done by an induction argument, estimating the conditional probability that $\charc(A)$ lands in some short interval modulo $\pi^k$ given that it landed in the interval modulo $\pi^{k-1}$. By equation \eqref{eqn_charpoly_lift_intro} and since the dimension of $\img(\partial\charc_{A_0})$ is big when $\deg(\min A_0\pmod{\pi})$ is, when the degree of the minimal polynomial is large the conditional probability is $q^{-d}$ as expected heuristically.

In Section \ref{section_min_poly_deg} we show that with high probability, the degree of the minimal polynomial is large (in a precise sense described there). Using this, we get good control on the discrepancy of the probability that $\charc(M)$ lands in a certain interval or Hayes residue class, which allows us to compute the total variation distance of the trace data from the uniform measure on trace data. We remark that we can not get a useful discrepancy bound for every short interval, unlike the strong result in \cite{GR21}. Instead, our bound is, in a sense, an average discrepancy bound. The proof of Theorem \ref{main_theorem_single_trace_gl} is somewhat similar, based on the structure of $\tr(A^r)$ for lifts of $A_0$. For every one of the results, we rely on equidistribution results in the residue field $\field_q$, due to Gorodetsky-Rodgers \cite{GR21} or Gorodetsky-Kovaleva \cite{GK23}.

\begin{rem}
    The assumption that the extension $\mathcal{F}/\mathbb{Q}_p$ is unramified can probably be removed, however, we decided to leave it in this form to ease the presentation. Similarly, one can consider the same problem over an extension of $\field_q((t))$, the field of formal Laurent series over $\field_q$. Our methods work in this setting as well, giving identical results, but we decided to leave this out of the presentation to avoid unnecessary complications.
\end{rem}

The paper is structured as follows. In Section \ref{section_preliminary_background} we provide the main preliminary results we need, mainly about the distribution of conjugacy classes in matrix groups due to Fulman \cite{Ful00} and ideas from Gorodetsky-Rodgers \cite{GR21}.  In Section \ref{section_preliminary_results} we prove some general results on polynomials and traces of powers of matrices needed later. In Section \ref{section_explicit_representatives} we compute explicit representatives for certain conjugacy classes in $G(\field_q)$ for every group $G$ we consider. In Section \ref{section_min_poly_deg} we discuss the typical degree of a minimal polynomial in $G(\field_q)$. In the rest of the sections, we prove results on the joint distribution of traces of powers for $G(\Ol)$:
\begin{enumerate}
    \item In Section \ref{section_gln_main} we discuss results for $GL_n$ and prove Theorem \ref{main_theorem_gl} and Theorem \ref{main_theorem_single_trace_gl}.
    \item  In Section \ref{section_sln_main} we discuss results for $SL_n$ and prove Theorem \ref{main_theorem_sl}.
    \item  In Section \ref{section_sp_main} we discuss results for $Sp_{2n}$ and prove Theorem \ref{main_theorem_sp}.
    \item In Section \ref{section_so_main} we discuss results for $SO_n$ and prove Theorem \ref{main_theorem_so}.
    \item In the last section, Section \ref{section_un_main}, we discuss results for $U_n$ and prove Theorem \ref{main_theorem_un} and Theorem \ref{thm_diaconis_shahshahani_analog}.
\end{enumerate}

\textbf{Acknowledgments.} 
The author would like to thank Alexei Entin for suggesting the problems studied in the present paper, for many helpful discussions about them, and for comments on earlier versions of the paper which significantly improved the presentation. The author would also like to thank Itamar Hadas for many useful discussions regarding the problems and Noam Ta-Shma for discussions leading to the proof of Lemma \ref{lem_small_radical_bound}, and also Ze\'ev Rudnick, Ofir Gorodetsky and Brad Rodgers for helpful comments on earlier versions of the paper.

\section{Preliminary background}\label{section_preliminary_background}

The following section describes some preliminary background needed later in the paper. First, we want to fix some notation used throughout the work. 

Everywhere in the paper, $p$ is an odd prime and $q$ is a power of $p$. We denote by $\field_q[x]_{=n}$ the set of polynomials of degree $n$ in $\field_q[x]$. The subset $\field_q[x]_{=n}^\mon$ consists of the monic polynomials of degree $n$. We similarly denote by $\field_q[x]_{\le n},\field_q[x]_{< n}$ the polynomials of degree at most $n$ and strictly smaller than $n$ (respectively). A polynomial $f\in\field_q[x]$ is irreducible if it cannot decompose non-trivially. A polynomial is called prime if it is irreducible and monic. We denote the set of primes in $\field_q[x]$ by $\mathcal{P}$. 

We now fix a finite, unramified field extension $\mathcal{F}/\mathbb{Q}_p$ with a ring of integers $\Ol$. We also fix $\pi$ to be a uniformizer in $\Ol$, that is, an element of $\Ol$ with $p$-adic valuation equal to $1$. We assume that the residue field of $\mathcal{F}$, that is the field $\resring{}$, is isomorphic to $\field_q$. We let $\sigma\in \mathrm{Gal}(\mathcal{F}/\mathbb{Q}_p)$ be the Frobenius element $\left(
\frac{\pi}{p}
\right)$, that is the unique element such that $\sigma(x)\equiv x^p\pmod{\pi}$ for every $x\in\Ol$. Finally, we let $\mathcal{K} / \mathcal{F}$ be an unramified quadratic extension, and denote the generator of $\mathrm{Gal}(\mathcal{K}/\mathcal{F})$ by $\tau$.  

For a matrix $M$ (which can be in $M_n(\Ol)$, $M_n(\field_q)$, $M_n(\resring{k})$, etc.) we denote by $\charc(M)$ the characteristic polynomial of $M$ and by $\min(M)$ the minimal polynomial of $M$. Also, for us $G$ will always be one of the groups $GL$, $SL$, $Sp$, $SO$, $U$. We will typically use the parenthesis to declare the ring of definition (e.g., write $GL_n(\Ol)$, $GL_n(\resring{k})$ or $GL_n(\field_q)$) unless it should be clear from the context. When we want to emphasize the dependence on $n$, we will sometimes write $G_n$ (for the case of $Sp$, $G_n=Sp_{2n}$).

\subsection{Distribution of conjugacy classes for matrices over finite fields}

We recall some facts about the distribution of conjugacy classes for matrices over a finite field $\field_q$. For a much more extensive review, describing also some generating functions techniques and discussing the details for other groups, see \cite{Ful00}.

We first recall that conjugacy classes for $GL_n(\field_q)$ may be parameterized using rational canonical forms. For a monic  irreducible polynomial $\phi\in\field_q[x]$, associate a partition $\lambda_\phi$ of some non-negative integer $|\lambda_\phi|$. We impose the restrictions $|\lambda_x|=0$ and $\sum_{\phi\in\mathcal{P}} |\lambda_\phi|\deg(\phi)=n$. 

Now let $\phi=x^{\deg\phi}+a_{\deg\phi -1}x^{\deg\phi -1}+\cdots+a_1 x+a_0$. We recall that the companion matrix $C(\phi)$ is the matrix
$$
  \left(\begin{array}{ccccc}
    0 & 1 & 0 & \cdots & 0 \\
    0 & 0 & 1 & \cdots & 0 \\
    \vdots & \vdots & \ddots & \cdots & \vdots \\
    0 & 0 & 0 & 0 & 1 \\
    -a_0 & -a_1 & -a_2 & \cdots & -a_{\deg(\phi)-1} 
  \end{array}\right).
$$

Let $\{\phi_i\}_{i=1}^k$ be a collection of monic, irreducible polynomials, together with partitions $|\lambda_{\phi_i}|>0$, $\lambda_{\phi_i}=(\lambda_{\phi_i,j})_{j=1}^{r_\phi}$. We order the parts of $\lambda_{\phi_i}$ so that $\lambda_{\phi_i,1}\ge\lambda_{\phi_i,2}\ge\ldots$. A collection of partitions $(\lambda_{\phi_i})_{i=1}^k$ is called a \emph{conjugacy class datum} and we attach to it a matrix

$$
  M=\left(\begin{array}{ccccc}
    R_1 & 0 & 0 & \cdots & 0 \\
    0 & R_2 & 0 & \cdots & 0 \\
    \vdots & \vdots & \ddots & \cdots & \vdots \\
    0 & 0 & 0 & R_{k-1} & 0 \\
    0 & 0 & 0 & \cdots & R_k 
  \end{array}\right),
$$

where $R_i$ is the matrix

$$
  \left(\begin{array}{ccccc}
    C(\phi_i^{\lambda_{\phi_i,1}}) & 0 & 0 & \cdots & 0 \\
    0 & C(\phi_i^{\lambda_{\phi_i,2}}) & 0 & \cdots & 0 \\
    \vdots & \vdots & \ddots & \cdots & \vdots \\
  \end{array}\right).
$$

The conjugacy class of the matrix $M$ is 
a matrix attached to the conjugacy class datum. Every conjugacy class of $GL_n(\field_q)$ is obtained from a unique conjugacy class datum. Many algebraic properties of $M$ may be derived from its conjugacy class datum alone. For example, the characteristic polynomial is $\charc(M)=\prod_{\phi\in\Pee} \phi^{|\lambda_\phi(M)|}$, and the minimal polynomial is $\min(M)=\prod_{\phi\in\Pee} \phi^{\lambda_{\phi,1}(M)}$. Moreover, if one draws a matrix uniformly at random from $GL_n(\field_q)$, the probability of getting a matrix from the conjugacy class of $M$ is (see for example \cite[Section 2.1]{Ful01})

$$\label{eqn_fulman_conj_classes_prob}
\frac{1}{\prod_\phi q^{\deg\phi\sum_i (\lambda_{\phi,i}')^2}\prod_{i\ge 1}\left(\frac{1}{q^{\deg(\phi)}}\right)_{m_i(\lambda_\phi)}}.
$$
Here, $m_i(\lambda_\phi)$ is the number of parts of $\lambda_\phi$ of size $i$, and $\lambda_{\phi,i}'$ is the partition dual to $\lambda_\phi$, that is, the partition 
$|\lambda_{\phi}'|=\sum_{i=1}^{|\lambda_\phi|}m_i(\lambda_{\phi})$. Finally, 

$$\left(\frac{1}{q^{\deg(\phi)}}\right)_{m_i(\lambda_\phi)}:=\left(1-\frac{1}{q^{\deg(\phi)}}\right)\cdots \left(1-\frac{1}{q^{m_i(\lambda_\phi)\deg(\phi)}}\right).$$

Fulman \cite{Ful00, Ful99} derive similar formulas for other classical finite matrix groups, which we use in the paper.

\subsection{Hayes characters}\label{section_hayes_characters}

Here we recall briefly basic facts about Hayes characters, introduced by Hayes \cite{Hay65}. For a more extensive review, see \cite{Gor20}, on which we rely in this short introduction. We also remark that this is a special case of Hecke characters over function fields (see \cite[\S 9]{Ros02}).

\subsubsection{Hayes' equivalence relation}
Let $\mathcal{M}_q \subset\field_q[x]$ be the set of monic polynomials. We note that it is a monoid under multiplication, and define an equivalence relation respecting this structure. Let $l\ge 0$ be an integer, and let $0\ne H\in\field_q[x]$ be a monic polynomial. We define an equivalence relation $R_{l,H}$ on $\mathcal{M}_q$, by saying that $A\equiv B$
mod $R_{l,H}$
if and only if $A,B$
have the same $l$-next-to-leading coefficients, and also $A\equiv B$ mod $H$. We remark that we adopt the convention that if $\deg(A)<j$, the $j$-th next-to-leading coefficient of $A$ is $0$. Note that we do not demand that $\deg(A)=\deg(B)$. For example, we have
$$
x^5+x^3+1\equiv x^3+x+1 \pmod{R_{2,x}},
$$
for any $q$. One can show that $R_{l,H}$ respects the multiplication on $\mathcal{M}_q$, and that the quotient $\mathcal{M}_q/R_{l,H}$ is finite. Further, the unit group $(\mathcal{M}_q/R_{l,H})^\times$ of the quotient $\mathcal{M}_q/R_{l,H}$ (which is composed precisely of the conjugacy classes of polynomials $A$ which are coprime with $H$) is of size $q^l\varphi(H)$, where $\varphi$ is Euler's totient function (for an introduction to the Euler's totient function over $\field_q[x]$, see \cite[\S 2]{Ros02}).

When we look at the subset $\field_q[x]_{=n}^{\mon}\subset\mathcal{M}_q$ of degree $n$ monic polynomials, the equivalence relation $R_{l,1}$ detects short intervals. That is, the equivalence classes restricted to $\field_q[x]_{=n}^{\mon}$ are of the form 

$$I(f,l):=\{g\in\field_q[x]_{=n}^{\mon}:\text{ The $l$ next-to-leading coeficcients of } g \text{ are the same as those of }f\}.$$
These sets are analogous to intervals in $\mathbb{Z}$, hence the name.

\subsubsection{Characters}

Let $G_{l,H}$ be the finite abelian group $(\mathcal{M}_q/R_{l,H})^\times$. For a character $\chi$ of $G_{l,H}$, we define a function (which we also denote by $\chi$) over the full monoid $\mathcal{M}_q$. The definition is analogous to that of Dirichlet characters, 

$$
\chi(A)=\begin{cases}
\chi(A \text{ mod }{R_{l,H}}), &(A,H)=1,\\
0,& \text{otherwise.}
\end{cases}
$$
These are called Hayes characters. We denote the group of Hayes characters by $\hat{R}_{l,H}$. As usual, we will denote by $\chi_0$ the Hayes character corresponding to the trivial character of $(\mathcal{M}_q/R_{l,H})^\times$. As in the case of Dirichlet characters, orthogonality of characters now implies that for $n\ge l+\deg(H)$,

\begin{equation}\label{eqn_hayes_first_orthogonality_of_characters}
\frac{1}{q^{n-\deg(H)}\varphi(H)}\sum_{f\in \field_q[x]_{=n}^\mon}\chi(f)=\begin{cases}
    1,&\chi=\chi_0,\\
    0,& \text{otherwise.}
\end{cases}
\end{equation}
and for any $f,g\in \mathcal{M}_q$, coprime with $H$,

\begin{equation}\label{eqn_hayes_ortho_relation}
\frac{1}{q^l \varphi(H)}\sum_{\chi\in \hat{R}_{l,H}}\chi(f)\overline{\chi}(g)=\begin{cases}
    1,&f\equiv g\pmod{R_{l,H}},\\
    0,& \text{otherwise.}
\end{cases}
\end{equation}
The first orthogonality relation follows from the fact that for $n\ge l+\deg(H)$, $\field_q[x]_{=n}^\mon$ is a disjoint union of $q^{n-l-\deg(H)}$ systems of representatives for the Hayes equivalence relation $R_{l,H}$. 

Similarly to the way that Dirichlet characters are useful to detect modular relations, using the orthogonality relation \eqref{eqn_hayes_ortho_relation} short interval characters may be used to detect if an element belongs to a short interval in $\field_q[x]$. 

\subsubsection{Extension to $\resring{k}$}

We extend the Hayes relation for polynomials in $\left(\resring{k}\right)[x]$ in the following way. Let $d\ge 0$ be an integer and let $H\in(\resring{k})[x]$ be a polynomial. We define the Hayes relation modulo $\pi^k$, $R_{d,H,\pi^k}$ by saying that two monic polynomials $A,B\in(\resring{k})[x]$ are equivalent (denoted $A\equiv B\text{ mod }R_{d,H,\pi^k}$) if $A-B=0\pmod{H}$ and $A,B$ have the same $d$ next-to-leading coefficients. We again adopt the convention that if $\deg(A)<j$, the $j$-th next-to-leading coefficient of $A$ is $0$. 

The special case $k=1$ is precisely the case of the ``regular" Hayes relation defined above. Finally, we will sometimes denote the Hayes characters (defined only for $d=1$) by $\widehat{R}_{d,H,\pi}$. 

\subsection{Distribution over finite fields}
The distribution of the traces of powers and characteristic polynomials of a matrix in $GL_n(\field_q)$ (and other groups) is studied in \cite{GR21}. They start from the pointwise probability for a given characteristic polynomial. Let $f\in\field_q[x]$ be monic such that $f(0)\ne 0$ (this condition is necessary for a polynomial to be a characteristic polynomial of a matrix in $GL_n(\field_q)$). Factor $f$ to irreducible polynomials: $f=\prod_{i=1}^r P_i^{e_i}$. Denote $\deg(P_i)=m_i$. Then, a classical result proved by Reiner \cite[Theorem 2]{Rei61} and Gerstenhaber \cite[Section 2]{Ger61} (independently) shows that

\begin{equation}\label{eqn_pointwise_probability_gorodetsky_rodgers}
P_{GL}(f):=\Pr_{A\in GL_n(\field_q)}[\charc(A)=f]=\prod_{i=1}^r \frac{q^{m_i e_i(e_i-1)}}{|GL_{e_i}(\field_{q^{m_i}})|}.    
\end{equation}

Fixing a Hayes character $\chi$, Gorodetsky-Rodgers define an $L$-function

$$L_{GL}(u,\chi)=\sum_{f\in \mathcal{M}_{q}^{gl}}\chi(f)u^{\deg f}, \quad |u|<1/q.$$
(Here $\mathcal{M}_{q}^{\mathrm{gl}}$ is the set of all monic polynomials $f$ with $f(0)\ne 0$). Using the formula \eqref{eqn_pointwise_probability_gorodetsky_rodgers} and its structure, Gorodetsky-Rodgers show that the generating function 

$$
\sum_{f\in\mathcal{M}_q^{\mathrm{gl}}}P_{GL}(f)\chi(f)u^{\deg f}
$$
factors as a product $\prod_{i=1}^\infty L_{GL}(\frac{u}{q^i},\chi)$. Now using character orthogonality and a variant of the Riemann hypothesis for the $L$-function $L_{GL}(u,\chi)$, they prove the following 

\begin{theorem}[\cite{GR21}, Theorem 1.3]\label{thm_gorodetsky_rodgers_for_charpoly}
    Let $A\in GL_n(\field_q)$ be a random matrix chosen according to the Haar measure. Then for $0\le d<n$ and for $f\in\mathcal{M}_{n,q}$,

    $$
    \left| \Pr[\charc(A)\in I(f,n-d)]-q^{-d}\right|<q^{-\frac{n^2}{2d}}\left(1+\frac{1}{q-1}\right)^n\binom{n+d-1}{n}.
    $$
\end{theorem}

Using similar ideas Gorodetsky-Rodgers also prove a theorem regarding the equidistribution of traces of powers:

\begin{theorem}[\cite{GR21}, Theorem 1.1]\label{thm_gorodetsky_rodgers_for_traces}
Let $A\in GL_n(\field_q)$ be a random matrix chosen according to the Haar measure. Fix a strictly increasing sequence $b_1,\ldots,b_k$ of positive integers coprime with $p=\charc(\field_q)$. Then for any sequence $a_1,\ldots,a_k$ of elements from $\field_q$, we have
    $$
    \left| \Pr[\forall 1\le i\le k:\tr(A^{b_i})=a_i]-q^{-k}\right|<q^{-\frac{n^2}{2b_k}}\left(1+\frac{1}{q-1}\right)^n\binom{n-1+b_k}{n}.
    $$
\end{theorem}

They have analogous results for other classical finite matrix groups, which we use throughout the present paper. In Section \ref{section_neg_traces_finite_groups} we discuss the results and methods of \cite{GR21} more extensively and we also prove a slight strengthening of the results for $GL_n$ and $SL_n$, to include traces of negative powers.

\subsection{$p$-adic matrix groups and measures}

We recall that throughout the paper, $p$ is an odd prime. We fix some finite, unramified field extension $\mathcal{F}/\mathbb{Q}_p$ with a ring of integers $\Ol$. We also denote by $\mathcal{K}/\mathcal{F}$ a quadratic, unramified extension with a ring of integers $\mathcal{R}$. We recall the definitions of the various $p$-adic groups we consider and set some notation for the rest of the paper.

\begin{definition}
The group $GL_n(\mathcal{O})$ consists of all invertible matrices $M\in M_n(\mathcal{O})$ with $M^{-1}\in M_n(\mathcal{O})$. The group $SL_n(\mathcal{O})$ is the subgroup of $GL_n(\Ol)$ of matrices with determinant $1$. 
\end{definition}

\begin{definition}
Let $\Omega_n=\left(
\begin{array}{cc}
     & I_n \\
    -I_n & 
\end{array}
\right)$ be the standard symplectic form. The group $Sp_{2n}(\Ol)$ is the subgroup of $GL_{2n}(\Ol)$ of matrices preserving the form defined by $\Omega_n$, that is of all matrices $M$ with $M\Omega_n M^t=\Omega_n$. 
\end{definition}

\begin{definition}
The group $O_n(\Ol)$ is the subgroup of $GL_n(\Ol)$ of matrices $M$ satisfying $MM^t=I$. The group $SO_n(\Ol)$ is the subgroup of $O_n(\Ol)$, consisting of matrices with determinant $1$.
\end{definition}

\begin{definition}
Let $\tau$ be the non-trivial Galois automorphism of $\mathcal{K}/\mathcal{F}$. 
For a matrix $M\in GL_n(\mathcal{R})$, define $M^*$ to be the matrix $\tau(M^t)$. The group $U_n(\Ol)$ is the subgroup of $GL_n(\mathcal{R})$ of matrices $M$ satisfying $MM^*=I$.
\end{definition}
We remark that these are not all the possible compact $p$-adic algebraic groups, and modifications of our methods could be used to treat some other groups.

\subsubsection{Lifting from the mod $\pi$ group}

Throughout the paper $G$ is one of the groups $GL_n$, $SL_n$, $Sp_{2n}$, $O_n$ or $U_n$. We note that the $p$-adic groups project onto their mod $p$ analogs. For an element $x\in\Ol$, define its mod $\pi^k$ reduction $\rho_k(x)=x\pmod{\pi^k}$. For a matrix $M\in GL_n(\Ol)$, we define its mod $\pi^k$ reduction $\rho_k(M)$ by applying $\rho_k$ on every entry of $M$. This defines a map $\rho_k$ for every one of the groups $G$ considered (we remark that for $U_n$ the image is inside $GL_n(\mathcal{R}/\pi^k)$ and not $GL_n(\resring{k})$). As a special case, when $k=1$, we denote $\overline{x}:=\rho_1(x)$ and for a matrix $M\in GL_n(\Ol)$ we denote $\overline{M}:=\rho_1(M)$.

We note that the $p$-adic groups project onto their $\resring{k}$ analogs. This allows us to carry out some of our main ideas in Section \ref{section_lifting} below. 

\begin{lem}\label{lem_PRR_reductive_surjective}
    Let $G$ be one of the groups considered. Then the map $\rho_k:G(\Ol)\to G(\resring{k})$ is surjective.
\end{lem}

\begin{proof}
The claim follows from \cite[\S 3.3]{PR93} since every $G$ we consider is reductive.
\end{proof}

\subsubsection{Measures}

Each of the groups we consider is equipped with a Haar probability measure, which is simply the limit of the uniform measures on the groups $G(\resring{k})$. We will translate our main theorems to claims about the uniform measures on the finite groups $G(\resring{k})$ and prove the claims about the $\resring{k}$ groups. 

Let $(\Omega,\Sigma)$ be a measure space, and let $\mu_1,\mu_2$ be two probability measures on $(\Omega,\Sigma)$. Then we recall that one defines the total variation distance of $\mu_1,\mu_2$ by

\begin{equation}
    d_{TV}(\mu_1,\mu_2)=\sup_{A\in\Sigma}|\mu_1(A)-\mu_2(A)|.
\end{equation}

\begin{definition}
Let $\mu$ be a Borel probability measure on $\Ol^d$. We define a probability measure $\rho_k(\mu)$ on $(\resring{k})^d$ by setting $\rho_k(\mu)(x)=\mu(x+\pi^k\Ol^d)$ for every $x\in (\resring{k})^d$. 
\end{definition}

We remark that $\rho_k(\mu)=(\rho_k)_*\mu$ is the push-forward measure of $\mu$ through the measurable map $\rho_k$. 

\begin{prop}\label{prop_measure_of_limit_of_sets}
    Let $\mu$ be a Borel probability measure on $\Ol^d$. Let $A\subset \Ol^d$ be a measurable set, and denote $A_k\subset (\resring{k})^d$ the set $\rho_k(A)$. Then, $\lim_{k\to\infty} \rho_k(\mu)(A_k)=\mu(A)$.
\end{prop}

\begin{proof}
Indeed, we have $\rho_k(\mu)(A_k)=\mu(A+\pi^k\Ol^d)$. Since $A\subset A+\pi^k
\Ol^d$, we have $
\rho_k(\mu)(A_k)-\mu(A)=\mu((A+\pi^k\Ol^d)-A)
$. Define 
$
B_k:=(A+\pi^k\Ol^d)-A.
$
We note that 
$$
\ldots\subset B_k\subset\ldots\subset B_2\subset B_1
$$
is a decreasing chain of sets and that $\cap_{k=1}^\infty B_k=\emptyset$. Hence, $\lim_{k\to\infty} \mu(B_k)=0$, and the claim is proved.
\end{proof}

\begin{lem}\label{lem_total_variation_dist_limit}
Let $\mu_1,\mu_2$ be two Borel probability measures on $\mathcal{O}^d$. Then, $$d_{TV}(\mu_1,\mu_2)=\lim_{k\to\infty}d_{TV}(\rho_k(\mu_1),\rho_k(\mu_2)).$$  
\end{lem}

\begin{proof}

We note that for every choice of $k$, $d_{TV}(\mu_1,\mu_2)\ge d_{TV}(\rho_k(\mu_1),\rho_k(\mu_2))$. Indeed, every subset $A$ of $(\resring{k})^d$ is measurable, and if we let $\tilde{A}$ be some lift of $A$ to $\Ol^d$, then $\rho_k(\mu_i)(A)=\mu_i(\tilde{A}+\pi^k\Ol^d)$ (note that $\tilde{A}+\pi^k\Ol^d$ is always measurable). Hence 

$$|\rho_k(\mu_1)(A)-\rho_k(\mu_2)(A)|=|\mu_1(\tilde{A}+\pi^k\Ol^d)-\mu_2(\tilde{A}+\pi^k\Ol^d)|\le d_{TV}(\mu_1,\mu_2),$$
and by taking the supremum over all choices of $A$ we get $d_{TV}(\rho_k(\mu_1),\rho_k(\mu_2))\le d_{TV}(\mu_1,\mu_2)$. 

In the other direction, fix some $\epsilon>0$ and take a measurable set $A\subset \Ol^d$ such that $|\mu_1(A)-\mu_2(A)|\ge d_{TV}(\mu_1,\mu_2)-\epsilon$. Define $A_k=\rho_k(A)$. We have, by the triangle inequality,

$$
|\mu_1(A)-\mu_2(A)|\le |\mu_1(A)-\rho_k(\mu_1)(A_k)|+|\rho_k(\mu_1)(A_k)-\rho_k(\mu_2)(A_k)|+|\mu_2(A)-\rho_k(\mu_2)(A_k)|.
$$
We get that 

$$
|\rho_k(\mu_1)(A_k)-\rho_k(\mu_2)(A_k)|\ge d_{TV}(\mu_1,\mu_2)-\epsilon-|\mu_1(A)-\rho_k(\mu_1)(A_k)|-|\mu_2(A)-\rho_k(\mu_2)(A_k)|.
$$
Taking the limit $k\to\infty$ and using Proposition \ref{prop_measure_of_limit_of_sets} we get that $d_{TV}(\rho_k(\mu_1),\rho_k(\mu_2))\ge d_{TV}(\mu_1,\mu_2)-\epsilon$. This is true for every $\epsilon>0$, hence we get $d_{TV}(\rho_k(\mu_1),\rho_k(\mu_2))\ge d_{TV}(\mu_1,\mu_2)$ and thus both sides are equal.
\end{proof}

\section{Preliminary results}\label{section_preliminary_results}

\subsection{Polynomials with small radical}

In what follows, we would need some results on the number of polynomials in $\field_q[x]$ with a small radical. We recall that for a polynomial $f\in\F_q[x]$, one can factor $f$ into irreducibles as $P_1^{e_1}\cdots P_k^{e_k}$. The radical of $f$ is the polynomial $\rad(f)=P_1\cdots P_g$.

\begin{definition}
Let $n,d$ be natural numbers such that $d<n$. Let $g\in\field_q[x]_{\le d}^\mon$ be a fixed, squarefree polynomial. We denote

$$
F_{g}(n)=|\{
f\in\field_q[x]_{=n}^\mon:\rad(f)|g
\}|.
$$
\end{definition}

We order the prime polynomials in $\field_q[x]$ by the degree, breaking ties arbitrarily, 

\begin{equation}\label{eqn_primes_order}
\varphi_1,\varphi_2,\ldots,\varphi_k,\ldots
\end{equation}

\begin{prop}\label{prop_number_of_polynomials_with_given_radical}
Let $n,d$ be natural numbers such that $d<n$. Let $g\in\field_q[x]_{\le d}^\mon$ be a fixed, squarefree polynomial. Let $g=P_1\cdots P_r$ be the prime factorization of $g$.
\begin{enumerate}
    \item $F_g(n)$ is the number of solutions to the equation $e_1\deg(P_1)+\cdots+e_r\deg(P_r)=n$ (in non-negative integers $e_1,\ldots,e_r$).
    \item Assume that $P_1,\ldots,P_r$ are ordered w.r.t. the order defined in \eqref{eqn_primes_order}. Assume further that $(P_1,\ldots,P_r)\neq (\varphi_1,\ldots,\varphi_r)$. Then, there exist a squarefree polynomial $h\in\field_q[x]^\mon_{\le d}$, with prime factorization $h=P_1\cdots P_{r-1}Q_r$, where $Q_r$ is different from $P_1,\ldots,P_r$ and $Q_r$ is smaller than $P_r$ w.r.t. the order defined in \eqref{eqn_primes_order}, and such that $F_h(n)\ge F_g(n)$. 
\end{enumerate}
\end{prop}

\begin{proof}
    \begin{enumerate}
        \item A monic polynomial $f\in\field_q[x]_{=n}$ such that $\rad(f)|g$ has prime factorization $P_1^{e_1}\cdots P_r^{e_r}$, where $e_i$ are non-negative integers. Conversely, every prime factorization of this form gives a unique polynomial counted by $F_g(n)$, hence the result.
        \item We divide the proof into two cases.

\maincase{For all $\alpha\in\field_q$, $x-\alpha$ is one of $(P_i)_{i=1}^r$} Choose $Q_r$ to be the smallest polynomial among $\varphi_1,\ldots,\varphi_r$ which is not one of $P_1,\ldots,P_r$. We construct an injection from the set of solutions of the equation 

$$
e_1\deg(P_1)+\cdots+e_r\deg(P_r)=n,
$$

into the set of solutions of the equation 

$$
e_1'\deg(P_1)+\cdots+e_{r-1}'\deg(P_{r-1})+e_r'\deg(Q_r)=n.
$$

Without loss of generality, we assume that $\deg(P_1)=1$. We construct the map by sending 

$$
(e_1,\ldots,e_r)\mapsto (e_1+e_r(\deg P_r-\deg Q_r),e_2,\ldots,e_{r-1},e_r).
$$

This is a bijection since $e_2,\ldots,e_r$ are determined from the image, and in particular since $e_r$ is determined also $e_1$ can be reconstructed (note that by our choice, $\deg Q_r\le \deg P_r$ so all numbers are non-negative).

\maincase{There is some $\alpha\in\field_q$ such that $x-\alpha$ is not one of $(P_i)_{i=1}^r$} Denote this particular $\alpha$ by $\alpha_0$. Take $Q_r=x-\alpha_0$. We construct an injection from the set of solutions of the equation 

$$
e_1\deg(P_1)+\cdots+e_r\deg(P_r)=n,
$$

into the set of solutions of 

$$
e_1'\deg(P_1)+\cdots+e_{r-1}'\deg(P_{r-1})+e_r'\deg(x-\alpha_0)=n.
$$

This is done by sending the solution $(e_1,\cdots,e_r)\mapsto (e_1,\cdots,e_{r-1},\deg(P_r)\cdot e_r)$. This is an injection hence we get the bound.

    \end{enumerate}
\end{proof}

\begin{lem}\label{lem_small_radical_bound}
Let $n,d$ natural numbers such that $d<n$. Then
$$
|\{f\in\field_q[x]_{=n}^\mon :\deg\rad(f)\le d\}|=q^{d+o(n)}.
$$
Here, the rate of decay in $n$ is uniform in $d$.
\end{lem}

\begin{proof}

We have 

$$
|\{f\in\field_q[x]_{=n}^\mon:\deg\rad(f)\le d\}|\le \sum_{h\in\field_q[x]_{\le d}^\mon\atop{h\text{ squarefree}}}F_h(n).
$$
We will prove that $F_h(n)=q^{o(n)}$ for every squarefree $h$ of degree $\le d$ (uniformly in $d$), and this will prove the claim since there are at most $q^d$ such polynomials $h$. From now on fix some particular squarefree $h\in\field_q[x]_{\le d}^\mon$. Write $h=P_1\cdots P_r$ its prime factorization. Then by Proposition \ref{prop_number_of_polynomials_with_given_radical}, $F_h(n)$ is the number of non-negative integer solutions to the equation

$$
n=\deg(f)=e_1\deg(P_1)+\cdots+e_r\deg(P_r).
$$

Using Proposition \ref{prop_number_of_polynomials_with_given_radical} again, we see that we can replace $P_r$ by some smaller polynomial (w.r.t. the order defined in \eqref{eqn_primes_order}), without decreasing the number of possible solutions. Now Set $\pi_q(k)$ to be the number of monic prime polynomials (from the prime polynomial theorem, $\pi_q(k)=\frac{q^k}{k}+O(q^{k/2})$). By the above, and possibly adding some new variables $e_i$, we may bound $F_h(n)$ by the number of non-negative integer solutions to the equation

\begin{equation}\label{eqn_number_of_polys_with_given_radical}
(e_1^{(1)}+\cdots+e_{\pi_q(1)}^{(1)})+2(e_1^{(2)}+\cdots+e_{\pi_k(2)}^{(2)})+\cdots+l(e_1^{(l)}+\cdots+e_{\pi_q(l)}^{(l)})=n,   
\end{equation}

where $l< \log_q(n)+1$ is the least integer such that $\sum_{k=1}^l k\pi_q(k)>n$. Now let $\sigma\vdash n$ be a partition of $n$, such that $i|\sigma(i)$. If we let $e_1^{(i)}+\cdots+e_{\pi_q(i)}^{(i)}=\frac{\sigma(i)}{i}$ for $i=1,\ldots, l$ be our new equivalent system of equations, we get that the number of possible solutions for Equation \eqref{eqn_number_of_polys_with_given_radical} is 

\begin{equation}\label{eqn_number_of_radicals_partition_viewpoint}
\sum_{\sigma\vdash n}\prod_{i=1}^l \binom{\frac{\sigma(i)}{i}+\pi_q(i)-1}{\pi_q(i)-1}.    
\end{equation}

By the Hardy-Ramanujan formula for the partition function \cite{HR18}, the number of summands in Equation \ref{eqn_number_of_radicals_partition_viewpoint} is $q^{o(n)}$. Hence it is enough to show that for every partition $\sigma\vdash n$, 

$$
\prod_{i=1}^l \binom{\frac{\sigma(i)}{i}+\pi_q(i)-1}{\pi_q(i)-1}=q^{o(n)}.
$$

We note that since for any $l,k\ge 0$ integers, $\binom{l}{k}<\frac{l^k}{k!}$. Using Striling's formula and the prime polynomials theorem we get

$$
\ln\left[\binom{\frac{\sigma(i)}{i}+\pi_q(i)-1}{\pi_q(i)-1}\right]=\ln \left[\binom{\frac{\sigma(i)}{i}+\frac{q^i}{i}+O(q^{i/2})}{\frac{q^{i}}{i}+O(q^{i/2})}\right] < \left(\frac{q^i}{i}+O(q^{i/2})\right)\ln\left(\frac{\sigma(i)+q^i+O(q^{i/2})}{q^i+O(q^{i/2})}\right)+\frac{q^i}{i}+O(q^{i/2}).
$$

We note that

$$
\sum_{i=1}^l \left(
\frac{q^i}{i}+O(q^{i/2})
\right)\ll \frac{q^l}{l}<n/l=o(n).
$$

since $l\to\infty$ as $n\to\infty$ (to be precise, $l$ grows logarithmically in $n$). Regarding the second term, it sums to
$
\ll \sum_{i=1}^l \frac{q^i}{i}\ln\left(1+\frac{\sigma(i)}{q^i}\right)=:S
$. We fix some $A>0$, and divide the sum $S$ to the terms with $q^i>A$ and those with $q^i\le A$:

$$
\ll S=\sum_{i=1}^l \frac{q^i}{i}\ln\left(1+\frac{\sigma(i)}{q^i}\right)=\sum_{q^i\le A}\frac{q^i}{i}\ln\left(1+\frac{\sigma(i)}{q^i}\right)+\sum_{q^i> A}\frac{\sigma(i)}{i}\frac{q^i}{\sigma(i)} \ln\left(1+\frac{\sigma(i)}{q^i}\right).
$$

The function $x\mapsto \frac{\ln(1+x)}{x}$ is bounded above by $1$, hence 

$$
\sum_{q^i> A}\frac{\sigma(i)}{i}\frac{q^i}{\sigma(i)} \ln\left(1+\frac{\sigma(i)}{q^i}\right)\le \sum_{q^i>A}\frac{\sigma(i)}{i}<\frac{1}{\log_q(A)}\sum_{q^i>A}\sigma(i)<\frac{n}{\log_q(A)}.
$$

If $q^i\le A$ we have $\ln\left(1+\frac{\sigma(i)}{q^i}\right)<\ln(n+1)$ hence 

$$
\sum_{q^i\le A}\frac{q^i}{i}\ln\left(1+\frac{\sigma(i)}{q^i}\right)<\ln(n+1)\sum_{q^i\le A}\frac{q^i}{i}\ll \ln(n+1)\cdot A.
$$

Overall we get that $S<\ln(n+1)\cdot A+\frac{n}{\log_q(A)}$. In particular, 
$
\frac{S}{n}<\frac{A\ln(n+1)}{n}+\frac{1}{\log_q(A)}
$, so that $\limsup \frac{S}{n}<\frac{1}{\log_q(A)}$. Since $A>0$ was arbitrary, letting $A\to\infty$ we get that $S=o(n)$.
\end{proof}

\subsection{Characteristic polynomial coefficients and traces of powers}\label{section_trace_data_poly_intervals}

A basic concept we use is that of an interval in $(\resring{k})[x]$. This is an extension of an interval in $\field_q[x]$, which we define as follows:

\begin{definition}
 Let $g\in(\resring{k})[x]$ be a monic polynomial of degree $n$. We define 
 $$
 I(g,n-d)=\left\{
h\in(\resring{k})[x]:\text{ monic },\deg(h)=n,\deg(g-h)<n-d
 \right\}.
 $$
 $I(g,n-d)$ is called the interval of length $n-d$ around $g$.
\end{definition}

To consider also traces of negative powers, we also define

\begin{definition}
Let $h\in (\resring{k})[x]$ be a monic polynomial, such that $h(0)$ is invertible in $\resring{k}$. We define $h^\rs=\frac{x^n h(1/x)}{h(0)}$. We define 

$$
I^\rs(g,n-d)=\{h\in(\resring{k})[x]\text{ monic },h(0)\in(\resring{k})^\times,h^\rs\in I(g,n-d)\}.
$$

\end{definition}

For later use, we will establish a formula relating the probability

$$
\Pr_{M\in G(\resring{k})}[\tr(M^i)=a_i,i=1,\ldots,d], a_i\in\resring{k},
$$
to the probability 
$$
\Pr_{M\in G(\resring{k})}[\charc(M)\in I(g,n-d)], g\in (\resring{k})[x]\text{ monic}.
$$
This will be done in Lemma \ref{lem_connection_traces_charpoly}, which connects the distribution of traces of positive powers and the distribution of $\charc(M)$ in short intervals. In Lemma \ref{lem_connection_traces_hayes_classes} we will generalize this connection, to include traces of negative powers. This time, the distribution of traces of powers will be related to the distribution of $\charc(M)$ in Hayes residue classes. Both Lemmas are analogs of \cite[Lemma 2.1]{GR21} and we will follow its proof closely. For our setting however more data need to be specified to determine the traces, thus we define the concept of \textit{trace data}. 

\begin{definition}
Let $k,n\ge 1$ be integers, and let $1\le d\le n$. 
\begin{enumerate}
    \item A trace data modulo $\pi^k$ of length $d$ is a sequence $(a_i)_{1\le i\le\atop p^k\nmid i}$ of elements $a_i$ in $\resring{k}$, such that for every $i$ with $p^k\nmid i$, there is some $b_i\in\resring{k}$ so that $a_i=p^r b_i$, where $p^r||i$. 
    \item A $p$-adic trace data of length $d$ is a sequence $(\pi^{v_p(i)}a_i)_{i=1}^d$, where $a_i\in\Ol$ is some $p$-adic integer. Note that since $p$ is unramified, we get the same object by picking a sequence $(\frac{a_i}{|i|_p})_{i=1}^d$. 
\end{enumerate}
\end{definition}

We remark that if $(a_i)_{i=1}^d$ is a $p$-adic trace datum of length $d$, then the sequence $(a_i\text{ mod }\pi^k)_{i=1\atop{p^k\nmid i}}^d$ is a trace datum modulo $\pi^k$ of length $d$. This could be an alternative definition for the trace data modulo $\pi^k$.

\begin{definition}
    Define $S=S(d,k)=\left[\frac{d}{p}\right]+\left[\frac{d}{p^2}\right]+\cdots+\left[\frac{d}{p^k}\right]$. The number of different possible values for a $d$-trace data modulo $\pi^k$ is $q^{kd-S}$.
\end{definition}

We note the significance of the trace data. Since we are working in $\resring{k}$ or $\Ol$, there are dependencies between the traces of different powers. For example, $\tr(M^p)=\tr(M)^p \pmod{\pi}$. This is generalized in the next lemma. The lemma generalizes a phenomenon first observed by Arnold \cite{Arn04} for $\Z_p$, though he proved it only for the ring $\Z/p^3\Z$. Later the result was extended for the $p$-adics in \cite{MP10}. The lemma generalizes it to $\Ol$.

We recall that we denote by $\sigma$ the Frobenius automorphism $\left(
\frac{\pi}{p}
\right):\Ol\to\Ol$, and also its induced mod $\pi^k$ action. We extend $\sigma$ to act coordinate-wise on matrices and polynomials.

\begin{lem}\label{lem_trace_dependencies}
Let $n,i,k\ge 1$ be natural numbers, and let $M\in M_n(\Ol)$ be a matrix. Let $p^{j'}||i$, and assume that $j'\ge 1$. Then 
$$
\tr(M^i)\equiv 
\sigma(\tr(M^{i/p}))
\pmod{\pi^{j'}}.
$$
In particular, if we now let $M\in M_n(\resring{k})$, and $j=\min(j',k)$, then 

$$
\tr(M^i)\equiv\sigma(\tr(M^{i/p}))\pmod{\pi^{j'}}.
$$
\end{lem}

\begin{proof}
Our proof is based on the main result of \cite{MP10} which essentially proves the lemma for the case $\Ol=\Z_p$. Let $M\in M_n(\Ol)$. We will prove that for any integer $k\ge 1$, 
\begin{equation}\label{eq_claim_traces_dependencies}
\charc(M^{p^k})\equiv \charc(\sigma(M)^{p^{k-1}})\equiv\sigma(\charc(M^{p^{k-1}}))\pmod{\pi^k}.
\end{equation}
This will prove the lemma by looking at the matrix $M^{i/p^{j}}$ and looking at the next-to-leading coefficient of the polynomials in \eqref{eq_claim_traces_dependencies}.

Let $e_i(x_1,\ldots,x_n),i=1,\ldots,n$ be the $i$-th symmetric polynomial (so $e_i$ is homogenous of degree $i$, $e_1=x_1+\cdots+x_n$, etc.). Define as in \cite{MP10} the polynomials $T_{m,i}$ by the formula 
$$
e_i(x_1^m,\ldots,x_n^m)=T_{m,i}(e_1(x_1,\ldots,x_n),\ldots,e_n(x_1,\ldots,x_n)).
$$
One can always write such an identity with $T_{m,i}\in \Z[x_1,\ldots,x_n]$ by the main theorem on symmetric polynomials (see \cite[Theorem 3.1.1]{Pra04}). Now for a monic polynomial $f=x^n+\sum_{i=1}^{n-1} (-1)^i a_ix^{n-i}$, define $T_m(f)$ (again, following \cite{MP10}) by setting 
$$
T_m(f)=x^n+\sum_{i=0}^{n-1}(-1)^iT_{m,i}(a_1,\ldots,a_n)x^{n-i}.
$$
We remark that if $f=\charc(M)$ for some matrix $M\in M_n(\Ol)$, then $T_m(\charc(M))=\charc(M^m)$ (See in \cite{MP10}, in the discussion after lemma 2.3). 
By \cite[Corollary 2.2]{MP10}, 
$$T_{p,i}(x_1,\ldots,x_n)\equiv x_i^p\pmod {p},$$
that is, $T_{p,i}(x_1,\ldots,x_n)=x_i^p+ph_i$, for some $h_i\in\Z[x_1,\ldots,x_n]$. Since $\pi$ divides $p$, we get that if we let $\charc(M)=x^n+\sum_{i=0}^{n-1}(-1)^ia_ix^{n-i}$, then

$$
T_p(\charc(M))\equiv x^n+\sum_{i=0}^{n-1}(-1)^i a_i^px^{n-i}\pmod{\pi},
$$
so that $\charc(M^p)=T_p(\charc(M))\equiv 
\sigma(\charc(M))\pmod{\pi}$. In particular, since we have the equality $\sigma(\charc(M))=\charc(\sigma(M))$,  we proved \eqref{eq_claim_traces_dependencies} for $k=1$. 

We will now prove \eqref{eq_claim_traces_dependencies} by induction on $k$. The base case for the induction was already proved, so assume we know the claim for $k\ge 1$. By \cite[Theorem 2.5]{MP10}, if $f,g\in\Ol[x]$ are such that $f\equiv g\pmod{\pi^k}$ then $T_p(f)\equiv T_p(g)\pmod{\pi^{k+1}}$. Thus using the induction hypothesis,
$$
\charc(M^{p^{k+1}})=T_p(\charc(M^{p^k}))\equiv T_p(\charc(\sigma(M)^{p^{k-1}}))=\charc(\sigma(M)^{p^k})\pmod{\pi^k},
$$
and we proved the result. 
\end{proof}

\begin{definition}
Let $n,d\ge 1$ be two integers such that $d<n$.

\begin{enumerate}
    \item For a matrix $M\in GL_n(\resring{k})$, we define its trace data of length $d$, $\TR^d(M)$, to be the sequence $(\tr(M^i)-1_{p|i}\sigma(\tr(M^{i/p})))_{i=1\atop{p^k\nmid i}}^d$. 
    \item For a matrix $M\in GL_n(\Ol)$, we define its trace data of length $d$, $\TR^d(M)$, to be the following sequence  $(\tr(M^i)-1_{p|i}\sigma(\tr(M^{i/p})))_{i=1}^d$. 
\end{enumerate}
\end{definition}

By Lemma \ref{lem_trace_dependencies}, the trace data of a matrix in $GL_n(\resring{k})$ is a trace data modulo $\pi^k$. The trace data of a matrix in $GL_n(\Ol)$ is a $p$-adic trace data, again by the same lemma. Recall that we denote by $G_n$ the $n\times n$ group (e.g, if $G=GL$, $G_n=GL_n$, etc.).

\begin{lem}\label{lem_connection_traces_charpoly}
Let $d,n\ge 1$ be integers, and assume that $d\le n$. Let 
$(a_i)_{1\le i\le d,p^k\nmid i}$ be a trace datum. Let $S=S(d,k)=\left[\frac{d}{p}\right]+\left[\frac{d}{p^2}\right]+\cdots+\left[\frac{d}{p^k}\right]$. Finally, let $M\in GL_n(\resring{k})$. There are monic polynomials of degree $n$, $(f_i)_{1\le i\le q^{S}\subset (\resring{k})[x]}$, depending on $d,n,k$, such that the following two conditions are equivalent:

\begin{enumerate}
    \item For all $1\le i\le d$ with $p^k\nmid i$, we have $\tr(M^i)-1_{p|i}\sigma(a_{i/p})=a_i$ (i.e., $\TR^d(M)=(a_i)_{1\le i\le d,p^k\nmid i}$).
    \item For some $1\le j\le q^S$ we have $\charc(M)\in I(f_j,n-d-1)$.
\end{enumerate}Moreover, the sets $\{I(f_j,n-d-1)\}_{1\le j\le q^S}$ are disjoint, so if $G$ is one of the groups we consider then

$$
\Pr_{M\in G_n(\resring{k})}[\forall 1\le i\le d:\tr(M^i)-1_{p|i}\sigma(a_{i/p})=a_i]=\sum_{j=1}^{q^S}
\Pr_{M\in G_n(\resring{k})}[\charc(M)\in I(f_j,n-d-1)].$$
\end{lem}

\begin{proof}
For $M\in G_n(\resring{k})$, write its characteristic polynomial as 

$$
\charc(M)=x^n+e_1(M)x^{n-1}+e_2(M)x^{n-2}+\cdots+e_n(M),
$$

where $(-1)^i e_i(M)$ is the $i$-th symmetric polynomial in the eigenvalues of $M$. Using Newton's identities (See \cite[\S 3.1.1]{Pra04}), we can write
\begin{equation}\label{eqn_newton_formula}
\tr(M^i)=-ie_i(M)+P_i(e_1(M),\ldots,e_{i-1}(M)),    
\end{equation}

for some explicit $P_i\in\Z[x_1,\ldots,x_n]$. Let $p^{j'(i)}||i$, $j(i)=\min(k,j'(i))$. Then equation \eqref{eqn_newton_formula} determines $e_i(M)\pmod{\pi^{k-j(i)}}$, if $\tr(M^i)$ is given and the previous coefficients $e_1(M),\ldots,e_{i-1}(M)$ were already determined. In particular, $e_i(M)$ has $q^{j(i)}$ options. Thus iteratively we get that $\charc(M)$ is in a set of intervals $\{I(f_j,n-d-1)\}_{1\le j\le q^S}$, determined completely by the traces of powers, and of size $q^S$.

This works in the other direction as well: let $f_r$ be one of the polynomials constructed above, $f_r=x^n+a_{1}^{(r)}+x^{n-1}+\cdots+a_n^{(r)}$ (only the $d$ next-to-leading coefficients are relevant, the others may be picked arbitrarily). If $e_i(M)=a_i^{(r)}$, for 
$i=1,\ldots,d$, then equation \eqref{eqn_newton_formula} determines all the traces of powers $\tr(M^i)$, for $i=1,\ldots,d$. Moreover, the resulting traces of powers do not depend on the choice of $r$, since $a_i^{(r)}\equiv a_i^{(r')}\pmod{\pi^{k-j(i)}}$, and only this joint value is relevant for equation \eqref{eqn_newton_formula}.
\end{proof}

Finally, we define random variables $\UTR_d$ capturing the distribution of ``uniformly random" trace data. We also define the random variable $\TR_d^G$, assigning to a matrix $M\in G(\Ol)$ its $d$-trace data.

\begin{definition}
Let $d,n$ be integers with $1\le d<n$, and let $G$ be one of the groups considered in the paper.

\begin{enumerate}
    \item Define the random variable $\UTR_d$ to be the sequence $(u_i)_{i=1}^d$ of $d$ $p$-adic integers, such that $u_i\sim \frac{1}{|i|_p} U(\Ol)$ are independent random variables, and $U(\Ol)$ is the uniform random variable (w.r.t. the Haar probability measure).
    \item Define the associated $\resring{k}$ random variable $\UTR_d^k$ via the equation $\UTR_d \pmod{\pi^k}$.
    \item Define $\TR_d^G$ to be the random variable $\TR^d(M)$ where $M$ is distributed uniformly at random w.r.t. the Haar probability measure on $G(\Ol)$.
    \item Define the associated $\resring{k}$ random variable $\TR_d^{G,k}=\TR_d(M)$ where $M\in G(\resring{k})$ is chosen uniformly at random from $G(\resring{k})$ (equivalently, it is $\TR_d^G\pmod{\pi^k}$).
\end{enumerate}
\end{definition}

When $G\neq GL_n, SL_n$, the traces of negative powers for a matrix $M\in G$ may be directly obtained from the traces of positive powers. For example, if $M\in O_n(\Ol)$, then since $M^{-1}=M^t$ we get that for a positive integer $k$, $\tr(M^{-k})=\tr(M^k)$. We thus consider the traces of negative powers only for $GL_n, SL_n$, and for the rest of the subsection, we assume $G$ is one of them.

\begin{definition}
Let $0\le d_1,d_2<n$ be two integers such that $d_1+d_2=d<n-1$. A $p$-adic $(d_1,d_2)$-trace data is obtained by a taking a sequence $(\frac{ a_i}{|i|_p})_{0\neq i=-d_1}^{d_2}$, where $a_i\in\Ol$ are some $p$-adic integers.  A $(d_1,d_2)$-trace data for $\resring{k}$ is obtained by reducing the $p$-adic trace data modulo $\pi^k$. The length of the trace data is $d=d_1+d_2$.
\end{definition}

\begin{definition}
For a matrix $M\in GL_n(\Ol)$, denote the $d_1$-trace data for $M^{-1}$ by $(a_{-i})_{i=1}^{d_1}$ and the $d_2$-trace data for $M$ by $(a_i)_{i=1}^{d_2}$. Define the $(d_1,d_2)$-trace data of $M$, denoted by $\TR^{d_1,d_2}(M)$, by the sequence $(a_i)_{0\neq i=-d_1}^{d_2}$. We define similarly the $(d_1,d_2)$-trace data for a matrix $M\in GL_n(\resring{k})$. 
\end{definition}

As a corollary of the above, we get 

\begin{lem}\label{lem_connection_traces_hayes_classes}
Fix $d_1,d_2\in [n]$ such that $d_1+d_2=d<n-1$, and let 
$\mathfrak{a}$ be a sequence of $(d_1,d_2)$-trace data. Let $S=S(d,k)=\left[\frac{d}{p}\right]+\left[\frac{d}{p^2}\right]+\cdots+\left[\frac{d}{p^k}\right]$. There are monic polynomials of degree $n$, 

$$\{f_i\}_{1\le i\le q^{S(d_1,k)+S(d_2,k)+k}\subset (\resring{k})[x]},$$
such that the following two conditions are equivalent:

\begin{enumerate}
    \item The $(d_1,d_2)$-trace data of $M$ is equal to $\mathfrak{a}$.
    \item For some $1\le i\le q^{S(d_1,k)+S(d_2,k)+k}$ we have $\charc(M)\equiv f_i\pmod{\hayesrelation{k}}$.
\end{enumerate}

Moreover, the residue classes of $\hayesrelation{k}$ are disjoint, so if $G\subset GL_n(\Ol)$ is a set of matrices with a fixed determinant then

\begin{multline*}
\Pr_{M\in G(\resring{k})}
\left[M \text{ has }(d_1,d_2)\text{-trace data equal to }\mathfrak{a}
\right]=\\
=\sum_{i=1}^{q^{S(d_1,k)+S(d_2,k)+k}}
\Pr_{M\in G(\resring{k})}\left[\charc(M)\equiv f_i\pmod{\hayesrelation{k}}\right].    
\end{multline*}

\end{lem}

\begin{proof}
By Lemma \ref{lem_connection_traces_charpoly}, having a $(d_1,d_2)$-trace data equal to $\mathfrak{a}$ translates to having $\charc(M)$ in $I(f_i,n-d_1)$ for one of $q^{S(d_1,k)}$ polynomials and simultaneously $\charc(M)\in I^{\rs}(g_j,n-d_2)$ for one of $q^{S(d_2,k)}$ polynomials. Note however that being in $I(f_i,n-d_1)\cap I^{\rs}(g_j,n-d_2)$ determines $q^k$ Hayes residue classes in $\hayesrelation{k}$, since the free coefficient is not determined.
\end{proof}

\begin{definition}
Let $0\le d_1,d_2<n$ be two integers such that $d_1+d_2=d<n-1$, and let $G$ be one of the matrix groups considered in the paper.

\begin{enumerate}
    \item Define the random variable $\UTR_{d_1,d_2}$ to be the sequence $(u_i)_{0\ne i=-d_1}^{d_2}$ of $d$ $p$-adic integers, such that $u_i\sim \frac{1}{|i|_p} U(\Ol)$ are independent random variables, and $U(\Ol)$ is the uniform (w.r.t. the Haar probability measure) random variable.
    \item Define the associated $\resring{k}$ random variable $\UTR_{d_1,d_2}^k$ by $\UTR_{d_1,d_2} \pmod{\pi^k}$.
    \item Define $\TR_{d_1,d_2}^G$ to be $\TR^{d_1,d_2}(M)$ for a matrix $M$ which is distributed uniformly at random w.r.t. the Haar probability measure on $G(\Ol)$.
    \item The associated $\resring{k}$ random variable is denoted by $\TR_{d_1,d_2}^{G,k}$, and is defined by $\TR_{d_1,d_2}^G\pmod{\pi^k}$.
\end{enumerate}
\end{definition}

As a final remark, we note that the number of different possible values for a $(d_1,d_2)$-trace data is $q^{d-S(d_1,k)-S(d_2,k)}$.

\subsection{Hayes characters of palindromic polynomials}

\begin{definition}
Let $n\ge 1$ be an integer.
\begin{enumerate}
    \item A polynomial $f\in\field_q[x]_{\le n}$ is called $n$-palindromic (or just palindromic, if the value of $n$ is clear from the context) if $x^nf(1/x)=f(x)$.
    \item A monic polynomial $f\in\field_q[x]$ with $f(0)\neq 0$ is called self-reciprocal if $x^{\deg f}\frac{f(1/x)}{f(0)}=f(x)$.
\end{enumerate}
\end{definition}

We remark that the two notions are related: e.g, if $f\in\field_q[x]_{=n}$ is self-reciprocal such that $f(0)=1$, then $f$ is $\deg(f)$-palindromic. Conversely, if $\deg(f)=n$, $f(0)=1$, and $f$ is $n$-palindromic, then $f$ is self-reciprocal. Also, if $f$ is a polynomial of degree $<n$ and $f$ is $n$-palindromic, then $x^n+f+1$ is self-reciprocal (note that the free coefficient of $f$ must be $0$). Every monic, $n$-palindromic polynomial can be obtained this way.

We recall that in this paper, we denote the set of monic polynomials in $\field_q[x]$ by $\mathcal{M}_q$. We denote the set of monic, self-reciprocal polynomials in $\field_q[x]$ by $\mathcal{M}_q^{\mathrm{sr}}$. Finally, we denote the set of $n$-palindromic polynomials of degree $\le n$ (or $<n$, or $=n$) by $\field_q[x]^{\pal}_{\le n}$ (or $\field_q[x]^{\pal}_{< n}$, or $\field_q[x]^{\pal}_{=n}$). Self-reciprocal polynomials and their relations with Hayes characters were studied in \cite{GR21}, and the key construction there is that of the map 

$$
\Psi:\mathcal{M}_q\to\mathcal{M}_q^{sr}, g\mapsto g\left(x+\frac{1}{x}\right)x^{\deg(g)},
$$
and its inverse $\Phi$. We adapt the results in \cite[Lemma 6.12, Proposition 6.13]{GR21} to a slightly different setting which we will use later.

\begin{prop}\label{prop_psi_even}
Define $\Psi_{2n}:\field_q[x]_{=n}^{\mon}\to \field_q[x]_{=2n}^{\mon}$ by $\Psi_{2n}(g)=g\left(
x+\frac{1}{x}
\right)x^n$. Then,

\begin{enumerate}
    \item $\Psi_{2n}$ is injective.
    \item The image of $\Psi_{2n}$ is precisely $\field_q[x]^{\mathrm{pal}}_{=2n}\cap \field_q[x]_{=2n}^\mon$.
    \item Denote the inverse of $\Psi_{2n}$, defined on $\field_q[x]^{\mathrm{pal}}_{=2n}\cap \field_q[x]_{=2n}^\mon$, by $\Phi_{2n}$. Let $\chi\in \hat{R}_{l,1}$. There is a unique $\chi'\in \hat{R}_{l,1}$ such that $\chi(f)=\chi'(\Phi_{2n}(f))$ for every $f\in \field_q[x]^{\mathrm{pal}}_{=2n}\cap \field_q[x]_{=2n}^\mon$. Moreover, $\chi$ is non-trivial iff $\chi'$ is non-trivial. 
\end{enumerate}
\end{prop}

\begin{proof}
The fact that $\Psi_{2n}$ is injective follows from the fact that $\Psi$ is injective, proved in \cite[Lemma 6.12]{GR21}. We note further that by direct computation, for every $g$, $\Psi_{2n}(g)$ is palindromic. The set of monic $2n$-palindromic polynomials of degree $2n$ is a vector space of dimension $n$, hence $\Psi_{2n}$ cover all monic $2n$ palindromic polynomials. 

For the last assertion, define 
$
\chi'(g)=\chi(g\left(x+\frac{1}{x}\right)x^{\deg(g)})
$. By \cite[Proposition 6.13]{GR21}, this function is a short interval character of $l$ coefficients and is non-trivial iff $\chi$ is non-trivial. Hence, denoting $g=\Phi_{2n}(f)$, $\chi(f)=\chi(g\left(x+\frac{1}{x}\right)x^n)=\chi'(g)$ as claimed. 
\end{proof}

We adapt this also to palindromic polynomials of odd degrees. We note that if $f \in \field_q[x]_{=2n-1}^{\pal}$, then $f(-1)=0$. Hence $f=(x+1)g$, where $g$ is $(2n-2)$-palindromic. We thus define $\Psi_{2n-1}:\field_q[x]_{=n-1}^{\mon}\to\field_q[x]_{=2n-1}^{\pal}$ by $\Psi_{2n-1}(g)=(x+1)g\left(
x+\frac{1}{x}
\right)x^{n-1}$. 

\begin{prop}\label{prop_psi_odd}
Define $\Psi_{2n-1}:\field_q[x]_{=n-1}^{\mon}\to\field_q[x]_{=2n-1}^{\pal}$ by $\Psi_{2n-1}(g)=(x+1)g\left(
x+\frac{1}{x}
\right)x^{n-1}$ by $\Psi_{2n-1}(g)=(x+1)g\left(
x+\frac{1}{x}
\right)x^{n-1}$. Then,

\begin{enumerate}
    \item $\Psi_{2n-1}$ is injective.
    \item The image of $\Psi_{2n-1}$ is precisely $\field_q[x]^{\mathrm{pal}}_{=2n-1}\cap\field_q[x]_{=2n-1}^{\mon}$.
    \item Denote the inverse of $\Psi_{2n-1}$, defined on $\field_q[x]^{\mathrm{pal}}_{=2n-1}\cap\field_q[x]_{=2n-1}^{\mon}$, by $\Phi_{2n-1}$. Let $\chi\in\hat{R}_{l,1}$. There is a unique $\chi'\in\hat{R}_{l,1}$ and a constant $\alpha(\chi)\in\field_q$ such that $\chi(f)=\alpha(\chi)\chi'(\Phi_{2n-1}(f))$ for every $f\in\field_q[x]^{\mathrm{pal}}_{=2n-1}\cap\field_q[x]_{=2n-1}^{\mon}$. Moreover, $\chi$ is non-trivial iff $\chi'$ is non-trivial, and in the case that $\chi$ is trivial, $\alpha(\chi)=1$.
\end{enumerate}
\end{prop}

\begin{proof}
The fact that $\Psi_{2n-1}$ is injective follows from the claim on $\Psi_{2n}$. Regarding the image of $\Psi_{2n-1}$, we note that $\Psi_{2n-1}(g)$ is always monic and palindromic, and that the dimension of the linear space $\field_q[x]_{=n-1}$ is $n-1$ which is equal to the dimension of $\field_q[x]^{\mathrm{pal}}_{=2n-1}\cap\field_q[x]_{=2n-1}^{\mon}$. 

Finally, let $\chi$ be a short interval character of length $l$. Now $f=(x+1)f_0$ where $f_0$ is $(2n-2)$-palindromic. Here, $\Phi_{2n-1}(f)=\Psi_{2n-2}(f_0)$. Denote their joint value by $g$. Then by the last proposition, there is a unique $\chi'$ such that $\chi(f_0)=\chi'(g)$. Letting $\alpha(\chi)=\chi(x+1)$ we get the result. 
\end{proof}

We now prove an orthogonality relation for Hayes characters evaluated at palindromic polynomials. 

\begin{lem}\label{lem_hayes_character_sum_over_palindromic_polys}
Let $n$ be an integer and let  $\delta=\lceil \frac{n-1}{2}\rceil$. Let $\chi$ be a short interval character of $l$ coefficients, with $l\le\delta$. Then, 

$$
q^{-\delta}\sum_{f\in\field_q[x]_{<n}^{\pal}}\chi(x^n+f)=1_{\chi=\chi_0}.
$$
\end{lem}

\begin{proof}
Since $\delta\ge l$, we have in particular that $n>l$ and we may change the free coefficient of the polynomial without changing the value of the character. Hence, 

$$
q^{-\delta}\sum_{f\in\field_q[x]_{<n}^{\pal}}\chi(x^n+f)=
q^{-\delta}\sum_{f\in\field_q[x]_{<n}^{\pal}}\chi(x^n+f+1).
$$

We use the map $\Psi_n:\field_q[x]_{=\delta}^\mon\to\field_q[x]_{=n}^{\pal}\cap\field_q[x]_{=n}^\mon$. By Propositions \ref{prop_psi_even}, \ref{prop_psi_odd} we have a unique character $\chi'$ and a constant $\alpha(\chi)$ (which is $1$ if $\chi=\chi_0$) such that 
$
\chi(x^n+f+1)=\alpha(\chi)\chi'(\Phi_n(x^n+f+1))
$. In particular, by the orthogonal relations for Hayes characters \eqref{eqn_hayes_first_orthogonality_of_characters},

$$
q^{-\delta}\sum_{f\in\field_q[x]_{<n}^{\pal}}\chi(x^n+f+1)=q^{-\delta}\alpha(\chi)\sum_{g\in\field_q[x]_{=\delta}^{\mon}}\chi'(g)=1_{\chi=\chi_0}.
$$
\end{proof}

We will also consider the closely related notions of $*$-symmetric polynomials and skew-palindromic polynomials. Let $\field_{q^2}/\field_q$ be a quadratic extension, and let $\tau$ denote the action of $q$-Frobenius on $\field_{q^2}$, that is $\tau(\alpha)=\alpha^q$ for $\alpha\in\field_{q^2}$. 
Let $f\in \field_{q^2}[x]_{<n}$ be a polynomial, and let $\alpha\in\field_{q^2}$ be such that $\alpha\tau(\alpha)=1$. Also for a polynomial $f\in\field_{q^2}(x)$, define $f^*(x)=x^n \tau(f)(1/x)$. We say that $f$ is $\alpha$-skew palindromic polynomial (or simply skew-palindromic polynomial, if $\alpha$ is clear from the context) if $\frac{f^*(x)}{\tau(\alpha)}=f$. If we want to emphasize the dependence on $n$, we say that $f$ is skew-palindromic w.r.t. $n$.  

We note that due to the assumption on $\alpha$, $f\mapsto \frac{f^*(x)}{\tau(\alpha)}$ is an involution on $\field_{q^2}[x]_{<n}$. We denote the set of $\alpha$-skew palindromic polynomials of degree $<n$ (or $\le n$, $=n$) by $\field_{q^2}[x]_{<n}^{\alpha\text{-skpal}}$ (or $\field_{q^2}[x]_{\le n}^{\alpha-\skpal}$, $\field_{q^2}[x]_{=n}^{\alpha-\skpal}$ respectively), or simply $\field_{q^2}[x]_{<n}^{\text{skpal}}$ (or $\field_{q^2}[x]_{\le n}^{\skpal}$, $\field_{q^2}[x]_{=n}^{\skpal}$ respectively) if $\alpha$ is clear from the context. When $\alpha=1$, we call the polynomial $*$-symmetric, and we denote the set of $*$-symmetric polynomials of degree $<n$ (or $\le n$, $=n$) by $\field_{q^2}[x]_{<n}^{*-\mathrm{sym}}$ (or $\field_{q^2}[x]_{\le n}^{*-\mathrm{sym}}$, $\field_{q^2}[x]_{=n}^{*-\mathrm{sym}}$ respectively). For the rest of the subsection, skew-palindromic or $*$-symmetric polynomials are assumed to be skew-palindromic w.r.t. $n$. 

\begin{claim}\label{claim_structure_of_skpal}
Let $f\in \field_{q^2}[x]_{<n}^{\alpha-\skpal}$. Then there exist a $\beta=\beta(\alpha)\in\field_{q^2}$ and $g\in \field_{q^2}[x]_{<n}^{*-\sym}$ such that $\beta f=g$. 
\end{claim}

\begin{proof}
By Hilbert's theorem 90, one can write $\alpha=\frac{\tau(\beta)}{\beta}$. Hence, 
$$
f^*(x)=\tau(\alpha)f\iff (\beta f)^*=\tau(\beta) f^*(x)=\beta f.
$$
\end{proof}

We see that to study the structure of $\alpha$ skew-palindromic polynomials it is enough to study the structure of $*$-symmetric polynomials. 

\begin{claim}\label{claim_structure_of_start_sym}
Let $f\in\field_{q^2}[x]_{<n}^{*-\sym}$. Then,
\begin{enumerate}
    \item If $n$ is odd, there exist a unique $h\in\field_{q^2}[x]_{\le\frac{n-1}{2}}$ with $h(0)=0$ such that $f(x)=h(x)+h^*(x)$. 
    \item If $n$ is even, there exist a unique $h\in\field_{q^2}[x]_{\le\frac{n-1}{2}}$ with $h(0)=0$ and $\alpha\in\field_q$ such that $f(x)=h(x)+\alpha x^d+h^*(x)$.
\end{enumerate}
Moreover, any choice of $h$ gives a $*$-symmetric polynomial. In each case, the $l$ next-to-leading coefficients of $x^n+f$, for $l\le \frac{n-1}{2}$, are simply the $l$ next-to-leading coefficients of $x^n+h^*$.  
\end{claim}

\begin{proof}
Note that for $f$ to be in $\field_{q^2}[x]_{<n}^{*\text{-sym}}$, $f(0)$ must be $0$. We split into cases depending on the parity of $n$.

\case{$n=2d+1$ is odd} 
Write $f(x)=a_1x+a_2x^2+\cdots+a_{2d}x^{2d}$, the condition that $f$ is $*$-symmetric is equivalent to requiring $a_{2d+1-i}=\tau(a_i)$. Setting $h(x)=a_1x+\cdots+a_dx^d$ we get the result. 

\case{$n=2d$ is even}
Write $f(x)=a_1x+a_2x^2+\cdots+a_{2d-1}x^{2d-1}$, the condition that $f$ is $*$-symmetric is equivalent to requiring $a_{2d-i}=\tau(a_i)$. In particular, $\tau(a_d)=a_d$ and thus $a_d\in\field_q$. Thus, setting $h=a_1x+\cdots+a_{d-1}x^{d-1}$ we get the result.
\end{proof}

\begin{lem}\label{lem_hayes_character_sum_over_skpalindromic_polys}
Let $n$ be an integer and set $\delta=\lceil\frac{n-1}{2}\rceil$.
Let $\chi$ be a short interval character of $l$ coefficients, with $l\le\delta$. Then,

$$
q^{-\delta}\sum_{f\in\field_{q^2}[x]_{<n}^{\alpha-\skpal}}\chi(x^n+f)=1_{\chi=\chi_0}.
$$
\end{lem}

\begin{proof}
We first note by Claim \ref{claim_structure_of_skpal} there is a $\beta\in\field_{q^2}$ such that $\field_{q^2}[x]_{<n}^{\alpha\text{-skpal}}=\beta \field_{q^2}[x]_{<n}^{*\text{-sym}}$. In particular, 
$$
q^{-\delta}\sum_{f\in\field_{q^2}[x]_{<n}^{\alpha-\skpal}}\chi(x^n+f)=
q^{-\delta}\sum_{f\in\field_{q^2}[x]_{<n}^{*-\sym}}\chi(x^n+\beta f).
$$

We now split the proof for the case of $n$ even and $n$ odd, and use Claim \ref{claim_structure_of_start_sym}.

\case{$n$ odd}
In this case
$$
q^{-\delta}\sum_{f\in\field_{q^2}[x]_{<n}^{*-\sym}}\chi(x^n+\beta f)=q^{-\delta}\sum_{h\in\field_{q^2}[x]_{<\delta}}\chi(x^n+\beta h+\beta h^*).
$$
Here $\chi(x^n+\beta f)=\chi(\beta h^*)$, and since $h$ is arbitrary, the last sum is equal to $1_{\chi=\chi_0}$ by the Hayes orthogonality relation \eqref{eqn_hayes_first_orthogonality_of_characters}.

\case{$n$ even}
In this case
$$
q^{-\delta}\sum_{f\in\field_{q^2}[x]_{<n}^{*-\sym}}\chi(x^n+\beta f)=q^{-\delta}\sum_{\alpha\in\field_q}\sum_{h\in\field_{q^2}[x]_{<\delta-1}}\chi(x^n+\beta h+\beta \alpha x^\delta+\beta h^*)=1_{\chi=\chi_0},
$$
again using the Hayes orthogonality relation.
\end{proof}

\subsection{Traces of negative powers for the finite groups}\label{section_neg_traces_finite_groups}

We show that slight modifications in the proofs of \cite{GR21} lead to results on the traces of negative powers for $GL_n, SL_n$. For the other groups in consideration, traces of negative powers depend simply on the traces of positive powers:

\begin{enumerate}
    \item For $M\in Sp_{2n}(\Ol)$ we have $M\Omega_n M^t=\Omega_n$. In particular $M^{-1}=\Omega_n M^t\Omega_n^{-1}$ so that $\tr(M^{-1})=\tr(M^t)=\tr(M)$. This generalizes to all negative powers of $M$.
    \item For $SO_n(\Ol)$ we have $MM^t=I$ so that $M^{-1}=M^t$, and in particular $\tr(M^{-1})=\tr(M^t)=\tr(M)$. Similarly to the case of $Sp_{2n}$, this generalizes to all negative powers of $M$. 
    \item For $U_n(\Ol)$ we have $MM^*=I$, so that $M^{-1}=M^*$. In particular, $\tr(M^{-1})=\tr(M^*)=\tr(\sigma(M^t))=\sigma(\tr(M))$. This generalizes to all negative powers of $M$.
\end{enumerate}

Fix a non-trivial additive character $\psi$ of $\field_q$. For a vector $\overrightarrow{\lambda}=(\lambda_i)_{i=1\atop{p\nmid i}}^k$ defined over $\field_q$ and $f\in\mathcal{M}_q$, Gorodetsky-Rodgers define in \cite[\S 2.2]{GR21} the function $\chi_{\overrightarrow{\lambda}}(f)=\psi(\sum_{i=1\atop{p\nmid i}}^k \lambda_i(\alpha_1^i+\cdots+\alpha_n^i))$, where $\alpha_i$ are the roots of $f$ in $\algfp$, listed with multiplicities. We extend this notation to include cases where the powers $i$ can be negative (but are always non-zero).

\begin{claim}
Let $k$ be a positive integer coprime with $p$, and let $\overrightarrow{\lambda}=(\lambda_{-i})_{i=1\atop{p\nmid i}}^{k}\subset \field_q$ with $\lambda_k\neq 0$. Let 
$\overrightarrow{\eta}=(\eta_i)_{i=1\atop{p\nmid i}}^k=(\lambda_{-i})_{i=1\atop{p\nmid i}}^k$. Then, $\chi_{\overrightarrow{\lambda}}(f)=\chi_{\overrightarrow{\eta}}(f^\rs)$. In particular, the function $\chi_{\overrightarrow{\lambda}}(f)$ is a Dirichlet character with modulus $x^{k+1}$. It is non-trivial when $k>0$.
\end{claim}

\begin{proof}
By \cite[Lemma 2.2]{GR21}, $\chi_{\overrightarrow{\eta}}(f^\rs)$ is a short interval character of $k$ coefficients evaluated at the polynomial $f^{\rs}$. That means that it is a multiplicative function depending only on the $k$ next-to-leading coefficients of $f^\rs$, which in turn, depend only on the coefficients of $1,x,x^2,\ldots,x^k$ in $f$. Since clearly $\chi_{\overrightarrow{\lambda}}(f)$ is multiplicative we get that it is a Dirichlet character modulo $x^{k+1}$.

To see that $\chi_{\overrightarrow{\lambda}}$ is non-trivial, notice that $f\mapsto f^\rs$ is an involution on the monic polynomials and use the fact that $\chi_{\overrightarrow{\lambda}}(f)$ is non-trivial (proved in \cite[Lemma 2.2]{GR21}).
\end{proof}

\begin{lem}\label{lem_chi_lambda_is_hayes}
Let $l,k$ be two non-negative integers, and let $\overrightarrow{\lambda}=(\lambda_i)_{p\nmid i=-l}^k\subset\field_q$. The function $\chi_{\overrightarrow{\lambda}}(f)$ is a Hayes character modulo $R_{k,x^{l+1}}$.
\end{lem}

\begin{proof}
Indeed, $\chi_{\overrightarrow{\lambda}}(f)$ is the product of a short interval character of $k$ coefficients with a Dirichlet character of modulus $x^{l+1}$.
\end{proof}

We can get as a result an extension of the results of Gorodetsky-Rodgers for negative traces as well, and also results on equidistribution in Hayes residue classes. 

\begin{corollary}\label{corollary_neg_traces_gln}
Let $M\in GL_n(\field_q)$ be a matrix chosen uniformly at random w.r.t. the Haar measure. Let $(b_i)_{i=1}^k$ be an increasing sequence of integers coprime with $p$. Let $d=\max(0,b_k)-\min(0,b_1)+1$. 

Then, for any sequence $(a_i)_{i=1}^k$ of elements in $\field_q$, we have

$$
\left|
\Pr_{M\in GL_n(\field_q)}\left[
(\tr (M^{b_i}))_{i=1}^k =(a_i)_{i=1}^k
\right]-q^{-k}
\right|
\le q^{-\frac{n^2}{2d}}\left(1+\frac{1}{q-1}\right)^n\binom{n+d}{n}.
$$

Now let $f\in\field_q[x]$ be a monic polynomial. Let $0\le d_1,d_2<n$ be integers such that $d_1+d_2=d<n-1$. Then, 

$$
\left|
\Pr_{M\in GL_n(\field_q)}\left[
\charc(M)\equiv f\pmod{R_{d_1,x^{d_2}}}
\right]-q^{-d}
\right|
< q^{-\frac{n^2}{2d}}\left(1+\frac{1}{q-1}\right)^n\binom{n+d}{n}.
$$

\end{corollary}
\begin{proof}
    The proof is essentially the same as the proof of \cite[Theorem 1.1]{GR21} and \cite[Theorem 1.3]{GR21}, using our Lemma \ref{lem_chi_lambda_is_hayes}.
\end{proof}

\begin{corollary}\label{cor_neg_traces_mod_q_sln}
Let $M\in SL_n(\field_q)$ be a matrix chosen uniformly at random w.r.t. the Haar measure. Let $(b_i)_{i=1}^k$ be an increasing sequence of integers coprime with $p$. Let $d=\max(0,b_k)-\min(0,b_1)+1$. 

Then, for any sequence $(a_i)_{i=1}^k$ of elements in $\field_q$, we have

$$
\left|
\Pr_{M\in SL_n(\field_q)}\left[
(\tr (M^{b_i}))_{i=1}^k =(a_i)_{i=1}^k\right]
-q^{-k}
\right|
\le q^{-\frac{n^2}{2d}}(q-1)\left(1+\frac{1}{q-1}\right)^n\binom{n+d}{n}.
$$

Now let $f\in\field_q[x]$ be a monic polynomial with $f(0)=(-1)^n$. Let $0\le d_1,d_2<n$ be integers such that $d_1+d_2=d<n-1$. Then, 

$$
\left|
\Pr_{M\in GL_n(\field_q)}\left[
\charc(M)\equiv f\pmod{R_{d_1,x^{d_2}}}
\right]-q^{-d}
\right|
< q^{-\frac{n^2}{2d}}(q-1)\left(1+\frac{1}{q-1}\right)^n\binom{n+d}{n}.
$$

\end{corollary}

\begin{proof}
The proof is essentially the same as the proof of \cite[Theorem 4.6]{GR21} and \cite[Theorem 4.5]{GR21}, using our Lemma \ref{lem_chi_lambda_is_hayes}.
\end{proof}

\section{Finite groups of Lie type and their Lie algebras}\label{section_explicit_representatives}

\subsection{Finite matrix groups}

For every group considered in this paper, we recall some basic facts which will be used later, regarding the Jordan type decomposition of the group and some related constructions.

\subsubsection{$GL_n$ and $SL_n$}

Recall that $GL_n(\field_q)$ is the set of invertible $n\times n$ matrices over $\field_q$, $SL_n(\field_q)\subset GL_n(\field_q)$ the set of matrices with determinant $=1$. It is well known that the associated Lie algebras $\mathfrak{gl}_n(\field_q)$, $\mathfrak{sl}_n(\field_q)$ are $M_n(\field_q)$, $\{M\in M_n(\field_q)|\tr(M)=0\}$ (respectively). We note that $\mathfrak{sl}_n(\field_q)$ is invariant under conjugation from $GL_n(\field_q)$, a property that will be used later.

Let $M$ be a matrix in either group. The vector space $V=\field_q^n$ becomes a $\field_q[x]$-module with the action $f(x)v=f(M)v$, for every $f\in\field_q[x]$, $v\in V$. A polynomial $f\in\field_q[x]$ is called prime if it is irreducible and monic. Let $\mathcal{P}$ be the set of primes in $\field_q[x]$. By the structure theorem for finitely generated modules over a PID, one can decompose

$$
V\cong \bigoplus_{f\in \mathcal{P}}\bigoplus_{i=1}^{k_f} \field_q[x]/(f)^{m_i}\cong \bigoplus_{f\in\mathcal{P}} V_f,V_f\cong\bigoplus_{i=1}^{k_f}\field_q[x]/(f)^{m_i}.
$$
Here $V_f$ is called the $f$-primary part. The primary components of $M$ in $GL_n(\field_q)$ (or in $SL_n(\field_q)$) determine its conjugacy class in $GL_n(\field_q)$ (or in $SL_n(\field_q)$). 

\begin{definition}
For a matrix $M\in GL_n(\field_q)$ we define the conjugacy class data of $M$ to be the multiset $(f^{m_i})_{f\in\mathcal{P}\atop{i=1,\cdots,k_f}}$.
\end{definition}

Finally, we want to understand how to combine matrices with different conjugacy class data. The idea is to take the direct sum of the spaces, and in the case of $GL_n, SL_n$ this is straightforward.

\begin{definition}
    Let $A,B\in GL_n(\field_q),GL_m(\field_q)$. Their direct sum is the matrix $A\oplus B\in GL_{n+m}(\field_q)$, which in block form looks as 

    $$
    A\oplus B=\left(
\begin{array}{cc}
     A & 0 \\
     0 & B
\end{array}
    \right).
    $$
\end{definition}

Let $(f^{m_i})_{f\in\mathcal{P}\atop{i=1,\ldots,k_f}}$, $(f^{e_j})_{f\in\mathcal{P}\atop{j=1,\ldots,l_f}}$ be two conjugacy class data. We define their join to be the union of conjugacy class data (counted multiplicities), $(f^{m_i})_{f\in\mathcal{P}\atop{i=1,\ldots,k_f}}\cup (f^{e_j})_{f\in\mathcal{P}\atop{j=1,\ldots,l_f}}$. A direct computation shows that the conjugacy class data of $A\oplus B$ is the join of the conjugacy class data for $A$, $B$. 

\subsubsection{The symplectic group}\label{subsubsection_symplectic_group_intro}
Throughout the paper $\field_q$ is a finite field of characteristic $p\neq 2$. Define the standard alternating form via the matrix $\Omega_n=\left(
\begin{array}{cc}
0 & I_n \\
-I_n & 0
\end{array}
\right)$. We define the symplectic group 
$$
Sp_{2n}(\field_q)=\left\{
M\in GL_{2n}(\field_q)|M^t\Omega_n M=\Omega_n
\right\}.
$$
The associated Lie algebra is the algebra
$$\mathfrak{sp}_{2n}(\field_q)=\left \{ X\in M_{2n\times 2n}(\field_q): \Omega_n X+X^{t}\Omega_n=0 \right\}.$$ 
More explicitly, writing an element $X$ in $\mathfrak{sp}_{2n}(\field_q)$ as a $2\times 2$ block matrix by
$$
X=\left(
\begin{array}{cc}
     A & B \\
     C & D
\end{array}
\right),
$$
we get that $X\in\mathfrak{sp}_{2n}(\field_q)$ is equivalent to $B,C$ being symmetric and $A^t=-D$. We note that the characteristic polynomial of an element in the symplectic group admits symmetry.
\begin{definition}
Let $f\in\field_q[x]$ be a polynomial such that $f(0)\neq 0$. We define its reciprocal by the formula $f^\rs(x)=\frac{x^{\deg f}f(1/x)}{f(0)}$.
\end{definition}

For $g\in Sp_{2n}(\field_q)$, $\charc(g)$ is self-reciprocal (i.e, $\charc(g)^\rs=\charc(g)$). We call a polynomial $f\in \field_q[x]$ reciprocally irreducible if it cannot be decomposed (non-trivially) to self-reciprocal polynomials, and we denote the set of reciprocally irreducible polynomials by $\mathcal{P}^\rs$. As an example, over $\field_5$, $x^2-4$ is in $\mathcal{P}^\rs$ even though it is decomposable in $\field_5[x]$. The minimal polynomial of $g$, $\min(g)$, is also self-reciprocal. 

Similarly to the $GL_n$ case, the Jordan decomposition is obtained from examining the action of the group $Sp_{2n}(\field_q)$ on $V=\field_q^{2n}$. For $g\in Sp_{2n}(\field_q)$, the action of $g$ on this space induces a structure of $\field_q[x]$-module. As explained in \cite{Tay21a}, one can decompose this finitely generated module $V$ as a direct sum 
$\bigoplus_{f\in \Pee^\rs,i}\field_q[x]/(f)^{m_i}$. Moreover, the action of $g$ restricted to each space $\field_q[x]/(f)^{m_i}$ is given by a symplectic matrix (\cite[\S 4]{Tay21a}).

For a given $f\in\Pee^\rs$, we define its $f$-primary component
$$
V_{f}\cong\bigoplus_{i\ge 1}(\field_q[x]/(f^{m_i})),
$$
where here the direct sum is finite. Let $g\in Sp_{2n}(\field_q)$ be a matrix such that one of the primary components 
$$V_{x\pm 1}\cong\bigoplus(\field_q[x]/(x\pm 1)^{i})^{m_i}$$
is non-trivial. Wall \cite[p. 36]{Wal63} proved that for odd $i$, $2|m_i$. Moreover, for even $i$, he attached a certain bilinear form (one of two forms, $\Psi_{-}$ and $\Psi_{+}$) to the space $(\field_q[x]/(x\pm 1)^i)^{m_i}$. The sign of the form is determined by the induced action of $g$ on $(\field_q[x]/(x\pm 1)^i)^{m_i}$. He then proved the following classification of conjugacy classes in $Sp_{2n}(\field_q)$.
\begin{thm}[\cite{Wal63}, p. 36, rephrased]
Two matrices $g,h\in Sp_{2n}(\field_q)$ are conjugates if and only if 
\begin{enumerate}
    \item Their conjugacy class data is the same and in addition,
    \item For every even $i$, the sign of the form corresponding with $(\field_q[x]/(x\pm 1)^i)^{m_i}$ is the same between $g,h$. 
\end{enumerate}
\end{thm}

This motivates the following
\begin{definition}\label{def_symplectic_conj_class_data}
For a matrix $g\in Sp_{2n}(\field_q)$, we define the symplectic conjugacy class data by the multiset 
$$((f^{i},m_i))_{f\in\mathcal{P}^\rs\atop{f\neq x\pm 1\atop{i=1,\ldots}}}\cup ((f^i,{m_i}))_{2\nmid i\atop{f=x\pm 1\atop{i=1,\ldots}}}\cup ((f^i,{m_i},\mu(i)))_{2|i\atop{f=x\pm 1\atop{i=1,\ldots}}},$$
where $\mu(i)=\pm$ is the sign of the invariant attached to the space $(\field_q[x]/(x\pm 1)^i)^{m_i}$. 
\end{definition}
If it is clear from the context that we consider the conjugacy class of $g$ in $Sp_{2n}(\field_q)$, we will sometimes call the symplectic conjugacy class data of $g$ just the conjugacy class data of $g$. The previous theorem of Wall states that the symplectic conjugacy class data parametrizes the conjugacy classes of $Sp_{2n}(\field_q)$. We also define a symplectic signed partition, which provides a purely combinatorial parametrization of $Sp_{2n}(\field_q)$-conjugacy classes.

\begin{definition}
A symplectic signed partition $\lambda$ is a collection of partitions $(\lambda_f)_{f\in\Pee}$ (some of them are tagged partitions), at most finitely many of them are non-empty, such that 
\begin{enumerate}
    \item For $\alpha=\pm 1$,
    \begin{enumerate}
        \item Each odd part of $\lambda_{x-\alpha}$ has even multiplicity.
        \item Let $i>0$ be even and assume that there is at least one part in $\lambda_{x-\alpha}$ with size $i$. To the collection of all parts of size $i$ we attach a sign $\mu(i)=\pm$ (corresponding to $\mu(i)$ in definition \ref{def_symplectic_conj_class_data}).
    \end{enumerate}
\item For $f\in\Pee$ with $f\neq x\pm 1$, $\lambda_f=\lambda_{f^\rs}$.
\end{enumerate}
\end{definition}
Let $\lambda$ be a symplectic signed partition. We will denote $|\lambda|=\sum_{f\in\Pee}|\lambda_f|$. The set of symplectic signed partitions with $|\lambda|=2n$ gives a combinatorial parametrization of $Sp_{2n}(\field_q)$-conjugacy classes, since a polynomial $f\in\Pee^\rs$ is either in $\Pee$ or factors as $f=hh^\rs$ where $h,h^\rs\in\Pee$.

To complement this view, we would like to understand how to join matrices with different primary components (meanwhile preserving the induced action). This will be done using the \textit{triangular join}. Given two matrices $A\in Sp_{2n}(\field_q)$, $B\in Sp_{2m}(\field_q)$, where $A$ is a lower triangular block matrix
$$
\left(
\begin{array}{cc}
    A_1 & 0 \\
    A_3 & A_2
\end{array}
\right),
$$
one may define their \textit{triangular join}, a matrix in $Sp_{2n+2m}(\field_q)$ whose conjugacy class datum corresponds with the join of the conjugacy class data for $A$ and $B$ together, and such that the induced action is preserved (see \cite[\S 5]{Tay21a}).   

\begin{definition}\label{def_triangular_join}
Let 
$$
A = \left(
\begin{array}{cc}
    A_1 & 0 \\
    A_3 & A_2
\end{array}
\right)\in Sp_{2n}(\field_q), B\in Sp_{2m}(\field_q).
$$
The triangular join of $A,B$, also denoted by $A\triangle B$, is a matrix in $Sp_{2n+2m}(\field_q)$ given by
$$
\left(
\begin{array}{ccc}
    A_1 & 0 & 0 \\
    0 & B & 0 \\
    A_3 & 0 & A_2
\end{array}
\right).
$$
\end{definition}

As a special case, we mention the case where $A,B$ are both block diagonal,
$$
A = \left(
\begin{array}{cc}
    A_1 & 0 \\
    0 & A_2
\end{array}
\right), 
B = \left(
\begin{array}{cc}
    B_1 & 0 \\
    0 & B_2
\end{array}
\right), 
$$
in which case the triangular join is called the \textit{diagonal join} and is given by 
$$
A\triangle B = \left(
\begin{array}{cccc}
    A_1 & 0 & 0 & 0 \\
    0 & B_1 & 0 & 0 \\
    0 & 0 & B_2 & 0 \\
    0 & 0 & 0 & A_2
\end{array}
\right).
$$
Later, we will be able to construct explicit representatives to conjugacy classes of $Sp_{2n}(\field_q)$ by patching together matrices with different primary components, using the triangular join.

\subsubsection{The special orthogonal group}\label{section_so_generalities}
We remind the reader again that for the current paper, $\field_q$ is a finite field of odd characteristic $p$. Let $K\in GL_n(\field_q)$ be a symmetric matrix. The orthogonal group preserving the quadratic form with matrix $K$ is 
$$
O_n(\field_q,K)=\left\{g\in GL_n(\field_q):g^t Kg=K\right\}.
$$
When the forms defined by the matrices $K_1, K_2$ are congruent, the corresponding orthogonal groups are conjugate in $GL_n(\field_q)$. Let $\chi_2$ be the quadratic character of $\field_q^\times$. For each $n$, there are two congruence classes for invertible symmetric matrices of size $n$: one for symmetric matrices with $\chi_2(\det(K)(-1)^{[n/2]})=1$ and one for those with $\chi_2(\det(K)(-1)^{[n/2]})=-1$. 

For concreteness, following \cite{Tay16}, we may pick the following representatives for the congruence classes. From now on, these representatives will be called the standard symmetric forms (or just standard forms when the context is clear). 

\maincase{$n=2m+1$ is odd}
Let 
$$\Lambda_m=\left(
\begin{array}{ccccc}
     0 & 0 & \cdots & 0 & 1 \\
     0 & 0 & \cdots & 1 & 0 \\
       &   & \ddots &   &  \\
     0 & 1 & \cdots & 0 & 0 \\
     1 & 0 & \cdots & 0 & 0 \\
\end{array}
\right),$$
and pick $\delta\in\field_q$ which is not a square. Then the matrices are 
$$
K_1:=\left(
\begin{array}{ccc}
     0 & 0 & \Lambda_m \\
     0 & 1 & 0 \\
     \Lambda_m & 0 & 0 \\
\end{array}
\right),
K_\delta:=\left(
\begin{array}{ccc}
     0 & 0 & \Lambda_m \\
     0 & \delta & 0 \\
     \Lambda_m & 0 & 0 \\
\end{array}
\right).
$$
We denote the corresponding orthogonal groups by $O_n^+(\field_q),O_n^{-}(\field_q)$ (respectively). We remark (for later use) that the two groups become conjugate in $GL_n(\field_{q^2})$. 

\maincase{$n=2m$ is even}
Then the matrices are
$$
K_0=\left(
\begin{array}{cc}
     0 & \Lambda_m \\
     \Lambda_m & 0
\end{array}
\right),
K_{1,-\delta}=\left(
\begin{array}{ccc}
     0 & 0 & \Lambda_{m-1} \\
     0 & J_2 & 0 \\
     \Lambda_{m-1} & 0 & 0 \\
\end{array}
\right),
$$
where here $J_2=\left(
\begin{array}{cc}
     1 & 0 \\
     0 & -\delta \\
\end{array}
\right)$. Again, We denote the corresponding orthogonal groups by $O_n^+(\field_q), O_n^{-}(\field_q)$ (respectively), and note that they become conjugate in the quadratic extension $\field_{q^2}/\field_q$.

Fix $\epsilon=\pm 1$. We define the special orthogonal group 

$$
SO_n^{\epsilon}(\field_q)=\{
g\in O_n^{\epsilon}(\field_q):\det(g)=1
\}.
$$ Let $K$ be the form preserved by $O_n^\epsilon(\field_q)$, $SO_n^\epsilon(\field_q)$. The associated Lie algebra of $O_n^\epsilon(\field_q)$ is
\begin{equation}\label{eqn_defining_go_lie_algebra}
\mathfrak{o}_n^\epsilon(\field_q)=\mathfrak{o}_n(\field_q,K)=\left\{
A\in M_n(\field_q):KA=-A^tK
\right\},    
\end{equation}
and that of $SO_n^\epsilon(\field_q)$ is 
\begin{equation*}
\mathfrak{so}_n^\epsilon(\field_q)=\mathfrak{so}_n(\field_q,K)=\left\{
A\in M_n(\field_q):\tr(A)=0,KA=-A^tK
\right\}.
\end{equation*}
Note that in characteristic $\neq 2$, $\mathfrak{so}^\epsilon_n(\field_q)=\mathfrak{o}^\epsilon_n(\field_q)$, since for $A\in \mathfrak{o}_n^\epsilon(\field_q)$ we have $\tr(A)=\tr(-K^{-1}A^t K)=-\tr(A)$.

We pay special attention to the two forms $K_0=\left(
\begin{array}{cc}
     0 & \Lambda_m \\
     \Lambda_m & 0
\end{array}
\right)=\Lambda_{2m},K_1=\left(
\begin{array}{ccc}
     0 & 0 & \Lambda_m \\
     0 & 1 & 0 \\
     \Lambda_m & 0 & 0 \\
\end{array}
\right)=\Lambda_{2m+1}$. We define the anti-diagonal transpose of a matrix $A$ by reflecting it along the main anti-diagonal (analogously to the regular transpose). Denote this anti-diagonal transpose of $A$ by $A^\mathfrak{t}$. One can see that the condition $\Lambda_n A=-A^t \Lambda_n$ simply means that $A$ is anti-symmetric w.r.t. $A^\mathfrak{t}$, that is, $A^\mathfrak{t}=-A$. Thus the Lie algebra in those cases has a nice interpretation. 

As in the symplectic case, the characteristic polynomial of a matrix $M\in SO_n^{\epsilon}(\field_q)$ must be self-reciprocal. Moreover, also in this case one can derive the Jordan-type decomposition by observing the action of $M$ on $V=\field_q^n$, which gives it a structure of an $\field_q[x]$-module. As explained in \cite{Tay16}, one can decompose the space to a direct sum $\bigoplus_{f\in\Pee^\rs, i}(\field_q[x]/(f)^{n_i})$, such that $n=\sum_{f\in\Pee^\rs,i}n_i\deg(f)$. For a given $f\in\Pee^\rs$, we define the $f$-primary component 

$$
V_{f}\cong\bigoplus_{i\ge 1} (\field_q[x]/f^i)^{m_i}.
$$

Let $g\in O_n^\epsilon(\field_q)$. Consider now the primary components of $g$ associated with the polynomials $x\pm 1$,
$$
V_{x\pm 1}\cong \bigoplus (\field_q[x]/(x\pm 1)^i)^{m_i}.
$$
Wall \cite[p. 38]{Wal63} proved that for even $i$ the multiplicity $m_i$ must be even. For odd $i$, similarly to the symplectic case, he attached a certain invariant form (one of two forms $\Psi_{-}, \Psi_{+}$) to the space $(\field_q[x]/(x\pm 1)^i)^{m_i}$. As in the symplectic case, the sign of the form is determined by the induced action of $g$ on $(\field_q[x]/(x\pm 1)^i)^{m_i}$. 
\begin{theorem}[\cite{Wal63}, p. 38-9, rephrased]
Fix $\epsilon=\pm$. Two matrices $g,h\in O_n^{\epsilon}$ are conjugates if and only if 
\begin{enumerate}
    \item Their conjugacy class data is the same and in addition,
    \item For every odd $i$, the sign of the form corresponding with $(\field_q[x]/(x\pm 1)^i)^{m_i}$ is the same between $g,h$. 
\end{enumerate}
\end{theorem}

\begin{definition}\label{def_ortho_conjugacy_class_data}
Fix $\epsilon=\pm$. For a matrix $g\in O_{n}^\epsilon(\field_q)$, we define the orthogonal conjugacy class data by the multiset 
$$((f^{i},m_i))_{f\in\mathcal{P}^\rs\atop{f\neq x\pm 1\atop{i=1,\ldots}}}\cup ((f^i,{m_i}))_{2| i\atop{f=x\pm 1\atop{i=1,\ldots}}}\cup ((f^i,{m_i},\mu(i)))_{2\nmid i\atop{f=x\pm 1\atop{i=1,\ldots}}},$$
where $\mu(i)=\pm$ is the sign of the invariant attached to the space $(\field_q[x]/(x\pm 1)^{i})^{m_i}$.
\end{definition}
The last theorem by Wall shows that the orthogonal conjugacy class data parametrize conjugacy classes in $O_n^\epsilon(\field_q)$. When it is clear from the context that we consider the conjugacy class of $g$ in $O_n^\epsilon(\field_q)$, we will sometimes call the orthogonal conjugacy class data just the conjugacy class data of $g$. We also define an orthogonal signed partition, which provides a purely combinatorial parametrization of $O^\epsilon_n(\field_q)$-conjugacy classes. 

\begin{definition}
    An orthogonal signed partition is a collection of partitions $(\lambda_f)_{f\in\Pee}$ (some of them are tagged partitions), at most finitely many of them are non-empty, such that 
    
\begin{enumerate}
    \item For $\alpha=\pm 1$,
    \begin{enumerate}
        \item Each even part of $\lambda_{x-\alpha}$ has even multiplicity.
        \item Let $i>0$ be odd and assume that there is at least one part in $\lambda_{x-\alpha}$ with size $i$. To the collection of all parts of size $i$ we attach a sign $\mu(i)=\pm$ (corresponding to $\mu(i)$ in definition \ref{def_ortho_conjugacy_class_data}).
    \end{enumerate}
\item For $f\in\Pee$ with $f\neq x\pm 1$, $\lambda_f=\lambda_{f^\rs}$.
\end{enumerate}
\end{definition}

Let $\lambda$ be an orthogonal signed partition. We will denote $|\lambda|=\sum_{f\in\Pee}|\lambda_f|$. The set of orthogonal signed partitions with $|\lambda|=n$ gives a combinatorial parametrization of $O^\epsilon_n(\field_q)$-conjugacy classes, since a polynomial $f\in\Pee^\rs$ is either in $\Pee$ or factors as $f=hh^\rs$ for $h,h^\rs\in\Pee$. 

To complement these results, we would like to have (similarly to the symplectic case) a way to join matrices with different primary components, such that the resulting matrix will be in the correct orthogonal group. The problem is this: if we have matrices $A, B$ preserving the forms $J_A, J_B$, the diagonal join will preserve the matrix $\mathrm{diag}(J_A, J_B)$, which is not of the standard types we defined before. Thus we need first to understand how to compute the type of $\mathrm{diag}(J_A, J_B)$, and later to work out a linear transformation to the corresponding standard form.

This will be done using the orthogonal join, which in the case of the orthogonal group is more complicated, and must treat the Witt types of the associated quadratic forms. We will now briefly discuss the Witt ring $W(\field_q)$. See \cite[\S 2]{Lam05} for a detailed exposition of Witt rings in this setting and in general.

\begin{definition}
Let $V$ be a finite-dimensional vector space over $\field_q$, and let $Q: V\times V\to \field_q$ be a quadratic form. The pair $(V, Q)$ is called a quadratic space.
\end{definition}
 We sometimes abuse notation slightly and call the quadratic space $(V, Q)$ simply $V$. As an example for a quadratic space, we define $\mathbb{H}=(\field_q^2,(x,y)\mapsto x^2-y^2)$. It is not hard to see that this is a quadratic space. $\mathbb{H}$ is called the hyperbolic space.

For $(V_1,Q_1),(V_2,Q_2)$ one may define their direct sum $(V_1\oplus V_2,Q_1\oplus Q_2)$, where $Q_1\oplus Q_2$ is the quadratic form defined via the equation
$$
(Q_1\oplus Q_2)(v_1\oplus v_2,v_1'\oplus v_2')=Q_1(v_1,v_1')+Q_2(v_2,v_2').
$$
We denote by $V_1^{m}$ the direct sum $\underbrace{V_1\oplus V_1\cdots\oplus V_1}_{m\text{ times}}$.

\begin{definition}
Let $(V_1,Q_1), (V_2,Q_2)$ be two quadratic spaces over $\field_q$. A linear map $L: V_1 \to V_2$ is called an isomorphism of quadratic spaces if it is an isomorphism of vector spaces such that $Q_1=Q_2\circ L$. If such a map exists, we say that the quadratic spaces $V_1, V_2$ are isomorphic, and write $V_1\cong V_2$.
\end{definition}

\begin{definition}
Let $(V_1,Q_1),(V_2,Q_2)$ be two quadratic spaces. $V_1,V_2$ are called Witt equivalent if $V_1\cong V_2\oplus\mathbb{H}^m$ for some non-negative integer $m$. We denote $V_1\sim V_2$. 
\end{definition}
One can show that if $V_1\sim V_1'$, $V_2\sim V_2'$, then $V_1\oplus V_2\sim V_1'\oplus V_2'$, and $V_1\otimes V_2\sim V_1'\otimes V_2'$ (for an appropriate definition of the tensor product of quadratic spaces see \cite[Chapter 1, \S 6]{Lam05}). This induces a ring structure on the set of quadratic spaces over $\field_q$ modulo Witt equivalence. The resulting ring $W(\field_q)$ is called the Witt ring. For a quadratic form $(x_1,\ldots,x_n)\mapsto a_1x_1^2+\cdots+a_nx^n$, we denote its equivalence class in the Witt ring by $\langle
a_1,\ldots,a_n
\rangle$. 

\begin{theorem}
Let $\delta\in\field_q$ be a non-square. The Witt ring is composed of the elements $\langle
0
\rangle,\langle
1
\rangle,\langle
\delta
\rangle,\langle
1-\delta
\rangle
$ (corresponding with the equivalence classes of the hyperbolic form, $x^2$, $\delta x^2$ and $x^2-\delta y^2$, respectively). Moreover,
\begin{enumerate}
    \item If $q\equiv 1\pmod{4}$, $W(F)\cong \field_2[\epsilon|\epsilon^2=0]$.
    \item If $q\equiv 3\pmod{4}$, $W(F)\cong \Z/4\Z$.
\end{enumerate}
\end{theorem}

\begin{proof}
See \cite[Corollary 3.6]{Lam05} for a proof.
\end{proof}
Let $M$ be a symmetric matrix over $\field_q$. $M$ defines a quadratic form, which by the previous theorem must be equivalent to one of the elements $\langle
0
\rangle,\langle
1
\rangle,\langle
\delta
\rangle,\langle
1-\delta
\rangle
$. The Witt type of $M$ will be its Witt equivalence class. 

Now that we considered the necessary background regarding the Witt ring, we turn back to our problem of forming the orthogonal join. Let $A\in O_n(\field_q,J_A)$, $B\in O_m(\field_q,J_B)$. The block matrix 
$$
A\oplus B:=\left(
\begin{array}{cc}
    A & 0 \\
    0 & B
\end{array}
\right)
$$
is in $O_{n+m}(\field_q,J_A\oplus J_B)$, where 
$$
J_A\oplus J_B:=
\left(
\begin{array}{cc}
    J_A & 0 \\
    0 & J_B
\end{array}
\right).
$$
Even when $J_A, J_B$ are in one of the standard forms defined above, $J_A\oplus J_B$ need not be. However, it must be congruent to one of the standard forms $K_0, K_1, K_\delta, K_{1,-\delta}$ (of Witt types $\langle
0
\rangle,\langle
1
\rangle,\langle
\delta
\rangle,\langle
1-\delta
\rangle$, respectively). We remark that as a consequence when $n$ is odd, the only possible Witt types for a symmetric $M\in M_n(\field_q)$ are $\langle 1\rangle,\langle \delta\rangle$. Similarly, for $n$ even, the only possible Witt types are $\langle 0\rangle,\langle 1-\delta\rangle$

To understand which one of the forms is appropriate, we must compute the Witt type of $J_{A}\oplus J_B$. This corresponds to forming the addition $\langle
J_A
\rangle + \langle
J_B\rangle
$ in the Witt ring. Suppose that the standard form $K$ satisfies $\langle K\rangle =\langle
J_A
\rangle + \langle
J_B\rangle$. By the definition of the Witt type, there is always a matrix $X$ such that 
\begin{equation}\label{eqn_K_def}
X^t (J_A\oplus J_B) X=K.   
\end{equation}
When $J_A,J_B$ are standard forms, we denote $K$ by $J_{A\oplus_{\mathrm{orth}} B}$. We are now ready to define the orthogonal sum $A\oplus_{\mathrm{orth}}B$.

\begin{claim}\label{claim_orthogonal_join_def}
Let $A\in O_n(\field_q,J_A)$, $B\in O_m(\field_q,J_B)$. Let $K, X$ be as above. Then, the matrix $A\oplus_{\mathrm{orth}}B:=X^{-1}(A\oplus B)X$ preserves the form $K$. $A\oplus_{\mathrm{orth}} B$ is called the orthogonal direct sum of $A$ and $B$.
\end{claim}

\begin{proof}
Indeed, 

$$
(A\oplus_{\mathrm{orth}} B)^t K (A\oplus_{\mathrm{orth}} B)=X^{t}(A^t\oplus B^t)X^{-t}\cdot X^t (J_A\oplus J_B) X\cdot X^{-1}(A\oplus B)X=K.
$$
\end{proof}
In particular, when $J_A,J_B$ are standard forms, $J_{A\oplus_{\mathrm{orth}}B}$ is also standard and $A\oplus_{\mathrm{orth}}B\in O_n(\field_q,J_{A\oplus_{\mathrm{orth}} B})$. We now consider what happens to the associated Lie algebras.

\begin{claim}\label{claim_lie_algebra_of_sum}
Let $A\in O_n(\field_q,J_A)$, $B\in O_m(\field_q,J_B)$. Let $K, X$ be as above. Then, 
$$\mathfrak{o}(\field_q,K)=X^{-1}\mathfrak{o}(\field_q,J_A\oplus J_B)X.$$
\end{claim}

\begin{proof}
Indeed, let $Y\in \mathfrak{o}(\field_q,J_A \oplus J_B)$. Using equation \eqref{eqn_K_def} and the definition of the Lie algebra,
\begin{multline*}
K(X^{-1}YX)=X^t (J_A\oplus J_B)YX=-X^tY^t (J_A\oplus J_B)X=-X^tY^tX^{-t}X^t (J_A\oplus J_B)X
=(X^{-1}YX)^tK.
\end{multline*}
Hence we have the inclusion 
$$\mathfrak{o}(\field_q,K)\subset X^{-1}\mathfrak{o}(\field_q,J_A\oplus J_B)X,$$
Thus the LHS is contained in the RHS, and by symmetry we get equality.
\end{proof}

When $J_A,J_B$ are standard that means that $\mathfrak{o}(\field_q,J_{A\oplus_{\mathrm{orth}}B})=X^{-1}\mathfrak{o}(\field_q,J_A\oplus J_B)X$.

\subsubsection{The unitary group}\label{sec_unitary_generalities}

Let $\tau:\field_{q^2}\to\field_{q^2}$ be the $q$-Frobenius map, i.e. $\tau(x)=x^q$. We extend the action of Frobenius to matrices in $M_n(\field_{q^2})$ by setting $\tau(M)$ to be the matrix obtained from $M$ by raising each entry to the $q$-th power. Now define $M^*=\tau(M^t)$. The unitary group over $\field_q$ is defined by 
$$U_n(\field_q)=\{M\in GL_n(\field_{q^2}):MM^*=I\}.$$ The associated Lie algebra is 
$$\mathfrak{u}_n(\field_q)=\{M\in M_n(\field_{q^2}):M^*=-M\}.$$

Similarly to the symplectic and orthogonal case, the characteristic polynomial of a unitary matrix admits symmetries. For $f\in \field_{q^2}[x]$ monic, let $\tau(f)$ be the action of Frobenius on each coefficient separately. Then, define $\tilde{f}(t)=\frac{t^n\tau(f)(1/t)}{\tau(f)(0)}$. For a unitary matrix $M$, if $\charc(M)=f$, then $\tilde{f}=f$. Such a polynomial is called self-skew-reciprocal. A polynomial is called $\sim$-irreducible if it cannot be decomposed as the product of two self-skew-reciprocal polynomials. Denote by $\tilde{\Pee}$ the set of monic $\sim$-irreducible polynomials. Finally, we note that $f$ is $\tau(f(0))$-skew palindromic.

As in the $GL_n$ case, conjugacy classes are determined by the (usual) Jordan form over $GL_n(\field_{q^2})$. However, there are some restrictions on the conjugacy class data. Let $g\in U_n(\field_q)$ be a unitary matrix. View $\field_{q^2}^n$ as an $\field_{q^2}[x]$-module by the action $f(x)v=f(g)v$. It turns out (see \cite[p. 34]{Wal63}) that one can always write

$$
V\cong \bigoplus_{f\in \tilde{\mathcal{P}}}\bigoplus_{i=1}^{k_f} \field_q[x]/(f)^{m_i}\cong \bigoplus_{f\in\tilde{\mathcal{P}}} V_f,V_f\cong\bigoplus_{i=1}^{k_f}\field_q[x]/(f)^{m_i}.
$$
That is, the primary parts of $g$ must belong to $\sim$-irreducible polynomials. To join two unitary matrices, we simply write them on the diagonal of a block matrix. We end the section with the definitions of a unitary conjugacy class data and a unitary partition.

\begin{definition}
For a matrix $g\in U_n(\field_q)$, we define the unitary conjugacy class data by the multiset 
$
((f^i,m_i))_{f\in\tilde{\Pee}\atop{i=1,\ldots}}.
$
\end{definition}

\begin{definition}
A unitary partition $\lambda$ is a collection of partitions $(\lambda_f)_{f\in\Pee}$, at most finitely many of them are non-empty, such that $\lambda_f=\lambda_{\tilde{f}}$. 
\end{definition}
Let $\lambda$ be a unitary partition. We will denote $|\lambda|=\sum_{f\in\Pee} |\lambda_f|$. The set of unitary partitions with $|\lambda|=n$ gives a combinatorial parametrization of $U_n(\field_q)$-conjugacy classes, since a polynomial $f\in\tilde{\Pee}$ is either in $\Pee$ or it factors as $f=h\tilde{h}$ where $h,\tilde{h}\in\Pee$. 

\subsection{The effect of lifting on traces and the characteristic polynomial}\label{section_lifting}

For the following subsection, let $G\subset GL_n$ be a finite matrix group from the following list: $G=GL_n$, $SL_n$, $SO_n$, $Sp_{2n}$, $U_n$. The mod $\pi^k$ reduction map $G(\Ol)\to G(\resring{k})$ is surjective for all $k$ by Lemma \ref{lem_PRR_reductive_surjective}, so we can fix once and for all for every $k$ a system $H_{G,k}\subset G(\resring{k})$ of representatives for $G(\resring{k-1})$. Also, let $\mathfrak{g}$ be the associated Lie algebra. For the rest of section \ref{section_lifting}, we let $k\ge 2$ unless explicitly stated otherwise.

\begin{lem}
Any $A\in G(\resring{k})$ can be written uniquely as the product
$$
A = A_0(1+\pi^{k-1}A_1)
$$
where $A_0\in H_{G,k}$ and $A_1\in \mathfrak{g}(\mathbb{F}_q)$.     
\end{lem}

\begin{proof}
We prove the lemma for every matrix group separately.

\maincase{$G=GL_n,SL_n$}
Let $M\in GL_n(\resring{k})$, and write $M=M_0+\pi^{k-1}M_1$, where $M_0\in H_{G,k}$. Denote the $i$'th row of a matrix $X$ by $X_i$. The determinant is multilinear in the rows, and we write this explicitly by

\begin{multline*}
\det(M_0+\pi^{k-1}M_1)=\sum_{(s_1,\ldots,s_n)\in \{1,2\}^n} \det((M_{s_1})_1,\ldots,(M_{s_n})_n)=\\
=\det(M_0)+\pi^{k-1}\sum_{i=1}^n\det((M_0)_1,\ldots,(M_1)_i, \ldots, (M_0)_n)=\det(M_0)+\pi^{k-1}\tr(\Adj(M_0)M_1).
\end{multline*}

No matter what $M_1$ is, the determinant is thus invertible. Hence, we get that $M=M_0+\pi^{k-1}M_1$ where $M_1$ is any matrix $\pmod{\pi}$, and since $M_0$ is invertible, we get the original claim. For $SL_n$, we have the additional constraint $\det(M)=1$. Since $M_0\in G$, we conclude that we must have $\tr(\Adj(M_0)M_1)\equiv 0\pmod{\pi}$. If we write $M_1=M_0M_1'$, we get that 
$$\det(M_0)\tr(M_1')=\tr(\Adj(M_0)M_0M_1')\equiv 0\pmod{\pi}\iff \tr(M_1')\equiv 0\pmod {\pi},$$
and thus $M_1'\in\mathfrak{g}(\field_q)$, and the claim follows. 

\maincase{$G=Sp_{2n}$}
Let $M\in Sp_{2n}(\resring{k})$, writing $M=M_0+\pi^{k-1}M_1$, where $M_0\in H_{G,k}$. Let $\Omega_n$ be the matrix representing the standard symplectic form. Then for $M$ to be symplectic, we need to have

$$
\Omega_n=M\Omega_n M^t=(M_0+\pi^{k-1}M_1)\Omega_n (M_0^t+\pi^{k-1}M_1^t)=M_0\Omega_n M_0^t+\pi^{k-1}(M_0\Omega_n M_1^t+M_1\Omega_n M_0^t).
$$
Since we assumed $M_0$ is symplectic, we are left with the condition $M_0\Omega_n M_1^t   +M_1\Omega_n M_0^t=0\pmod {\pi}$. Write $M_1=M_0M_2$. Then we have
$$
M_0\Omega_n M_2^tM_0^t+M_0M_2\Omega_n M_0^t=0\iff \Omega_n M_2^t+M_2\Omega_n=0,
$$
which means that $M_2$ is in $\mathfrak{sp}_{2n}(\field_q)$.

\maincase{$G=SO_n$}
Let $M\in O_n(\resring{k})$, and write as usual $M=M_0+\pi^{k-1}M_1$ with $M_0\in H_{G,k}$. Let $K$ be the matrix representing the symmetric form preserved by the orthogonal group. Then 

$$
K=M^t KM=(M_0^t+\pi^{k-1}M_1^t)K(M_0+\pi^{k-1}M_1)=M_0^tKM_0+\pi^{k-1}(M_0^tKM_1+M_1^tKM_0).
$$
This implies that $M_0^tKM_1+M_1^tKM_0=0$, and the proof ends as in the symplectic case. The proof for the special orthogonal group goes the same, noting that $\mathfrak{so}=\mathfrak{go}$. We remark that the reduction of $O_n, SO_n$ modulo $\pi^k$ might have different signs when we change $n$, and hence we discussed both options.

\maincase{$G=U_n$}
Let $M\in U_n(\resring{k})$. Write as usual $M=M_0+\pi^{k-1} M_1$ with $M_0\in H_{G,k}$. Then,

$$
I=M^*M=(M_0^*+\pi^{k-1}M_1^*)(M_0+\pi^{k-1}M_1)=M_0^*M_0+\pi^{k-1}(M_0^*M_1+M_1^*M_0).
$$
This implies $M_1^*M_0=-M_0^*M_1$. Writing $M_1=M_0M_1'$, we get that $M_1'\in\mathfrak{g}(\field_q)$. 

\end{proof}

Using this decomposition, one can study the distribution of traces of powers and of the next-to-leading coefficients characteristic polynomial.

\begin{lem}\label{lem_trace_derivative}
    Let $A\in G(\resring{k})$ be written as $A=A_0(1+\pi^{k-1}A_1)$, where $A_0\in H_{G,k}$ and $A_1\in \mathfrak{g}(\field_q)$. Then,
    $$
    \tr(A^r)=\tr(A_0^r)+r\pi^{k-1}\tr(A_0^rA_1).
    $$
\end{lem}

\begin{proof}
    Write 
    $
    A^r=A_0^r+\pi^{k-1}\sum_{i=1}^{r} A_0^i A_1 A_0^{r-i}.
    $
    The claim now follows since the trace map is cyclic.
\end{proof}

In particular, if one understands the distribution of traces of powers for $G(\resring{k-1})$, the distribution of traces of powers for $G(\resring{k})$ reduces to the value distribution of the linear functional

\begin{equation}\label{eqn_def_trace_linear_functional}
 \partial\tr_{A_0}:\mathfrak{g}(\field_q)\to \field_q, A_1\mapsto r\tr(\overline{A_0}^r A_1).
\end{equation}
One can get similar results for the characteristic polynomial.

\begin{lem}\label{lem_charpoly_step_in_p_direction}
Let $G$ be one of the groups considered above.
Let $A\in G(\resring{k})$ be written as $A=A_0(1+\pi^{k-1}A_1)$, where $A_0\in H_{G,k}$ and $A_1\in \mathfrak{g}(\field_q)$. Then,
    $$
    \charc(A)=\charc(A_0)-\pi^{k-1}\tr(\Adj(x-A_0) A_0A_1).
    $$
Further, if $G\neq GL_n,U_n$,

    $$
    \charc(A)=\charc(A_0)-\pi^{k-1}x\cdot\tr(\Adj(x-A_0) A_1).
    $$
\end{lem}

\begin{proof}
Write $\charc(A)=\det(xI-(A_0+\pi^{k-1}A_0A_1))=\det((xI-A_0)-\pi^{k-1}A_0A_1)$. Denote the $i$'th row of a matrix $X$ by $X_i$, and denote $M_1=xI-A_0, M_2=-\pi^{k-1}A_0A_1$. The determinant is multilinear in the rows, and we write this explicitly by
$$
\det((xI-A_0)-\pi^{k-1}A_1)=\det((M_1)_1+(M_2)_1,\ldots,(M_1)_n+(M_2)_n)=\sum_{(s_1,\ldots,s_n)\in \{1,2\}^n} \det((M_{s_1})_1,\ldots,(M_{s_n})_n).
$$
Now note that since $M_2=-\pi^{k-1}A_0A_1$, and $\pi^{2k-2}=0$ in $\resring{k}$, we have
$$
\det((xI-A_0)+\pi^{k-1}A_0A_1)=\det(xI-A_0)-\pi^{k-1}\sum_{i=1}^n\det((xI-A_0)_1,\ldots,(A_0A_1)_i, \ldots, (xI-A_0)_n)
$$
Developing the determinant in the summand w.r.t. the $i$'th row, we get that

$$
\charc(A)=\charc(A_0)-\pi^{k-1}\tr(\Adj(x-A_0)A_0 A_1).
$$

This finishes the first part of the claim. For $G\neq GL_n, U_n$, we note that the trace map vanishes on the Lie algebra $\mathfrak{g}(\field_q)$. Thus,

\begin{multline*}
\tr(\Adj(x-A_0)A_0 A_1)=\tr(\Adj(x-A_0)(A_0-x)A_1)+\tr(\Adj(x-A_0)xA_1)=\\ =-\charc(A_0)\tr(A_1)+x\cdot\tr(\Adj(x-A_0)A_1)=x\cdot\tr(\Adj(x-A_0)A_1).
\end{multline*}

\end{proof}

Thus, given results on the distribution of the next-to-leading coefficients of the characteristic polynomial for $G(\resring{k-1})$, one can lift them to $G(\resring{k})$ by studying linear maps on the Lie algebra $\mathfrak{g}(\field_q)$. To do this, we define the $\field_q$-linear maps

\begin{equation}\label{eqn_def_charpoly_linear_maps}
\partial\charc_{A_0}:\mathfrak{g}(\field_q)\to\field_q[x]_{<n}, A_1\mapsto \tr(\Adj(x-\overline{A_0}) A_0A_1).
\end{equation}
We remark that the image of $\partial\charc_{A_0}$ is indeed inside $\field_q[x]_{<n}$ due to the fact that $\charc(A),\charc(A_0)$ are monic of degree $n$. We also remark that for $G\neq GL_n,U_n$ we have 
$$
\partial\charc_{A_0}(A_1)=x\cdot\tr(\Adj(x-\overline{A_0})A_1).
$$

\subsection{Explicit representatives for primary conjugacy classes} \label{subsection_explicit_representatives}

In sections \ref{section_gln_main}-\ref{section_un_main} we will compute the image of the linear maps
$\partial\charc_{A_0}$ defined in \eqref{eqn_def_charpoly_linear_maps} for the finite matrix groups $G=GL_n(\field_q), SL_n(\field_q), Sp_{2n}(\field_q), SO_n(\field_q)$ and $U_n(\field_q)$. For any of the groups $G$ considered, and for all $a\in G$, the Lie algebra $\mathfrak{g}$ is preserved by conjugation by $a$: $a\mathfrak{g}a^{-1}=\mathfrak{g}$. Now note that for any $a\in G$, $A_0\in G(\field_q)$, $A_1\in\g(\field_q)$,

$$
\tr(\Adj(x-A_0)\cdot aA_1a^{-1})=\tr(a^{-1}\Adj(x-A_0) aA_1)=\tr(\Adj(x-a^{-1}A_0a) A_1).
$$
We conclude that $\img(\partial\charc_{A_0})$ depends only on the conjugacy class of $\overline{A_0}$ in $G(\field_q)$ (at least in the cases of $G\neq GL_n, U_n$. For $G=GL_n, U_n$ we will use slightly modified ideas described in the corresponding sections). Hence it is useful to compute the matrix $\Adj(x-A_0)$ explicitly (at least up to conjugation). 

Therefore, we would like to construct explicit representatives for the various conjugacy classes in each group. We will study the image $\img(\partial\charc_{A_0})$ over an algebraic extension $\field_{q^l}$ in which $\charc(A_0)$ splits completely, and use this data to understand the image over $\field_q$. Thus we will be interested in conjugacy classes of $G(\field_{q^l})$ for matrices whose characteristic polynomial splits completely into linear factors. We will also compute the matrix $\Adj(x-A_0)$ for the explicit representatives $A_0$. In fact, we may focus on cases where $\charc(M)=\min(M)=(x-\alpha)^e$ for some $\alpha\in\field_{q^l}$, $e>0$. We will later use a ``join" operator to join different primary parts, creating more complex conjugacy classes.

\subsubsection{For $GL_n$ and $SL_n$}

As is well known any matrix $M\in GL_n(\field_q)$ whose characteristic polynomial splits completely admits a Jordan normal form. The Jordan form characterizes such conjugacy classes. As mentioned above, it is enough to treat the case of $\charc(M)=\min(M)=(x-\alpha)^e$, in which $M$ is a conjugate of a Jordan block $J_e(\alpha)$, where here (and for the rest of the paper)

$$
J_e(\alpha):=\left(
\begin{array}{ccccc}
    \alpha & 1 & 0 & \cdots & 0 \\
    0 & \alpha & 1 & \ddots & 0 \\
     & & \ddots & \ddots &  \vdots\\
     &  &  & \alpha & 1 \\
     &  &  &  0 & \alpha \\
\end{array}
\right),\alpha\in\field_q,e\ge 1\text{ integer}.
$$

\begin{lem}\label{lem_gl_jordan_block_adj}
Let $J_e(\alpha)$ be a Jordan block. Then, 

$$
 X_\alpha^e:=\Adj(x-J_e(\alpha))=(x-\alpha)^e\left(\begin{array}{ccccc}
    (x-\alpha)^{-1} & (x-\alpha)^{-2} & (x-\alpha)^{-3} & \cdots & (x-\alpha)^{-e} \\
    0 & (x-\alpha)^{-1} & (x-\alpha)^{-2} & \cdots & (x-\alpha)^{-(e-1)} \\
    \vdots & \vdots & \ddots & \cdots & \vdots \\
    0 & \cdots & 0 & (x-\alpha)^{-1} & (x-\alpha)^{-2} \\
    0 & 0 & \cdots & 0 & (x-\alpha)^{-1} \end{array}\right).
$$
\end{lem}

\begin{proof}
We have $\Adj(x-J_e(\alpha))=(x-\alpha)^e (x-J_e(\alpha))^{-1}$. Note that $x-J_e(\alpha)=-J_e(\alpha-x)$, and using the formula for the inverse of a Jordan block we finish.
\end{proof}

For $SL_n$, we note that $\mathfrak{sl}_n=\{a\in M_n:\tr(a)=0\}$. This Lie algebra admits more symmetries than conjugation by $SL_n$: indeed, for any $g\in GL_n$, $a\in \mathfrak{sl}_n$, we have $\tr(gag^{-1})=\tr(a)=0$. Thus, $gag^{-1}\in\mathfrak{sl}_n$ as well. We get that to compute the image of $\partial\charc_{A_0}$ for $A_0\in SL_n$, one can replace $A_0$ by an element from its conjugacy class in $GL_n$ (in particular, a matrix given in a Jordan form). 

\subsubsection{For $Sp_{2n}$}\label{subsubsection_explicit_representatives_sp2n}

We discussed $Sp_{2n}(\field_q)$-conjugacy classes in section \ref{subsubsection_symplectic_group_intro}. To compute $\Adj(x-A_0)$ for representatives of $Sp_{2n}(\field_q)$-conjugacy classes, we would like to deal with simple building blocks like the Jordan block. The analog of a single Jordan block is a matrix conjugate in $GL_n(\field_q)$ to one of the following types I, II, and III. We also provide conjugacy class representatives (from $Sp_{2n}(\field_q)$) to each type, following \cite{Tay21a}. 

\begin{enumerate}
    \item \textbf{Type I:} The class with Jordan form $
\left(
\begin{array}{cc}
    J_m(\alpha) & 0 \\
    0 & J_m(\alpha^{-1})
\end{array}
\right) \text{ for } \alpha\neq \pm 1$. The representative for this class is $A_{\alpha,\alpha^{-1}}:=
\left(
\begin{array}{cc}
    J_m(\alpha)^{-1} & 0 \\
    0 & J_m(\alpha)^t
\end{array}
\right).
$

\item \textbf{Type II:} The class with Jordan form $\left(
\begin{array}{cc}
    J_m(\alpha) & 0 \\
    0 & J_m(\alpha)
\end{array}
\right)$ for $\alpha=\pm 1$,  $m$ odd. The representative for this class is $
A_{\alpha}^{m}:=\left(
\begin{array}{cc}
    J_m(\alpha)^{-1} & 0 \\
    0 & J_m(\alpha)^t
\end{array}
\right)$.

\item \textbf{Type III:} The class with Jordan form $J_{2m}(\alpha)$, for $\alpha=\pm 1$. Let $\delta\in \field_q$ be a non-square, and define 
$$a=
\begin{cases}
    1, & \text{ for } + \text{ type,}\\
    (-1)^m 2\delta, & \text{ for } - \text{ type.}\\
\end{cases}
$$

Then, a representative for such matrices with $\pm$ sign is

$$
A_{\alpha,\pm}^{2m}:=\left(
\begin{array}{cc}
    \alpha J_m(1)^{-1}  & 0 \\
    aS & \alpha J_m(1)^{t}
\end{array}
\right),
$$

where we note that $a$ depends on the sign, and we denote 

$$
S = 
\left(
\begin{array}{ccccc}
    1 & -1 & \cdots & (-1)^{m-2} & (-1)^{m-1} \\
    0 & 0 & \cdots & 0 & 0 \\
     &  & \ddots&  &  \\
    0 & 0 & \cdots & 0 & 0
\end{array}
\right).
$$
\end{enumerate}

We may construct every conjugacy class from these building blocks using the triangular join as in Definition \ref{def_triangular_join}. It now remains to compute the adjoint matrix $\Adj(x-M)$ for $M$ a representative for one of the types. Before we do this, we need the following

\begin{prop}\label{prop_matrix_comp_sp2n}
Let $e>0$ be an integer and let $\alpha\in\field_q$. Let $Y_\alpha^e=\Adj(x-J_e(\alpha)^{-1})$. Then,
$$
Y_\alpha^e = (x-\alpha^{-1})^e \left(
\begin{array}{ccccc}
    \frac{-\alpha}{1-x\alpha} & \frac{-1}{(1-x\alpha)^2} & \cdots & \frac{-x^{e-3}}{(1-x\alpha)^{e-1}} & \frac{-x^{e-2}}{(1-x\alpha)^e} \\
     & \frac{-\alpha}{1-x\alpha} & & & \frac{-x^{e-3}}{(1-x\alpha)^{e-1}} \\
     & & \ddots & & \vdots \\
     & &  & \frac{-\alpha}{1-x\alpha} & \\
     & & & & \frac{-\alpha}{1-x\alpha}
\end{array}
\right).
$$
Now Let $\alpha=\pm 1$. Let $S$ be as defined above in the description of Type III. Define a sequence of rational functions
$f_1=\frac{1}{x-\alpha}$, $f_2=\frac{1}{(x-\alpha)^2}+\frac{1}{(x-\alpha)}$,\ldots, $f_e=\frac{x^{e-2}}{(x-\alpha)^e}+\frac{x^{e-3}}{(x-\alpha)^{e-1}}+\cdots
+\frac{1}{(x-\alpha)^2}+\frac{1}{(x-\alpha)}$. Then, 

\begin{multline*}
Z_\alpha^{2e}:=(x-\alpha)^{2e}(x-\alpha J_e(1))^{-t}S (x-\alpha J_e(1)^{-1})^{-1}=\\
=(x-\alpha)^{2e}\left(\begin{array}{ccccc}
   \frac{ f_1}{(x-\alpha)} &  \frac{- f_2}{(x-\alpha)} &  \frac{ f_3}{(x-\alpha)} & \cdots &  \frac{(-1)^{e-1} f_e}{(x-\alpha)} \\
    \frac{\alpha f_1}{(x-\alpha)^2} &  \frac{-\alpha f_2}{(x-\alpha)^2} &  \frac{\alpha f_3}{(x-\alpha)^2} & \cdots &  \frac{(-1)^{e-1}\alpha f_e}{(x-\alpha)^2} \\
    & & \ddots & & \\
    \frac{\alpha^{e-2} f_1}{(x-\alpha)^{e-1}} &  \frac{-\alpha^{e-2} f_2}{(x-\alpha)^{e-1}} &  \frac{\alpha^{e-2} f_3}{(x-\alpha)^{e-1}} & \cdots &  \frac{(-1)^{e-2}\alpha^{e-2} f_e}{(x-\alpha)^{e-1}} \\
    \frac{\alpha^{e-1} f_1}{(x-\alpha)^{e}} &  \frac{-\alpha^{e-1} f_2}{(x-\alpha)^{e}} &  \frac{\alpha^{e-1} f_3}{(x-\alpha)^{e}} & \cdots &  \frac{(-1)^{e-1}\alpha^{e-1} f_e}{(x-\alpha)^{e}} \end{array}\right).  
\end{multline*}

\end{prop}

\begin{proof}
We start with the claim on $Y_\alpha^e$. It is enough to compute $(x-J_e(\alpha)^{-1})^{-1}$. Write $J_e(\alpha)=\alpha+J_e(0)$, where $J_e(0)^e=0$ is nilpotent. We have:

$$
(x-J_e(\alpha)^{-1})^{-1}=\frac{1}{x-(\alpha+J_e(0))^{-1}}=\frac{-(\alpha+J_e(0))}{1-(\alpha+J_e(0))x}=\frac{-(\alpha+J_e(0))}{(1-x\alpha)-xJ_e(0)}=\frac{\frac{-(\alpha+J_e(0))}{1-x\alpha}}{1-\frac{x}{1-x\alpha}J_e(0)}.
$$
Since $J_e(0)$ is nilpotent, we get

$$
(x-J_e(\alpha)^{-1})^{-1}=-\frac{\alpha+J_e(0)}{1-x\alpha}\left(1+\frac{x}{1-x\alpha}J_e(0)+\cdots+\frac{x^{e-1}}{(1-x\alpha)^{e-1}}J_e(0)^{e-1}\right).
$$
Rewriting this we get

$$
(x-J_e(\alpha)^{-1})^{-1}=\frac{-1}{1-
x\alpha}\left(
\alpha+\left(1+\frac{\alpha x}{1-x\alpha}\right)J_e(0)+\cdots+\left(
\frac{x^{e-2}}{(1-x\alpha)^{e-2}}
+\frac{\alpha x^{e-1}}{(1-x\alpha)^{e-1}}\right)J_{e-1}(0)
\right).$$
For $i\ge 0$,

$$
\frac{x^i}{(1-x\alpha)^i}+\frac{\alpha x^{i+1}}{(1-x\alpha)^{i+1}}=\frac{x^i}{(1-x\alpha)^{i+1}},
$$
which finishes the proof of the first claim. Regarding the second claim, we first note that 
$
(x-\alpha J_e(1))^{-1}=\frac{1}{\alpha}\left(\frac{x}{\alpha}-J_e(1)\right)^{-1}=\frac{1}{\alpha(x/\alpha-1)^e}X_1^e\left(\frac{x}{\alpha}\right)
$, 
$
(x-\alpha J_e(1)^{-1})^{-1}=\frac{1}{\alpha}(\frac{x}{\alpha}-J_e(1)^{-1})^{-1}=\frac{1}{\alpha(x/\alpha-1)^e}Y_1^e(\frac{x}{\alpha})
$. Noting that $\alpha^2=1$ (and hence also $\alpha^{2e}=1$),  we get that

$$
(x-\alpha)^{2e}(x-\alpha J_e(1))^{-t}S(x-\alpha J_e(1)^{-1})^{-1}=(X_1^e)^t\left(\frac{x}{\alpha}\right) SY_1^e\left(\frac{x}{\alpha}\right).
$$
Again using the fact that $\alpha^2=1$, we see that

\begin{multline*}
Y_1^e\left(\frac{x}{\alpha}\right)=(x/\alpha-1)^e \left(
\begin{array}{ccccc}
    \frac{-1}{1-x/\alpha} & \frac{-1}{(1-x/\alpha)^2} & \cdots & \frac{-(x/\alpha)^{e-3}}{(1-x/\alpha)^{e-1}} & \frac{-(x/\alpha)^{e-2}}{(1-x/\alpha)^e} \\
     & \frac{-1}{1-x/\alpha} & & & \frac{-(x/\alpha)^{e-3}}{(1-x/\alpha)^{e-1}} \\
     & & \ddots & & \vdots \\
     & &  & \frac{-1}{1-x/\alpha} & \\
     & & & & \frac{-1}{1-x/\alpha}
\end{array}
\right)=\\
=\alpha^{e+1}(x-\alpha)^e \left(
\begin{array}{ccccc}
    \frac{1}{x-\alpha} & \frac{-1}{(x-\alpha)^2} & \cdots & \frac{(-1)^{e}x^{e-3}}{(x-\alpha)^{e-1}} & \frac{(-1)^{e+1}x^{e-2}}{(x-\alpha)^e} \\
     & \frac{1}{x-\alpha} & & & \frac{(-1)^{e}x^{e-3}}{(x-\alpha)^{e-1}} \\
     & & \ddots & & \vdots \\
     & &  & \frac{1}{x-\alpha} & \\
     & & & & \frac{1}{x-\alpha}
\end{array}
\right).
\end{multline*}
Hence,

\begin{equation*}
\left(X_1^e\left(\frac{x}{\alpha}\right)\right)^t SY_1^e\left(\frac{x}{\alpha}\right)=\alpha^e(x-\alpha)^e\left(X_1^e\left(\frac{x}{\alpha}\right)\right)^t\left(\begin{array}{ccccc}
   f_1 & -f_2 & f_3 & \cdots & (-1)^{e-1}f_e \\
    0 & 0 & 0 & \cdots & 0 \\
    \vdots & \vdots & \ddots & \cdots & \vdots \\
    0 & \cdots & 0 & 0 & 0 \\
    0 & 0 & \cdots & 0 & 0     
\end{array}\right).\end{equation*}
Now, compute

\begin{multline*}
X_1^e\left(\frac{x}{\alpha}\right)=(x/\alpha-1)^e\left(\begin{array}{ccccc}
    (x/\alpha-1)^{-1} & (x/\alpha-1)^{-2} & (x/\alpha-1)^{-3} & \cdots & (x/\alpha-1)^{-e} \\
    0 & (x/\alpha-1)^{-1} & (x/\alpha-1)^{-2} & \cdots & (x/\alpha-1)^{-(e-1)} \\
    \vdots & \vdots & \ddots & \cdots & \vdots \\
    0 & \cdots & 0 & (x/\alpha-1)^{-1} & (x/\alpha-1)^{-2} \\
    0 & 0 & \cdots & 0 & (x/\alpha-1)^{-1} \end{array}\right)=\\
    =\alpha^{-e}(x-\alpha)^e\left(\begin{array}{ccccc}
    \frac{\alpha}{x-\alpha} & \frac{\alpha^2}{(x-\alpha)^2} & \frac{\alpha^3}{(x-\alpha)^3} & \cdots & \frac{\alpha^e}{(x-\alpha)^e} \\
    0 & \frac{\alpha}{x-\alpha} & \frac{\alpha^2}{(x-\alpha)^2} & \cdots & \frac{\alpha^{e-1}}{(x-\alpha)^{e-1}} \\
    \vdots & \vdots & \ddots & \cdots & \vdots \\
    0 & \cdots & 0 & \frac{\alpha}{x-\alpha} & \frac{\alpha^2}{(x-\alpha)^2} \\
    0 & 0 & \cdots & 0 & \frac{\alpha}{x-\alpha} \end{array}\right)=\\
    =\alpha^{-e+1}(x-\alpha)^e\left(\begin{array}{ccccc}
    \frac{1}{x-\alpha} & \frac{\alpha}{(x-\alpha)^2} & \frac{\alpha^2}{(x-\alpha)^3} & \cdots & \frac{\alpha^{e-1}}{(x-\alpha)^e} \\
    0 & \frac{1}{x-\alpha} & \frac{\alpha}{(x-\alpha)^2} & \cdots & \frac{\alpha^{e-2}}{(x-\alpha)^{e-1}} \\
    \vdots & \vdots & \ddots & \cdots & \vdots \\
    0 & \cdots & 0 & \frac{1}{x-\alpha} & \frac{\alpha}{(x-\alpha)^2} \\
    0 & 0 & \cdots & 0 & \frac{1}{x-\alpha} \end{array}\right).
\end{multline*}    
Plugging this back into the last equation we get the result.
\end{proof}

\begin{lem}\label{lem_sp_jordan_block_adj}
For $e>0$ integer and $\alpha\in\field_q$, denote by $X_\alpha^e$ the matrix $\Adj(x-J_e(\alpha))$ as computed in Lemma \ref{lem_gl_jordan_block_adj}, and let $Y_\alpha^e,Z_\alpha^{2e}$ be as defined in Proposition \ref{prop_matrix_comp_sp2n}. Finally, let $a$ be as defined in the definition of Type III matrices. Then

\begin{enumerate}
    \item $\Adj(x-A_{\alpha,\alpha^{-1}}^m)=\left(
    \begin{array}{cc}
        (x-\alpha)^m Y_{\alpha}^m &  \\
         & (x-\alpha^{-1})^m (X_{\alpha}^m)^t
    \end{array}
    \right).$

    \item $\Adj(x-A_{\alpha}^{2m+1})=\left(
    \begin{array}{cc}
        (x-\alpha)^{2m+1} Y_{\alpha}^{2m+1} &  \\
         & (x-\alpha)^{2m+1}(X_{\alpha}^{2m+1})^t
    \end{array}
    \right).$

    \item $\Adj(x-A_{\alpha,\pm}^{2m})=\left(
    \begin{array}{cc}
        (x-\alpha)^mY_\alpha^m & 0 \\
         aZ_\alpha^{2m} & (x-\alpha)^m(X_\alpha^m)^t
    \end{array}
    \right)$.
    
\end{enumerate}
\end{lem}

\begin{proof}

Using Proposition \ref{prop_matrix_comp_sp2n} assertions 1,2 are simply an application of the formula $\Adj(M)=\det(M)M^{-1}$. As for the last assertion, it is enough to compute $(x-A_{\alpha,\pm}^{2m})^{-1}$. Let $S$ be as in the definition of Type III matrices. Let $\tilde{S}$ be the block matrix $\left(\begin{array}{cc}
    0 & 0 \\
    S & 0
\end{array}\right)$. Write $x-A_{\alpha,\pm}^{2m}=R-a\tilde{S}$ ($R$ is defined implicitly via this equation). We note that

$$
R^{-1}a\tilde{S}=\left(
\begin{array}{cc}
    *  & 0 \\
    0 & *
\end{array}
\right)\cdot 
\left(
\begin{array}{cc}
    0  & 0 \\
    * & 0
\end{array}
\right)
=\left(
\begin{array}{cc}
    0  & 0 \\
    * & 0
\end{array}
\right)
$$
as a block matrix (here $*$ denotes some block matrices we do not write explicitly, to avoid unnecessary complications). Further, by the structure of $R^{-1}a\tilde{S}$, we see that $(R^{-1}a\tilde{S})^2=0$. Hence,

$$
(1+R^{-1}a\tilde{S})R^{-1}\cdot (R-a\tilde{S})=(1+R^{-1}a\tilde{S})(1-R^{-1}a\tilde{S})=1-R^{-1}a\tilde{S}+R^{-1}a\tilde{S}-(R^{-1}a\tilde{S})^2=1.
$$
We get that $(x-A_{\alpha,\pm}^{2m})^{-1}=(1+aR^{-1}\tilde{S})R^{-1}$. Using Proposition \ref{prop_matrix_comp_sp2n} we finish the proof.
\end{proof}

\subsubsection{For $SO_{n}^{\pm}$}\label{section_explicit_rep_so}

We once again remind the reader that for us, $\field_q$ is a finite field of odd characteristic $p$. For the rest of section \ref{section_explicit_rep_so}, fix some $\epsilon=\pm$. We discussed $O^\pm_n(\field_q)$-conjugacy classes in section \ref{section_so_generalities}. As in the previous subsection, we now focus on conjugacy classes of matrices $M\in O_n^\epsilon(\field_q)$, where $\charc(M)$ splits in $\field_q[x]$. 

The analog of a single Jordan block for $O^\epsilon_n(\field_q)$ is a matrix that is conjugate in $GL_n(\field_q)$ to one of the following types I, II, and III. We also provide conjugacy class representatives from $O^\epsilon_{n}(\field_q)$, following \cite{Tay21a}. 

\begin{enumerate}
    \item \textbf{Type I:} The class with Jordan block $
\left(
\begin{array}{cc}
    J_m(\alpha) &  \\
     & J_m(\alpha^{-1})
\end{array}
\right)
$ for $\alpha\neq \pm 1$. For Witt type $\langle 0 \rangle$, a representative for this class is the matrix $A_{\alpha,\alpha^{-1}}^m:=\left(
\begin{array}{cc}
    J_m(\alpha) &  \\
     & J_m(\alpha)^{-\mathfrak{t}}
\end{array}\right)$. We do not construct an explicit representative for Witt type $\langle 1,\delta \rangle$. Instead in Section \ref{sec_chapoly_dist_ortho} we will use the fact that in an extension field, $\langle 1,\delta\rangle$ becomes congruent to $\langle 0\rangle$. This, together with Lemma \ref{lem_congruent_forms_same_image} below, will enable us to reduce the proof to the Witt type $\langle 0\rangle$ case.

    \item \textbf{Type II:} The class with Jordan block $
    \left(
\begin{array}{cc}
    J_m(\alpha) &  \\
     & J_m(\alpha)
\end{array}
\right)
    $, for $\alpha=\pm 1$, $m$ even. As in Type I, we only consider Witt type $\langle 0\rangle$, for which a representative is $A_\alpha:=\left(
\begin{array}{cc}
    J_m(\alpha) &  \\
     & J_m(\alpha)^{-\mathfrak{t}}
     
\end{array}
\right)$.

    \item \textbf{Type III:} The class with Jordan block
    $J_{2m+1}(\alpha)$, for $\alpha=\pm 1$. Similarly to Type I and Type II, we may treat only the case of Witt type $\langle 1\rangle$ (In Section \ref{sec_chapoly_dist_ortho}, we will reduce the case of Witt type $\langle \delta\rangle$ to the case of Witt type $\langle 1\rangle$ using Lemma \ref{lem_congruent_forms_same_image} below). Define $S$ to be the $m\times m$ matrix with all entries $0$, apart from the last row, whose entries alternate between $-1$ and $1$ (starting from $1$). Also define $u=(0,\ldots,0,1)^t$, $v=(-1,1,-1,\cdots,(-1)^m)$. Then the block matrix
$
A_{\alpha} := \left(
\begin{array}{ccc}
    J_m(\alpha) & u & \frac{(-1)^m}{2}\alpha S \\
    0 & \alpha & v \\
    0 & 0 & J_m(\alpha)^{-\mathfrak{t}}
\end{array}
\right)
$ is a representative of this conjugacy class. The case of Witt type $\langle 1\rangle$ translates to choosing the $+$ sign in the orthogonal signed partition, and the $\langle \delta\rangle$ Witt type case translates to choosing the $-$ sign (see Section \ref{section_so_generalities} for a reminder on the significance of signs). One can check that this matrix is symplectic and that its minimal and characteristic polynomial are equal to $(x-\alpha)^{2m+1}$.
    
\end{enumerate}

Now we want to compute $\Adj(x-M)$ for each type, where $M$ is a matrix of that type. To do that, we need the following

\begin{prop}\label{prop_matrix_comp_son}
Let $e>0$ be an integer and $\alpha=\pm 1$. Using the notation from the description of the Type III case, write $S_0=\left(
\begin{array}{c}
    \frac{(-1)^e}{2}\alpha S \\
     v
\end{array}
\right)$. Define a sequence of rational functions $g_1=\frac{\alpha}{1-x\alpha}$, $g_2=\frac{\alpha}{1-x\alpha}-\frac{1}{(1-x\alpha)^2}$,\ldots,$g_e=\frac{\alpha}{1-x\alpha}-\frac{1}{(1-x\alpha)^2}+\cdots+\frac{(-1)^{e-1}x^{e-2}}{(1-x\alpha)^e}$. Then, 

\begin{multline*}
W_\alpha^{2e+1}:=(x-\alpha)^{2e+1}(x-J_{e+1}(\alpha))^{-1} S_0 (x-J_e(\alpha)^{-1})^{-\mathfrak{t}}=\\
=(x-\alpha)^{2e+1}
\left(
\begin{array}{ccccc}
    f_1(x-\alpha)^{-e}\left( \frac{(-1)^{e}}{2}\alpha-\frac{1}{x-\alpha}\right) &  & \cdots &  & f_{e}(x-\alpha)^{-e}(-1)^{e-1}\left(\frac{(-1)^{e}}{2}\alpha-\frac{1}{x-\alpha}\right) \\
     &  & \ddots &  & \\
    \vdots &  & \cdots  & &  \\
    f_1(x-\alpha)^{-1}\left( \frac{(-1)^{e}}{2}\alpha-\frac{1}{x-\alpha}\right) &  &  & & f_{e}(x-\alpha)^{-1}(-1)^{e-1}\left(\frac{(-1)^{e}}{2}\alpha-\frac{1}{x-\alpha}\right) \\
\frac{f_1}{x-\alpha} &  & \cdots &  & \frac{(-1)^{e+1} f_{e}}{x-\alpha}
    \end{array}
\right).
\end{multline*}
\end{prop}

\begin{proof}
Note that $J_e(\alpha)^{-\mathfrak{t}}=J_e(\alpha)^{-1}$. Hence,
$
W_{\alpha}^{2e+1}=X_\alpha^{e+1}S_0Y_\alpha^e
$. We compute directly using the definition of $Y_\alpha^e$,

$$
S_0Y_\alpha^e=(x-\alpha)^e\left(
\begin{array}{ccccc}
    0 &  & \cdots &  & 0 \\
     &  & \ddots &  & \\
    0 &  & \cdots  & & 0 \\
    \frac{(-1)^{e+1}}{2}\alpha g_1 &  \frac{(-1)^{e+2}}{2}\alpha g_2 &  \cdots &  & \frac{(-1)^{2e}}{2}\alpha g_{e} \\
    g_1 & -g_2 & \cdots & (-1)^{e-2}g_{e-1} & 
    (-1)^{e-1} g_{e} \end{array}
    \right).
$$
Thus, using the last equation and the definition of $X_\alpha^{e+1}$,

\begin{multline*}
W_\alpha^{2e+1}
=\left(
\begin{array}{ccccc}
    g_1(x)(x-\alpha)^{-e}\left( \frac{(-1)^{e+1}}{2}\alpha+\frac{1}{x-\alpha}\right) &  & \cdots &  & g_{e}(x)(x-\alpha)^{-e}(-1)^{e-1}\left(\frac{(-1)^{e+1}}{2}\alpha+\frac{1}{x-\alpha}\right) \\
     &  & \ddots &  & \\
    \vdots &  & \cdots  & &  \\
    g_1(x)(x-\alpha)^{-1}\left( \frac{(-1)^{e+1}}{2}\alpha+\frac{1}{x-\alpha}\right) &  &  & & g_{e}(x)(x-\alpha)^{-1}(-1)^{e-1}\left(\frac{(-1)^{e+1}}{2}\alpha+\frac{1}{x-\alpha}\right) \\
\frac{g_1(x)}{x-\alpha} &  & \cdots &  & \frac{(-1)^{e-1} g_{m}(x)}{x-\alpha}
    \end{array}
\right).
\end{multline*}
\end{proof}

\begin{lem}\label{lem_so_jordan_block_adj}
Using the notation of Lemmas \ref{lem_sp_jordan_block_adj} and \ref{lem_gl_jordan_block_adj},
\begin{enumerate}
    \item $\Adj(x-A_{\alpha,\alpha^{-1}}^m)=\left(
    \begin{array}{cc}
       (x-\alpha^{-1})^mX_\alpha^m  &  \\
         & (x-\alpha)^m (Y_\alpha^m)^\mathfrak{t}
    \end{array}
    \right).$

    \item $\Adj(x-A_\alpha^{2m})=\left(
    \begin{array}{cc}
       (x-\alpha^{-1})^mX_\alpha^m  &  \\
         & (x-\alpha)^m (Y_\alpha^m)^\mathfrak{t}
    \end{array}
    \right)$.

    \item  $\Adj(x-A_\alpha^{2m+1})=\left(
    \begin{array}{cc}
        (x-\alpha^{-1})^{m+1} X_\alpha^{m+1} & -W_\alpha^{2m+1} \\
        0 & (x-\alpha)^m (Y_\alpha^m)^\mathfrak{t}
    \end{array}
    \right)$.
\end{enumerate}
\end{lem}

\begin{proof}
Assertions 1 and 2 follow similarly to the analogous assertions in Lemma \ref{lem_sp_jordan_block_adj}, recalling the definition of the anti-diagonal transpose $M\mapsto M^\mathfrak{t}$ from Section \ref{section_so_generalities}. 

For assertion 3, it is enough to compute $(x-A_\alpha^{2m+1})^{-1}$. Write 
$
x-A_\alpha^{2m+1} = B+C$, where
$$B=\left(
\begin{array}{cc}
    x-J_{m+1}(\alpha) &  \\
     & x-J_m(\alpha)^{-1}
\end{array}
\right), C=\left(
\begin{array}{cc}
    0 & -S_0 \\
    0 & 0
\end{array}
\right),
$$
and $S_0$ is defined as in Proposition \ref{prop_matrix_comp_son}. We note that $B^{-1}C$ is of the form $\left(
\begin{array}{cc}
    0 & * \\
    0 & 0
\end{array}
\right)$ and so $(B^{-1}C)^2=0$. Thus

$$
(1-B^{-1}C)B^{-1}(B+C)=(1-B^{-1}C)(1+B^{-1}C)=1-(B^{-1}C)^2=1,
$$
so that $(x-A_\alpha)^{-1}=(1-B^{-1}C)B^{-1}$. The claim now follows using Proposition \ref{prop_matrix_comp_son}.
\end{proof}

\section{The typical degree of the minimal polynomial}\label{section_min_poly_deg}

For the various groups we consider, we prove below equidistribution for the traces of powers or the next-to-leading coefficients of the characteristic polynomial, whenever the minimal polynomial modulo $\pi$, $\min(\overline{A_0})$, is of sufficiently high degree. We shall prove in this section that this is the typical case, in a precise sense. Our methods rely on the formulae for the distribution of matrices in conjugacy classes for the finite classical matrix groups, due to Fulman \cite{Ful00}.

\subsection{For $GL_n, SL_n$}

We examine the probability that a matrix $M$ has an exceptionally small minimal polynomial. 

\begin{claim}\label{claim_bounding_probability_small_minpoly}
Let $M\in GL_n(\field_q)$ be chosen uniformly at random, and let $h\in\field_q[x]_{=n}^\mon$. Then, for $0\le\delta<n$ integer,
    $$\Pr[\deg\min(M)=n-\delta\cap\charc(M)=h]<q^{-\frac{n^2}{n-\delta}+o(n)}.$$
The rate of decay is independent of $\delta$ (but may depend on $q$).
\end{claim}

\begin{proof}
Factor $h=\charc(M)=\prod_{\phi\in\Pee} \phi^{e_\phi}$. Now let $\min(M)=\prod_{\phi\in\Pee} \phi^{d_\phi}$. By \eqref{eqn_fulman_conj_classes_prob}, 
\begin{multline*}
\Pr[\deg\min(M)<n-\delta\cap\charc(M)=h]=\sum_{\substack{\lambda_\phi \vdash e_\phi \\ \max \lambda_\phi=d_\phi}} \frac{1}{\prod_\phi q^{\deg\phi\cdot \sum (\lambda_{\phi,i}^{'})^2}\prod (\frac{1}{q^{\deg\phi}})_{m_i(\lambda_\phi)}}\ll \\
\ll  \sum_{\substack{\lambda_\phi \vdash e_\phi \\ \max \lambda_\phi=d_\phi}} \frac{q^{o(n)}}{\prod_\phi q^{\deg\phi\cdot \sum (\lambda_{\phi,i}^{'})^2}}.
\end{multline*}

The proof for why $\prod (\frac{1}{q^{\deg\phi}})_{m_i(\lambda_\phi)}$ is $q^{o(n)}$ is essentially the same as the proof of Lemma \ref{lem_small_radical_bound}, hence is omitted. The number of partitions of $n$ is $q^{o(n)}$, so we just need to bound from above the summand
$
\frac{1}{\prod_\phi q^{\deg\phi\cdot \sum (\lambda_{\phi,i}^{'})^2}}.
$ Note that since $\lambda_\phi$ is a partition of $e_\phi$ with maximal part $d_\phi$, $\lambda'_\phi$ is a partition of $e_\phi$ with exactly $d_\phi$ parts. To minimize the sum $\sum (\lambda_{\phi,i}^{'})^2$ we thus take each part to be of the same size $\frac{e_\phi}{d_\phi}$. This gives 

$$
\frac{1}{\prod_\phi q^{\deg\phi\cdot \sum (\lambda_{\phi,i}^{'})^2}}<\frac{1}{\prod_\phi q^{\deg\phi\cdot d_\phi\cdot \frac{e_\phi^2}{d_\phi^2}}}=\frac{1}{\prod_\phi q^{e_\phi\deg\phi\cdot \frac{e_\phi}{d_\phi}}}.
$$

We thus want to minimize $\sum e_\phi\deg\phi \frac{e_\phi}{d_\phi}$. Define a probability measure on $\Pee$ by setting $\Pr[\phi]=\frac{e_\phi\deg\phi}{n}$. We thus want to minimize $\mathbb{E}[\frac{e_\phi}{d_\phi}]$. Note that 
$$
\mathbb{E}\left[\frac{d_\phi}{e_\phi}\right]=\frac{\deg(\min(M))}{n}<\frac{n-\delta}{n}.
$$
Now by Jensen's inequality, we get 
$
\mathbb{E}\left[\frac{e_\phi}{d_\phi}\right]<\frac{n}{n-\delta},
$
which concludes the proof.
\end{proof}

\begin{rem}
One could improve the bound a little by using Holder's defect formula (see \cite[\S 6]{Ste04}). Also, for some special polynomials (e.g. squarefree) one can improve the results by other methods. However, since the finite field results bound the ranges in our main theorems, this does not improve our main result, hence we did not pursue this.
\end{rem}

\begin{rem}
Despite the previous remark, the range in Claim \ref{claim_bounding_probability_small_minpoly} is not too far from optimal. One can see that by taking $h=(x-1)^n$, and then approximating the probability to get a conjugacy class which contains $\lfloor\frac{n}{n-\delta}\rfloor$ Jordan blocks $J_{n-\delta}(1)$. 
\end{rem}

We get as a corollary 

\begin{claim}
    Let $M\in SL_n(\field_q)$ be chosen uniformly at random, and let $h\in\field_q[x]_{=n}^\mon$ be a polynomial with $h(0)=(-1)^n$. Then, for $0\le\delta<n$ integer,
    $$\Pr[\deg\min(M)=n-\delta\cap\charc(M)=h]< q^{-\frac{n^2}{n-\delta}+o(n)}.$$
\end{claim}

\begin{proof}
Note that $\charc(M)(0)=(-1)^n$ is equivalent to $M\in SL_n(\field_q)$. Thus if we combine the last claim with the fact that $[GL_n(\field_q):SL_n(\field_q)]=q-1$, we finish.
\end{proof}

\subsection{For $Sp_{2n}$}

We want to understand the probability that a matrix $M\in Sp_{2n}(\field_q)$, chosen uniformly at random, has an exceptionally small minimal polynomial. To do this, we first recall some results of Fulman \cite{Ful00} on the distribution of $Sp_{2n}(\field_q)$-conjugacy classes. In the current section, we use the notation introduced and discussed in section \ref{subsubsection_symplectic_group_intro}.

The collection of polynomials arising as the characteristic polynomial of a matrix in $Sp_{2n}(\field_q)$ for some $n$ is the set of monic, even-degree self-reciprocal polynomials with $f(0)\neq 0$. Equivalently, one can show that it is the set of polynomials that factor into irreducibles as follows,

$$
f=(x-1)^{2a}(x+1)^{2b}\prod_{i=1}^r P_i^{e_i}\prod_{j=1}^s (Q_jQ_j^\rs)^s,
$$
where $P_i^\rs=P_i$ and $P_i,Q_i\neq x,x\pm 1$ are distinct irreducible polynomials. 

\begin{theorem}[\cite{Ful00}, Theorem 1]\label{thm_fulman_sp}
Let $\lambda=(\lambda_\phi)_{\phi\in\Pee}$ be a signed symplectic partition with $|\lambda|=2n$. Let $C$ be the conjugacy class of $Sp_{2n}(\field_q)$ parameterized by $\lambda$. Then,

\begin{equation*}
\Pr[M\in C]=\prod_{\phi\in\Pee}\frac{1}{q^{\deg\phi(\sum_{h<i}hm_h(\lambda_\phi)m_i(\lambda_\phi)+\frac{1}{2}\sum_i (i-1)m_i(\lambda_\phi)^2)}\prod_i A(\phi^i)}  
\end{equation*}
Where here, for $\phi=x\pm 1$
$$
A(\phi^{i})=\begin{cases}
    |Sp_{m_i(\lambda_\phi)}(\field_q)|, & \text{ if }i=1\pmod{2}, \\
    q^{m_i(\lambda_\phi)/2}|O_{m_i(\lambda_\phi)}^{\mu(i)}(\field_q)|, & \text{ else.}
\end{cases}
$$
Here $\mu(i)$ is the sign attached to the $i$-th part of $\phi$ by $\lambda$ (see section \ref{subsubsection_symplectic_group_intro}). For $\phi\neq x\pm 1$,
 
$$
A(\phi^i)=\begin{cases}
    |U_{m_i(\lambda_\phi)}(\field_{q^{\deg\phi/2}})|, & \text{ if }\phi=\phi^\rs, \\
    |GL_{m_i(\lambda_\phi)}(\field_{q^{\deg\phi}})|^{1/2}, & \text{ else.}
\end{cases}
$$
\end{theorem}

It will be useful to reformulate Fulman's result. To do that, we first estimate the size of $\frac{A(\phi^{i})}{q^{\frac{\deg\phi}{2}m_i(\lambda_\phi)^2}}$.

\begin{prop}\label{prop_A_phi_bound}
Let $\phi\in\Pee$. Then,

\begin{equation*}
\frac{A(\phi^{i})}{q^{\frac{\deg\phi}{2}m_i(\lambda_\phi)^2}}=
\begin{cases}
\Theta\left(q^{\frac{1}{2}m_i(\lambda_\phi)}\right), & \text{ if } \phi=x\pm 1 \text{ and } i=1\pmod{2}, \\
\Theta(1), & \text{ else.}
\end{cases}    
\end{equation*}
\end{prop}

\begin{proof}
We split to cases depending on whether $\phi=x\pm 1$ or not.

\maincase{$\phi\neq x\pm 1$} By \cite[\S 2.6]{Wal63} we have

$$
|U_{m_i(\lambda_\phi)}(\field_{q^{\deg\phi/2}})|=\Theta(q^{\frac{\deg\phi}{2}m_i(\lambda_\phi)^2}), |GL_{m_i(\lambda_\phi)}(\field_{q^{\deg\phi}})|
^{1/2}=\Theta(q^{\frac{\deg\phi}{2}m_i(\lambda_\phi)^2}),$$
where the constant implicit in the $\Theta$ notation is uniformly bounded in both $m_i,q$. Hence the result follows.

\maincase{$\phi=x\pm 1$} Again by \cite[\S 2.6]{Wal63}, 
$$
|Sp_{m_i(\lambda_\phi)}(\field_q)|=\Theta(q^{\frac{m_i(\lambda_\phi)^2+m_i(\lambda_\phi)}{2}}),|O_{m_i(\lambda_\phi^\pm)}^{\mu(i)}(\field_q)|=\Theta(q^{\frac{m_i(\lambda_\phi)^2-m_i(\lambda_\phi)}{2}}).
$$
where the constant implicit in the $\Theta$ notation is uniformly bounded in both $m_i,q$. Now, the result follows from the definition of $A(\phi^i)$.

\end{proof}

\begin{corollary} \label{cor_bound_probability_conjugacy_class}
Let $\lambda=(\lambda_\phi)_{\phi\in\Pee}$ be a signed symplectic partition with $|\lambda|=2n$. Let $C$ be the conjugacy class of $Sp_{2n}(\field_q)$ parameterized by $\lambda$. Then,

\begin{equation}
\Pr[M\in C]\ll q^{o(n)}\prod_{\phi\in\Pee}\frac{1}{q^{\frac{\deg\phi}{2}\sum_i (\lambda_{\phi,i}^{'})^2}}.
\end{equation}
\end{corollary}

\begin{proof}
    Using \cite[Lemma 1]{Ful00} we get that for any partition $\lambda_\phi$, we have

\begin{equation}\label{eqn_partition_eqn}
\sum_{h<i}hm_h(\lambda_\phi)m_i(\lambda_\phi)+\frac{1}{2}\sum_i (i-1)m_i(\lambda_\phi)^2=\frac{1}{2}\sum_i((\lambda_\phi)_i'^2-m_i(\lambda_\phi)^2).
\end{equation}
Combining this with the previous proposition and applying it to the formula in Fulman's Theorem \ref{thm_fulman_sp} we get the desired result.
\end{proof}

We can now prove a result on the typical degree of the minimal polynomial. 

\begin{theorem}\label{theorem_probability_small_min_poly_sp2n}
Let $f\in\field_q[x]_{=2n}^\mon$ be a self-reciprocal polynomial, and assume that $x\nmid f$. Then, for $0\le\delta<n$ integer,
$$
\Pr[\deg\min(M)=n-\delta\cap \charc(M)=f]\ll q^{-\frac{n^2}{2(n-\delta)}+o(n)}.
$$
\end{theorem}

\begin{proof}
Factor $f=\charc(M)=\prod_{\phi\in\Pee}\phi^{e_\phi}$, $\min(M)=\prod_{\phi\in\Pee}\phi^{d_\phi}$. By Corollary \ref{cor_bound_probability_conjugacy_class}, we have

\begin{equation*}
\Pr[\deg\min(M)=n-\delta\cap \charc(M)=f]\ll q^{o(n)}\sum_{
\substack{\lambda_\phi\vdash e_\phi \\ \max \lambda_\phi=d_\phi}
}
\frac{1}{\prod_\phi q^{\frac{\deg\phi}{2}\sum(\lambda_{\phi,i}^{'})^2}},
\end{equation*}
where the sum is over signed symplectic partitions. The total number of partitions $\lambda=(\lambda_\phi)$ is $q^{o(n)}$, so we just need to bound from above the summand 
$
\frac{1}{\prod_\phi p^{\frac{\deg\phi}{2}\sum(\lambda_{\phi,i}^{'})^2}}$.We finish as in the proof of Claim \ref{claim_bounding_probability_small_minpoly}.
\end{proof}
 
\subsection{For $SO_n^{\pm}$}
Let $\epsilon=\pm$ be fixed and work in $O_n^\epsilon(\field_q)$, $SO_n^\epsilon(\field_q)$. In the present section, we use the notation introduced and discussed in section \ref{section_so_generalities}. Similarly to the symplectic case, the characteristic polynomial $f$ of a matrix $M\in O_n^\epsilon(\field_q)$ factors as follows,

$$
f=(x-1)^a(x+1)^b\prod_{j=1}^r P_i^{e_i} \prod_{j=1}^s (Q_j Q_j^\rs)^s,
$$
where $P_i^\rs=P_i$ and $P_i$, $Q_i\neq x,x\pm 1$ are distinct irreducible polynomials.

\begin{theorem}[\cite{Ful00}, Theorem 2]
Let $\lambda=(\lambda_\phi)_{\phi\in\Pee}$ be a signed orthogonal partition with $|\lambda|=n$. Let $C$ be the conjugacy class of $O^\epsilon_n(\field_q)$ parametrized by $\lambda$. Then,
\begin{equation*}
\Pr[M\in C]<\prod_{\phi\in\Pee}\frac{1}{q^{\deg\phi(\sum_{h<i}hm_h(\lambda_\phi)m_i(\lambda_\phi)+\frac{1}{2}\sum_i (i-1)m_i(\lambda_\phi)^2)}\prod_i B(\phi^i)}.
\end{equation*}    
Here, for $\phi=x\pm 1$,
$$
B(\phi^{i}):=\begin{cases}
    q^{-m_i(\lambda_\phi)/2}|Sp_{m_i(\lambda_\phi)}(\field_q)|, & \text{ if }i=1\pmod{2}, \\
    |O_{m_i(\lambda_\phi)}^{\mu(i)}(\field_q)|, & \text{ else.}
\end{cases}
$$
Where $\mu(i)$ is the sign attached to the $i$-th part of $\lambda_\phi$ (see section \ref{section_so_generalities}). For $\phi\neq x\pm 1$,
$$
B(\phi^i)=\begin{cases}
    |U_{m_i(\lambda_\phi)}(\field_{q^{\deg\phi/2}})|, & \text{ if }\phi=\phi^\rs, \\
    |GL_{m_i(\lambda_\phi)}(\field_{q^{\deg\phi}})|^{1/2}, & \text{ else.}
\end{cases}
$$
\end{theorem}

\begin{prop}
Let $\phi\in\Pee$. Then,

\begin{equation*}
\frac{B(\phi^i)}{q^{\frac{\deg\phi}{2}m_i(\lambda_\phi)^2}}=\begin{cases}
    \Theta\left(
    q^{-\frac{1}{2}m_i(\lambda_\phi)}
    \right), & \text{ if } \phi=x\pm 1 \text{ and }i=0\pmod{2}, \\
    \Theta(1), & \text{ else.}
\end{cases}    
\end{equation*}

\end{prop}
\begin{proof}
We note that for $\phi\neq x\pm 1$, $B(\phi^i)=A(\phi^i)$, so that the result follows from proposition \ref{prop_A_phi_bound}. For $\phi=x\pm 1$, following the proof of proposition \ref{prop_A_phi_bound} we see that if $i$ is odd,
$
B(\phi^{i})=\Theta(1)
$, and else it is $\Theta(q^{-\frac{m_i(\lambda_\phi)}{2}})$.
\end{proof}

\begin{corollary} 
Let $\lambda=(\lambda_\phi)_{\phi\in\Pee}$ be a signed orthogonal partition with $|\lambda|=n$. Let $C$ be the conjugacy class of $O_{n}^\epsilon(\field_q)$ parameterized by $\lambda$. Then,

\begin{equation}
\Pr_{M\in O^\epsilon_n(\field_q)}[M\in C]\ll q^{n/2+o(n)}
\prod_{\phi\in\Pee}\frac{1}{q^{\frac{\deg\phi}{2}\sum_i (\lambda_{\phi,i}^{'})^2}}.
\end{equation}
\end{corollary}

\begin{proof}
The proof is almost the same as in the symplectic case, only now we have to bound the sum $\sum_{i\text{ is even}}\frac{m_{2i}(\lambda_{\phi})}{2}$ for $\phi=x\pm 1$. We note that if the total weight of the even parts is $k$, then the sum of multiplicities of the even parts is at most $k/2$. Since $k<n$, we get that $\sum_{i\text{ is even}}\frac{m_{2i}(\lambda_{\phi})}{2}<\frac{n}{4}$. Since there are two possible choices for $\phi$ we get the result.
\end{proof}

\begin{theorem}\label{theorem_probability_small_min_poly_son}
Let $f\in\field_q[x]_{=n}^\mon$ be a self-reciprocal polynomial, and assume that $x\nmid f$. Then, for $0\le\delta<n$ integer,
$$
\Pr_{M\in SO^\epsilon_n(\field_q)}[\deg\min(M)=n-\delta\cap \charc(M)=h]\ll q^{\frac{n}{2}-\frac{n^2}{2(n-\delta)}+o(n)}.
$$
\end{theorem}

\begin{proof}
The bound for $O^\epsilon_n(\field_q)$ is essentially the same as in the symplectic case, only now we use the bound from the last corollary. To go from $O_n^\epsilon(\field_q)$ to $SO^\epsilon_n(\field_q)$, we observe that for $M\in O_n^\epsilon(\field_q)$, $\charc(M)(0)=(-1)^n$ is equivalent to $M\in SO_n^\epsilon(\field_q)$. Hence since $[O_n^\epsilon(\field_q):O_n^\epsilon(\field_q)]=2$ we prove the theorem.
\end{proof}

\subsection{For $U_n$}
The characteristic polynomial $f$ of a matrix $M\in U_n(\field_{q})$ factors as follows,
$$
f=\prod_{i=1}^r P_i^{e_i}\prod_{j=1}^s (Q_j\widetilde{Q_j})^{d_j},
$$
where $\widetilde{P_i}=P_i$ and $P_i,Q_i$ are distinct irreducible polynomials in $\field_{q^2}[x]$. In the following section, we use the notation introduced and discussed in section \ref{sec_unitary_generalities}. 

\begin{theorem}[\cite{Ful99}, \S3]
Let $\lambda=(\lambda_\phi)_{\phi\in\Pee}$ be a unitary partition with $|\lambda|=n$. Let $C$ be the conjugacy class of $U_{n}(\field_q)$ parameterized by $\lambda$. Then,

\begin{equation*}
\Pr[M\in C]=
\prod_{\phi}\frac{1}{q^{2\deg\phi(\sum_{h<i}hm_h(\lambda_\phi)m_i(\lambda_\phi)+\frac{1}{2}\sum_i (i-1)m_i(\lambda_\phi)^2)}\prod_i C(\phi^i)},
\end{equation*}
Where here,
$$
C(\phi^{i})=\begin{cases}
    |U_{m_i(\lambda_\phi)}(\field_{q^{\deg\phi}})|, & \text{ if }\widetilde{\phi}=\phi, \\
    |GL_{m_i(\lambda_\phi)}(\field_{q^{2\deg\phi}})|^{1/2}, & \text{ else.}
\end{cases}
$$
\end{theorem}

\begin{corollary}
Let $\lambda=(\lambda_\phi)_{\phi\in\Pee}$ be a unitary partition with $|\lambda|=n$. Let $C$ be the conjugacy class of $U_n(\field_q)$ parametrized by $\lambda$. Then,
$$
\Pr[M\in C]\ll q^{o(n)}\prod_\phi q^{-\deg\phi\sum_i(\lambda_{\phi,i}')^2}.
$$
\end{corollary}

\begin{proof}
We recall that $|U_{m_i(\lambda_\phi)}(\field_{q^{\deg\phi}})|=\Theta(q^{\deg\phi m_i(\lambda_\phi)^2})$, and also $|GL_{m_i(\lambda_\phi)}(\field_{q^{2\deg\phi}})|=\Theta(q^{2\deg\phi\cdot m_i(\lambda_\phi)^2})$. Thus, $C(\phi^i)=\Theta(q^{\deg\phi\cdot m_i(\lambda_\phi)^2}).$ Combining this with equation \eqref{eqn_partition_eqn}, we get the result.
\end{proof}

Now using the previous corollary and a similar proof to that of claim \ref{claim_bounding_probability_small_minpoly}, we get

\begin{claim}\label{claim_bounding_probability_small_minpoly_un}
    Let $M\in U_n(\field_q)$ be chosen uniformly at random, and let $h\in\field_{q^2}[x]_{=n}^\mon$ be a self-skew-reciprocal polynomial. Then, for $0\le\delta<n$ integer,
    $$\Pr[\deg\min(M)=n-\delta\cap\charc(M)=h]<q^{-\frac{n^2}{n-\delta}+o(n)}.$$
\end{claim}

\section{Traces and characteristic polynomial for $GL_n$}\label{section_gln_main}

For the rest of Section \ref{section_gln_main}, $A$ will be a matrix from $GL_n(\resring{k})$ for some $k\ge 1$, unless explicitly stated otherwise.

\subsection{The image of $\partial\charc_{A_0}$}\label{section_charpoly_derivative_gln}

Recall that we defined the map $\partial\charc_{A_0}:\mathfrak{gl}_n(\field_q)\to\field_q[x]_{<n}$ by $A_1\mapsto \tr(\Adj(x-A_0)A_0A_1)$. We note that in the case of $G=GL_n$, since $A_1\mapsto A_0A_1$ is a permutation on $\mathfrak{gl}_n$, $\img(\partial\charc_{A_0})$ is equal to the image of the map $A_1\mapsto \tr(\Adj(x-A_0)A_1)$. Hence, as explained in Section \ref{subsection_explicit_representatives}, one can move from $A_0$ to any element in its conjugacy class without changing $\img(\partial\charc_{A_0})$. To calculate the image, we will need the following preparatory lemma:

\begin{lem}\label{lemma_derivative_of_direct_sums}
Let $A\in GL_n(\field_q)$, $B\in GL_m(\field_q)$ be two matrices. Then,

$$
\img(\partial\charc_{A\oplus B})=\mathrm{span}_{\field_q}(\charc(B)\cdot \img(\partial\charc_{A}),\charc(A)\cdot \img(\partial\charc_{B})).
$$
\end{lem}

\begin{proof}
We note that $\mathfrak{gl}_{n+m}(\field_q)=M_{n+m}(\field_q)$. Write an element $A_1\in M_{n+m}(\field_q)$ as a block matrix

$$
A_1=\left(
\begin{array}{cc}
    X & Y \\
    Z & W
\end{array}
\right).
$$
Then,

$$\partial\charc_{A\oplus B}(A_1)=\tr\left(
\Adj\left(
\begin{array}{cc}
    x-A & 0 \\
    0 & x-B
\end{array}
\right)\left(
\begin{array}{cc}
    X & Y \\
    Z & W
\end{array}
\right)
\right).$$
Note that the entries of $Y,Z$ do not contribute to the result. Thus we may assume freely that they are zero. We get that

\begin{multline*}
\partial\charc_{A\oplus B}(A_1)=\tr\left(
\Adj\left(
\begin{array}{cc}
    x-A & 0 \\
    0 & x-B
\end{array}
\right)\left(
\begin{array}{cc}
    X & 0 \\
    0 & W
\end{array}
\right)
\right)=\\ 
=\tr\left(
\Adj\left(
\begin{array}{cc}
    \charc(B)\Adj(x-A) & 0 \\
    0 & \charc(A)\Adj(x-B)
\end{array}
\right)\left(
\begin{array}{cc}
    X & 0 \\
    0 & W
\end{array}
\right)
\right)=\\ =\charc(B)\tr(\Adj(x-A)X)+\charc(A)\tr(\Adj(x-B)W),
\end{multline*}
which finishes the proof.
\end{proof}

\begin{theorem}\label{theorem_image_char_derivative_gl_n}
Let $A\in GL_n(\field_q)$ be a matrix. Then, $\img(\partial\charc_A)=\left\{\frac{\charc A}{\min A}a:a\in\field_q[x]_{<\deg\min A}\right\}$.
\end{theorem}

\begin{proof}
It's enough to prove the claim over an extension field $\field_{q^l}$ over which $\charc A$ splits completely. Indeed, assuming this, the image is a linear space contained in the RHS. If there is equality over an extension field, then $\mathrm{rank}(\partial\charc_A)=\deg\min A$ over an extension field. But the rank is invariant to field extensions, thus we get the result. 

Further, we claim that it is enough to prove the claim for the case of $A$ primary (that is, $\charc A=(x-\alpha)^n$). This can be shown by induction on the number of different primary parts in $A$. if $A=B\oplus C$, where $C$ is primary and $\gcd(\charc B,\charc C)=1$, then by Lemma \ref{lemma_derivative_of_direct_sums} and the induction hypothesis we get 

$$
\img(\partial\charc_{B\oplus C})=\mathrm{span}_{\field_{q^l}}\left\{
\charc C\left\{\frac{\charc B}{\min B}b:\deg b<\deg\min B\right\},
\charc B\left\{\frac{\charc C}{\min C}c:\deg c<\deg\min C\right\}
\right\},
$$
and this proves the result using the Chinese remainder theorem for $\field_{q^l}[x]$ and comparing dimensions. 

It remains to prove the claim for a primary matrix. Let $A$ be such that $\charc A=(x-\alpha)^n$. Write $A$ as a direct sum of $r$ matrices $A_i$, each with $\charc A_i=\min A_i=(x-\alpha)^{m_i}$, and such that $m_1\ge m_2\ge \cdots \ge m_r$. We again prove the claim by induction on the number of parts $r$. Write $A=B\oplus A_r$. Then,

$$
\img(\partial\charc_{B\oplus A_r})=\mathrm{span}_{\field_{q^l}}\left\{
(x-\alpha)^{m_r}\left\{\frac{(x-\alpha)^{n-m_r}}{(x-\alpha)^{m_1}}b:\deg b<m_1\right\},
(x-\alpha)^{n-m_r}\left\{c:\deg c<\deg m_r\right\}
\right\}.
$$
which proves the result. Thus, we are only left with the case of a matrix $A_0$ with $\charc A_0=\min A_0=(x-\alpha)^r$. We note that since conjugation from elements in $GL_n(\field_{q^l})$ preserves the Lie algebra $\mathfrak{gl}_n(\field_{q^l})$, we can study the image of $A_1\mapsto \tr(\Adj(x-A_0)gA_1g^{-1})$ instead, for any $g\in GL_n(\field_{q^l})$. The trace is cyclic and by the properties of the adjoint matrix, this is equal to $\tr(\Adj(x-gA_0g^{-1})A_1)$. Thus without loss of generality, we may assume that $A_0$ is given in Jordan form. But then we know that it is a Jordan block, and we have by Lemma \ref{lem_gl_jordan_block_adj},

$$\Adj(x-A_0)=\left(\begin{array}{ccccc}
    (x-\alpha)^{r-1} & (x-\alpha)^{r-2} & (x-\alpha)^{r-3} & \cdots & 1 \\
    0 & (x-\alpha)^{r-1} & (x-\alpha)^{r-2} & \cdots & x-\alpha \\
    \vdots & \vdots & \ddots & \cdots & \vdots \\
    0 & \cdots & 0 & (x-\alpha)^{r-1} & (x-\alpha)^{r-2} \\
    0 & 0 & \cdots & 0 & (x-\alpha)^{r-1}
  \end{array}\right).$$ 
Since $\mathfrak{gl}_n(\field_{q^l})=M_n(\field_{q^l})$, one can see that the claim follows for such $A_0$, and thus we finish the proof.
\end{proof}

\subsection{Distribution of a single trace}

Our goal is to bound the total variation distance $d_{TV}(\mu_{\tr(A^r)},\mu_U)$ for $A\in GL_n(\Ol)$ chosen uniformly at random w.r.t. the Haar probability measure. By Lemma \ref{lem_total_variation_dist_limit}, it is enough to bound $d_{TV}(\mu_{\tr(\rho_k(A^r))},\mu_{\rho_k(U)})$, and show that it remains $o(1)$ as $n\to\infty$, uniformly in $k$. In this part, we assume $p\nmid r$.

\begin{theorem}
Let $U$ be the uniform probability measure on $\resring{k}$, $A$ chosen u.a.r. from $GL_n(\resring{k})$. Let $r=r(n)$ be such that $\log r=o(n^2)$, $p\nmid r$. Then,
$$
d_{TV}(\mu_{\tr(A^r)},\mu_U)\to_{n\to\infty} 0,
$$
uniformly in $k$. 
\end{theorem}

\begin{proof}
Let $A=A_0+\pi^{k-1} A_0A_1$, where $A_0$ is chosen from a set of representatives for $GL_n(\resring{k-1})$, $A_1\in \mathfrak{gl}_n(\field_q)$. From Lemma \ref{lem_trace_derivative}, we have $\tr(A^r)=\tr(A_0^r)+r\pi^{k-1}\tr(A_0^r A_1)$. Recall that we denoted the linear functional $A_1\mapsto \tr(A_0^r A_1)$ by $\partial \tr_{\overline{A_0}}$. This linear functional is either identically zero or equidistributed. Note that for it to be identically zero we must have $\overline{A_0}^r$ in the orthogonal space to $\mathfrak{gl}_n(\field_q)$ under the bilinear form  $\tr$. But $\mathfrak{gl}_n(\field_q)=M_n(\field_q)$, and $\overline{A_0}\in GL_n(\field_q)$, thus $\partial\tr_{\overline{A_0}}$ always equidsitributes. By recursion, we get that for any $x\in \resring{k}$, 

$$
\Pr[\tr(A^r)=x]=\Pr[\tr(\overline{A}^r)=\overline{x}]\prod_{i=2}^k \Pr[\tr(\rho_i(A)^r)=\rho_i(x)|\tr(\rho_{i-1}(A)^r)=\tr(\rho_{i-1}(x))]=q^{-(k-1)}\Pr[\tr(\overline{A}^r)=\overline{x}].
$$

By \cite{GK23}, for any $x\in\field_q$ we have

$$
|\Pr[\tr(\overline{A}^r)=\overline{x}]-q^{-1}|\to_{n\to\infty} 0,
$$
independently of $k$. Thus, we get the desired result on the total variation distance, by summing the error terms and using the triangle inequality.
\end{proof}

As a corollary, we get using Lemma \ref{lem_total_variation_dist_limit}.

\begin{theorem}
Let $U$ be the Haar probability measure on $\Ol$, and let $r=r(n)$ be such that $\log r=o(n^2)$, and $p\nmid r$. Let $A$ be chosen from $GL_n(\Ol)$, uniformly at random w.r.t. the Haar measure. Then, 

$$
d_{TV}(\mu_{\tr(A^r)},\mu_U)\to_{n\to\infty} 0.
$$
\end{theorem}

\begin{rem}
Modifications of this argument can prove similar results on a single trace of power for the other groups. However, since there is no extended range result for a single trace of power in the other finite groups, and since we were not able to get better results here than the results for the joint distribution of traces of powers, we did not pursue this.
\end{rem}

\subsection{Distribution of the characteristic polynomial}\label{section_distribution_charpoly_gln}

To compute the total variation distance 
$d_{TV}(\mu_{\TR_{d_1,d_2}^G},\mu_{\UTR_{d_1,d_2}})$ we pass through estimating the distribution of the characteristic polynomial in short intervals of $\field_q[x]$. While we cannot get a discrepancy bound for the distribution in intervals as in \cite{GR21}, which is good for any interval, we do manage to show that on average the discrepancy is small. For the rest of the section, fix $0\le d_1,d_2<n$ two integers such that $d_1+d_2=d<n-1$.

\subsubsection{One step from $\pi^{k-1}$ to $\pi^{k}$}

The basic idea of the proof is sort of an induction step, assuming we already know that the distribution modulo $\pi^{k-1}$ is close to uniform and going up to mod $\pi^k$. The base case $k=1$ is given by our extension of the results in \cite{GR21}, namely Proposition \ref{corollary_neg_traces_gln}. The induction step is based on the properties of  $\img(\partial\charc_{A_0})$, as discussed in Section \ref{section_charpoly_derivative_gln}.

Let $g\in(\resring{k})[x]$ be a monic polynomial of degree $n$, and let $A\in GL_n(\resring{k})$ be written as $A_0+\pi^k A_1$, where $A_0\in H_{G,k}$ is a representative for $A\pmod {\pi^{k-1}}$. Recall that then, the possible values for $\charc(A)$ are precisely the polynomials in $\charc(A_0)+\pi^{k-1}\img(\partial\charc_{\overline{A}_0})$, and that

$$
\img(\partial\charc_{\overline{A}_0})=\left\{\frac{\charc A_0}{\min A_0}a\Bigg|
a\in\field_q[x]_{<\deg\min A_0}\right\}.
$$

\begin{lem}\label{lem_onestep_charpoly_distribution}
Let $g\in(\resring{k})[x]$ be a monic polynomial of degree $n$. Then 

$$
\Pr\left[\charc(A)\equiv g\pmod{R_{d_1,x^{d_2+1},\pi^k}}\Bigg|\charc(A)\equiv g\pmod{R_{d_1,x^{d_2+1},\pi^{k-1}}},\deg\min\overline{A}> d\right]=q^{-(d+1)}.
$$
\end{lem}

\begin{proof}
We show that for any possible choice of $A\mod{\pi^{k-1}}$, the conditional probability is $q^{-(d+1)}$, and this will prove the claim. 
Write $A=A_0+\pi^{k-1}A_0A_1$. Then, 

$$
\charc(A)=\charc(A_0)+\pi^{k-1} \partial\charc_{\overline{A}_0}(A_1).
$$

We are given that $\charc(\rho_{k-1}(A))\equiv g\pmod{R_{d_1,x^{d_2+1},\pi^{k-1}}}$, thus we can write $g-\charc(A_0)=e_0(x)+\pi^{k-1}e_1(x)$, where $\deg e_0(x)<n-d_1$ and $e_0(x)=0\pmod {x^{d_2+1}}$. Hence to have $\charc(A)\equiv g\pmod{R_{d_1,x^{d_2+1},\pi^{k}}}$ we actually need to have
$$
\deg(\partial\charc_{\overline{A_0}}(A_1)-\overline{e_1})<n-d_1,\partial\charc_{\overline{A_0}}(A_1)-\overline{e_1}\equiv 0\pmod{x^{d_2+1}}.
$$
Since $\deg(\overline{e_1}),\deg(\partial\charc_{\overline{A_0}}(A_1))<n$ this is equivalent to having
$$
x^n+\partial\charc_{\overline{A}_0}(A_1)\equiv x^n+\overline{e_1} \pmod{R_{d_1,x^{d_2+1},\pi}}.
$$
Finally, since $\frac{\charc\overline{A_0}}{\min\overline{A_0}}$ is monic, this is equivalent to having 
$$
x^{\deg\min\overline{A_0}}\frac{\charc\overline{A_0}}{\min\overline{A_0}}+\partial\charc_{\overline{A_0}}(A_1)\equiv x^{\deg\min\overline{A_0}}\frac{\charc\overline{A_0}}{\min\overline{A_0}}+\overline{e_1}\pmod{R_{d_1,x^{d_2+1},\pi}}.
$$
Let $r,s\in\field_q[x]_{<n}$ be two polynomials. We denote by $1_{r,s}$ the indicator that
$$
x^{\deg\min\overline{A_0}}\frac{\charc\overline{A_0}}{\min\overline{A_0}}+r\equiv x^{\deg\min\overline{A_0}}\frac{\charc\overline{A_0}}{\min\overline{A_0}}+s\pmod{R_{d_1,x^{d_2+1},\pi}}.
$$

We now note that the map $\partial\charc_{\overline{A}_0}:\mathfrak{gl}_n(\field_q)\to \field_q[x]_{<n}$ is linear, hence all fibers are of the same size. Thus, it is enough to argue that the image is equidistributed modulo $R_{d_1,x^{d_2+1},\pi}$. Now we are ready to compute the probability directly, using Hayes characters. From the above discussion, it is equal to

\begin{multline*}
q^{-\deg\min\overline{A_0}}\sum_{a\in \field_q[x]_{<\deg\min\overline{A_0}}}1_{\frac{\charc\overline{A_0}}{\min\overline{A_0}}a,\overline{e_1}}=\\
=q^{-\deg\min\overline{A_0}}\sum_{a\in \field_q[x]_{<\deg\min\overline{A_0}}}q^{-(d+1)}\sum_{\chi\in \widehat{R}_{d_1,x^{d_2+1},\pi}} \chi\left(\frac{\charc\overline{A_0}}{\min\overline{A_0}}(x^{\deg\min\overline{A_0}}+a)\right)\overline{\chi}
\left(x^{\deg\min\overline{A_0}}\frac{\charc\overline{A_0}}{\min\overline{A_0}}+\overline{e_1}\right)=\\
=q^{-(d+1)}\sum_{\chi\in \widehat{R}_{d_1,x^{d_2+1},\pi}}\chi\left(\frac{\charc\overline{A_0}}{\min\overline{A_0}}\right)\overline{\chi}
\left(
x^{\deg\min\overline{A_0}}\frac{\charc\overline{A_0}}{\min\overline{A_0}}+\overline{e_1}
\right)q^{-\deg\min\overline{A_0}}\sum_{a\in \field_q[x]_{<\deg\min\overline{A_0}}}\chi(x^{\deg\min\overline{A_0}}+a).
\end{multline*}
Here the innermost sum equals $1_{\chi=\chi_0}$ by the the orthogonality of characters, hence we finish.
\end{proof}

\subsubsection{Distribution mod $\pi^k$ given mod $\pi$ behavior}\label{section_gln_pk_given_mod_p}

Let $g\in(\resring{k})[x]$ be a monic polynomial of degree $n$. Let $B\in GL_n(\field_q)$ be a matrix with $\charc B\equiv \overline{g}\pmod{R_{d_1,x^{d_2+1},\pi}}$. Define the discrepancy of the conditional distribution starting from $B$ to be 

$$
\Delta_{B,k}^{g}:=
\condprob{\charc(A)\equiv g\pmod{R_{d_1,x^{d_2+1},\pi^k}}}{\overline{A}=B}
-q^{-(d+1)(k-1)}.
$$

For $\deg\min B> d$, the discrepancy is zero.

\begin{prop}\label{prop_deltabk_large_minpoly}
Let $g\in(\resring{k})[x]$ be a monic polynomial of degree $n$, and let $B\in GL_n(\field_q)$ be a matrix with $\charc(B)\equiv g\pmod{\hayesrelation{}}$. Then, if $\deg\min(B)> d$, $\Delta_{B,k}^{g}=0$. 
\end{prop}

\begin{proof}
One can write 

    \begin{multline*}
    \condprob{\charc(A)\equiv g\pmod{\hayesrelation{k}}}{\overline{A}=B}
    =\\
    =\prod_{i=2}^k \condprob{\charc(\rho_i(A))\equiv\rho_i(g)\pmod{\hayesrelation{i}}}{\charc(\rho_{i-1}(A))\equiv\rho_{i-1}(g)\pmod{\hayesrelation{i-1}},\overline{A}=B}.
    \end{multline*}
We get that to compute $\Delta_{B,k}^g$, it is enough to compute each term in the product above.
Using Lemma \ref{lem_onestep_charpoly_distribution}, we get that for all $i$,

$$
\condprob{\charc(\rho_i(A))\equiv\rho_i(g)\pmod{\hayesrelation{i}}}{\charc(\rho_{i-1}(A))\equiv\rho_{i-1}(g)\pmod{\hayesrelation{i-1}},\overline{A}=B}=q^{-(d+1)},
$$
so that  $\Delta_{B,k}^{g}=0$.

\end{proof}

In the case of $B$ with $\deg\min B$ small, we would like to bound this discrepancy as tightly as possible, to get a result on the distribution of $\TR_{d_1,d_2}^{G,k}$ (or equivalently, on the distribution of $\charc A$ in Hayes residue classes). We start by considering the lifts of $B\in GL_n(\field_q)$, when $\deg\min B=l\le d$. 

\begin{prop}\label{prop_deltabk_small_minpoly}
Let $g\in(\resring{k})[x]$ be a monic polynomial of degree $n$. Let $B\in GL_n(\field_q)$ be a matrix with $\deg\min B=l\le d$. Let $\mathcal{P}_B^k=\{\charc(A):A\in GL_n(\resring{k}),\overline{A}=B\}$. Then,

\begin{enumerate}
    \item $|\mathcal{P}_B^k|=q^{(k-1)l}$. Moreover, each element of $\mathcal{P}_B^k$ is in a different residue class modulo $\hayesrelation{k}$. 
    \item $\condprob{\charc(A)\equiv g\pmod{\hayesrelation{k}}}{\overline{A}=B}=1_{\mathcal{P}_B^k\cap g\hayesrelation{k}\neq \emptyset}q^{-(k-1)l}
$. Equivalently, we have
$$\Delta_{B,k}^{g}=1_{\mathcal{P}_B^k\cap g\hayesrelation{k}\neq \emptyset}q^{-(k-1)l}-q^{-(d+1)(k-1)}.$$
\end{enumerate}
\end{prop}

\begin{proof}
We prove the claim by induction on $k$. For $k=1$ this is trivial. Suppose that we proved the proposition for all natural numbers $<k$, and prove it to $k$. We note that the fibers of the map $\mathcal{P}_B^k\to\mathcal{P}_B^{k-1}$, $h\mapsto \rho_{k-1}(h)$, are precisely of the form $f_0+\pi^{k-1}\img(\partial\charc_{B})$, for some $f_0\in(\resring{k})[x]$. We note that by Theorem \ref{theorem_image_char_derivative_gl_n}, each fiber is of size $q^l$. Finally, we note that by Lemma \ref{lem_onestep_charpoly_distribution} with $d=l$, each of the elements in the fiber is in a different residue class of $\hayesrelation{k}$, hence the first assertion follows for $k$. As for the second assertion, it follows for $k$ from the previous proposition when we take $d=l$, together with the first assertion of the current proposition.
\end{proof}

We now let $\mathcal{G}$ be a set of representatives for the lifts of $\charc(B)$ to a monic polynomial $g\in(\resring{k})[x]$ of degree $n$, modulo the relation $R_{d_1,x^{d_2+1},\pi^k}$. We note that $|\mathcal{G}|=q^{(d+1)(k-1)}$. 

\begin{prop}\label{prop_summing_deltabbk_over_all_g}
Let $B\in GL_n(\field_q)$. Then,

\begin{enumerate}
    \item If $\deg\min{B}> d$, we have $\sum_{g\in \mathcal{G}}|\Delta_{B,k}^{g}|=0$.
    \item If $\deg\min B=l\le d$, we have $\sum_{g\in\mathcal{G}}|\Delta_{B,k}^{g}|\le 2$.
\end{enumerate}
\end{prop}

\begin{proof}
The first claim follows directly from Proposition \ref{prop_deltabk_large_minpoly}. The second follows from the triangle inequality, together with the two parts of Proposition \ref{prop_deltabk_small_minpoly}.
\end{proof}

Let $g\in(\resring{k})[x]$ be a monic polynomial of degree $n$. We define the discrepancy of a polynomial $\overline{f}\in\field_q[x]$ such that $\overline{f}\equiv \overline{g}\pmod {R_{d_1,x^{d_2+1},\pi}}$ by 

$$
\Delta_{\overline{f},k}^{g}:=\Pr[\charc(A)\equiv g\pmod{\hayesrelation{k}}|\charc(\overline{A})=\overline{f}]-q^{-(d+1)(k-1)}.
$$

\begin{lem}\label{lem_discrepancy_bound_interval_starting_from_f_mod_p}
Let $\overline{f}\in \field_q[x]_{=n}^\mon$ be a polynomial, and let $\mathcal{G}$ be a set of representatives for the possible lifts of $\overline{f}$ to a monic polynomial in $(\resring{k})[x]$ of degree $n$, modulo the relation $\hayesrelation{k}$.
\begin{enumerate}
    \item If $\deg\rad\overline{f}> d$, then $\sum_{g\in\mathcal{G}}|\Delta_{\overline{f},k}^g
    |=0$. 
    \item If $\deg\rad\overline{f}=s\le d$, then $\sum_{g\in\mathcal{G}}|\Delta_{\overline{f},k}^g|\le 2\Pr_{B\in GL_n(\field_q)}[\deg\min(B)\le d|\charc(B)=\overline{f}]$.
\end{enumerate}
\end{lem}

\begin{proof}
We have

\begin{equation*}\label{eqn_disc_sum_gl}
\Delta_{\overline{f},k}^g=\sum_{\substack{B\in GL_n(\field_q)\\ \charc(B)=\overline{f}}}\Delta_{B,k}^g\cdot \Pr[\overline{A}=B|\charc(\overline{A})=\overline{f}].    
\end{equation*}
By Proposition \ref{prop_deltabk_large_minpoly} we conclude that 

$$
\Delta_{\overline{f},k}^g=\sum_{\substack{B\in GL_n(\field_q)\\ \charc(B)=\overline{f},\deg\min B\le d}}\Delta_{B,k}^g\cdot \Pr[\overline{A}=B|\charc(\overline{A})=\overline{f}]
$$
For $\deg\rad\overline{f}>d$, since $\rad\overline{f}|\min B$, we get that $\Delta_{\overline{f},k}^g=0$ and this proves the first assertion of the lemma. For $\deg\rad\overline{f}=s\le d$, summing the previous equation over all $g\in\mathcal{G}$ and then applying the triangle inequality and exchanging the order of summation we get

\begin{multline*}
\sum_{g\in\mathcal{G}}|\Delta_{\overline{f},k}^g|\le
\sum_{\substack{B\in GL_n(\field_q)\\ \charc(B)=\overline{f},\deg\min B\le d}} \Pr[\overline{A}=B|\charc(\overline{A})=\overline{f}]\sum_{g\in\mathcal{G}}|\Delta_{B,k}^g|\le\\
\le 2\Pr_{B\in GL_n(\field_q)}[\deg\min(B)\le d|\charc(B)=\overline{f}].
\end{multline*}
\end{proof}

\subsubsection{Discrepancy bound on distribution in Hayes residue classes}\label{section_discrepancy_bound_on_intervals_gln}

Now, for a monic polynomial $g\in(\resring{k})[x]$ of degree $n$ we define the discrepancy of its Hayes residue class to be

$$
\Delta_{g}:=\Pr[\charc A\equiv g\pmod{\hayesrelation{k}}]-q^{-k(d+1)}.
$$

\begin{lem}\label{lem_average_discrepancy_bound_gln}
Let $\mathcal{G}$ be a set of representatives for monic polynomials in $(\resring{k})[x]$ of degree $n$ modulo the relation $\hayesrelation{k}$. Then,

\begin{equation*}
\sum_{g\in\mathcal{G}}|\Delta_g|=O\left(q^{-\frac{n^2}{2d}}\left(1+\frac{1}{q-1}\right)^n\binom{n+d}{n}+q^{d-\frac{n^2}{d}+o(n)}\right).
\end{equation*}
\end{lem}

\begin{proof}
Let $g\in\mathcal{G}$. We have (for $A\in GL_n(\resring{k})$)

\begin{multline*}
\Pr[\charc A\equiv g\pmod{\hayesrelation{k}}]=\\
=\sum_{\overline{f}\equiv \overline{g}\pmod{\hayesrelation{}}\atop{\overline{f}\in\field_q[x]}}\Pr[\charc A\equiv g\pmod{\hayesrelation{k}}|\charc\overline{A}=\overline{f}]\Pr[\charc\overline{A}=\overline{f}]=\\
=\sum_{\overline{f}\equiv \overline{g}\pmod{\hayesrelation{}}\atop{\overline{f}\in\field_q[x]}}\left[q^{-(k-1)(d+1)}+\Delta_{\overline{f},k}^g\right]\Pr[\charc\overline{A}=\overline{f}]=\\
=q^{-(k-1)(d+1)}\Pr[\charc\overline{A}\equiv \overline{g}\pmod{\hayesrelation{}}]+
\sum_{\overline{f}\equiv \overline{g}\pmod{\hayesrelation{}}}\Delta_{\overline{f},k}^g\cdot \Pr[\charc\overline{A}=\overline{f}].
\end{multline*}
By Corollary \ref{corollary_neg_traces_gln}, we have the following bound on the mod $\pi$ discrepancy:

$$
\Pr[\charc\overline{A}\equiv \overline{g}\pmod{\hayesrelation{}}]=q^{-(d+1)}+O\left(q^{-\frac{n^2}{2d}}\left(1+\frac{1}{q-1}\right)^n\binom{n+d}{n}\right).
$$
Hence,

$$
|\Delta_g|<O\left(q^{-(k-1)(d+1)-\frac{n^2}{2d}}\left(1+\frac{1}{q-1}\right)^n\binom{n+d}{n}\right)+\sum_{\overline{f}\equiv \overline{g}\pmod{\hayesrelation{}}}|\Delta_{\overline{f},k}^g|\cdot \Pr[\charc\overline{A}=\overline{f}].
$$
Summing the last equation over all $g\in\mathcal{G}$, and using Lemma \ref{lem_discrepancy_bound_interval_starting_from_f_mod_p}, we get that 

\begin{multline*}
\sum_{g\in\mathcal{G}}|\Delta_g|<\\
<O\left(q^{d-\frac{n^2}{2d}}\left(1+\frac{1}{q-1}\right)^n\binom{n+d}{n}\right)+\sum_{\substack{\overline{f}\in\field_q[x]\\
\deg\rad \overline{f}\le d}}2\Pr_{B\in GL_n(\field_q)}[\deg\min(B)\le d|\charc(B)=\overline{f}]\Pr_{B\in GL_n(\field_q)}[\charc B=\overline{f}].
\end{multline*}
We note that by Claim \ref{claim_bounding_probability_small_minpoly},

$$\Pr_{B\in GL_n(\field_q)}[\deg\min(B)\le d|\charc(B)=\overline{f}]\Pr_{B\in GL_n(\field_q)}[\charc B=\overline{f}]=\Pr[\deg\min(B)\le d\cap \charc(B)=\overline{f}]<q^{-\frac{n^2}{d}+o(n)},$$
By Lemma \ref{lem_small_radical_bound}, the number of polynomials in $\field_q[x]_{=n}^\mon$ with a radical of degree $\le d$ is $q^{d+o(n)}$. Hence we get the result.   
\end{proof}

\subsection{Joint distribution of traces}\label{section_join_distribution_traces_gln}

We would like to bound the total variation distance $d_{TV}(\mu_{\TR_{d_1,d_2}^{GL}},\mu_{\UTR_{d_1,d_2}})$. Using Lemma \ref{lem_total_variation_dist_limit} we may reduce to the $\pmod{\pi^k}$ total variation distance. We first recall our Corollary \ref{corollary_neg_traces_gln}, which extends the result \cite[Theorem 1.3]{GR21}. It follows from it that there is a constant $c_q^{GL}$ for which there is a super-exponential bound for the discrepancy of the distribution of $\charc(B)$ (for $B\in GL_n(\field_q)$) in Hayes residue classes of length $d<c_q^{GL}\cdot n$. The constant $c_{q}^{GL}$ is determined by the range of $d$ for which 

$$
q^{d-\frac{n^2}{2d}}\left(1+\frac{1}{q-1}\right)^n\binom{n+d}{n}=_{n\to\infty}o(1).
$$

\begin{theorem}\label{theorem_mod_pk_joint_traces}
Let $0\le d_1,d_2<n$ be integers such that $d_1+d_2=d<n-1$. Let $\TR_{d_1,d_2}^{GL,k}$ be the random variable attaching to a matrix $M\in GL_n(\resring{k})$ its $(d_1,d_2)$-trace datum, and let $\UTR_{d_1,d_2}^k$ be the uniform measure on $(d_1,d_2)$-trace data, as defined in Section \ref{section_trace_data_poly_intervals}. Then, the total variation distance

\begin{multline*}
d_{TV}(\mu_{\TR_{d_1,d_2}^{GL,k}},\mu_{\UTR_{d_1,d_2}^k})=\\
=\sum_{\mathfrak{a} \text{ is a }(d_1,d_2)-\text{trace datum}}
\Bigg|
\Pr_{M\in GL_n(\resring{k})}[M\text{ has trace data }\mathfrak{a}]-q^{-(kd-S(d_1,k)-S(d_2,k))}
\Bigg|
<\\
<O\left(
q^{d-\frac{n^2}{2d}}\left(1+\frac{1}{q-1}\right)^n\binom{n+d}{n}+q^{d-\frac{n^2}{d}+o(n)}
\right).
\end{multline*}
In particular, as long as $d<c_q^{GL}\cdot n$, this total variation distance is $o(1)$ as $n\to\infty$, uniformly in $k$. 
\end{theorem}

\begin{proof} 
Let $S=S(d_1,k)+S(d_2,k)$. 
Lemma \ref{lem_connection_traces_charpoly} gives that there are polynomials $\{g_j\}_{j=1}^{q^{S+k}}\subset(\resring{k})[x]$, such that

$$
\Pr_{M\in GL_n(\resring{k})}[M\text{ has trace datum }\mathfrak{a}]=\sum_{i=1}^{q^{S+k}}\Pr_{M\in GL_n(\resring{k})}[M\equiv g_j\pmod{\hayesrelation{k}}]=q^{S-kd}+\sum_{i=1}^{q^{S+k}}\Delta_{g_j}.
$$

Now, let $\mathcal{G}=\{g_j\}_{j=1}^{S+k}, \mathcal{G}\subset(\resring{k})[x]$ be a set of representatives for the monic polynomials modulo $\hayesrelation{k}$. Using the triangle inequality we get that 

\begin{equation*}
\sum_{\mathfrak{a}\text{ trace datum}}
\Bigg|
\Pr_{M\in GL_n(\resring{k})}[M\text{ has trace data }\mathfrak{a}]-q^{-(kd-S)}
\Bigg|
<\sum_{g\in \mathcal{G}}|\Delta_g|.
\end{equation*}
Now the theorem follows from Lemma \ref{lem_average_discrepancy_bound_gln}.
\end{proof}

By taking $k\to\infty$ we get as a consequence  Theorem \ref{main_theorem_gl} (using Lemma \ref{lem_total_variation_dist_limit}), which we restate here for convenience. 

\begin{theorem}
Let $0\le d_1,d_2<n$ be integers such that $d_1+d_2=d<n-1$. Let $\TR_{d_1,d_2}^{GL}$ be the random variable attaching to a matrix $M\in GL_n(\Ol)$ its $(d_1,d_2)$-trace datum. Let $\UTR_{d_1,d_2}$ be uniform $(d_1,d_2)$-trace datum. Then, we have the following total variation bound

$$
d_{TV}(\mu_{\TR_{d_1,d_2}^{GL}}, \mu_{\UTR_{d_1,d_2}})=O\left(
q^{d-\frac{n^2}{2d}}
\left(
1+\frac{1}{q-1}
\right)^n\binom{n+d}{n}+q^{d-\frac{n^2}{d}+o(n)}
\right).
$$
In particular, for $d<c_q^{GL}\cdot n$, this total variation distance is $o(1)$ as $n\to\infty$.
\end{theorem}

As a direct consequence, we get the following

\begin{theorem}
    Fix $d_1,d_2>0$. Let $B_{-d_2},\ldots,B_{-1},B_1,\ldots,B_{d_1}$ be a fixed collection of balls in $\Ol$, and let $U_{-1},\ldots,U_{-d_2},U_1,\ldots,U_{d_1}$ be independent uniform variables on $\Ol$. Set $Z_i=\frac{1}{|i|_p}U_i$. Then,
    $$
    \lim_{n\to\infty} \Pr_{M\in GL_n(\Ol)}
    \left[
\left(
\tr(M^i)-1_{p|i}\sigma(\tr(M^{i/p}))\in B_i
\right)_{0\neq i=-d_2}^{d_1}
\right]=\prod_{0\neq i=-d_2}^{d_1} \Pr[Z_i\in B_i].
    $$
\end{theorem}

\section{Traces and characteristic polynomial for $SL_n$}\label{section_sln_main}

\subsection{The image of $\partial \charc_{A_0}$}

Let $A_0\in SL_n(\field_q)$. Since $\mathfrak{sl}_n$ is preserved by conjugation from $GL_n(\field_q)$, the proofs for $SL_n$ are close to the proofs for $GL_n$. However, there are some modifications to make. In this section, we use the map $\partial\charc_{A_0}$ with domain equal to either $\mathfrak{sl}_n$, $\mathfrak{gl}_n$. When there is a chance of confusion we write $\partial\charc_{A_0}^{SL}$, $\partial\charc_{A_0}^{GL}$ (respectively) to make clear which domain we refer to. Whenever we write $\partial\charc_{A_0}$ alone we always mean the map $\partial\charc_{A_0}^{SL}$.  For the current section \ref{section_sln_main}, $A$ will be a matrix in $SL_n(\resring{k})$ unless explicitly stated otherwise.

\begin{theorem}\label{theorem_image_char_derivative_sln}
Let $B\in SL_n(\field_q)$ be a matrix.  Then, $\img(\partial\charc_B)=
\left\{\frac{\charc B}{\min B}b:\deg b<\deg\min B,b(0)=0\right\}$. 
\end{theorem}

\begin{proof}
We note that $\mathfrak{gl}_n(\field_q)$ is the direct sum of $\mathfrak{sl}_n(\field_q)$ and the one dimensional space 
$
\field_qE_{11}
$,
where $E_{11}$ is the matrix with $1$ in the $(1,1)$ entry and zero elsewhere. Comparing dimensions we get that $\dim \img(\partial\charc_{B}^{SL})\ge \dim(\partial\charc_{B}^{GL})-1=\deg\min B-1$. 

Now let $B_0$ be any specific lift of $B$ to $SL_n(\resring{2})$, and $B'=B_0(I+\pi B_1)$, where $B_1\in \mathfrak{sl}_n(\field_q)$ be any other lift. We recall that by Lemma \ref{lem_charpoly_step_in_p_direction},

$$
\charc(B')=\charc(B_0)-\pi \partial\charc_{B}^{SL}(B_1).
$$
Since $B',B_0\in SL_n(\resring{2})$ we must have that the free coefficients of $\charc(B'),\charc(B_0)$ are equal. Hence, the free coefficient of $\partial\charc_{B}^{SL}(B_1)$ must be zero. Thus, 
$$\img(\partial\charc_B^{SL}(B_1))\subset\left\{\frac{\charc B}{\min B}b:\deg b<\deg\min B,b(0)=0\right\},$$
and is of maximal dimension, hence the equality.
\end{proof}

\subsection{Distribution of the characteristic polynoimal}

The strategy for proving the total variation bound on $d_{TV}(\mu_{\TR_{d_1,d_2}^g},\mu_{\UTR_{d_1,d_2}})$ for p-adic matrix groups $G\neq GL_n$ is the same as for $GL_n(\Ol)$. In the $SL_n$ case, we omit many of the proofs, due to the high similarity to the $GL_n$ case. The main difference is that in this case, the free coefficient is determined, thus some of the constants look different. We now fix a choice of $0\le d_1,d_2<n$ such that $d_1+d_2=d<n-1$ for the rest of the section.

\subsubsection{One step from $\pi^{k-1}$ to $\pi^{k}$}

\begin{lem}
Let $g\in(\resring{k})[x]$ be a monic polynomial of degree $n$, with $g(0)=(-1)^n$. Then 

$$
\Pr[\charc(A)\equiv g\pmod{\hayesrelation{k}}|\charc(\rho_{k-1}(A))\equiv \rho_{k-1}(g)\pmod{\hayesrelation{k-1}},\deg\min\overline{A}\ge d+1]=q^{-d}.
$$
\end{lem}

\begin{proof}
The polynomials with $g(0)=(-1)^n$ form a Hayes residue class mod $R_{0,x,\pi^{k-1}}$. Hence this follows from Lemma \ref{lem_onestep_charpoly_distribution} and the formula for conditional probability.
\end{proof}

\subsubsection{Distribution mod $\pi^k$ given mod $\pi$ behavior}

We define the discrepancy 
of the conditional distribution given a polynomial $\overline{f}\equiv \overline{g}\pmod{\hayesrelation{}}$ by

$$
\Delta_{\overline{f},k}^g:=\Pr[\charc(A)\equiv g\pmod{\hayesrelation{k}}|\charc(\overline{A})=\overline{f}]-q^{-(k-1)d}.
$$
The $SL_n$ analog of the main lemma of Section \ref{section_gln_pk_given_mod_p} is then

\begin{lem}\label{lem_discrepancy_bound_interval_starting_from_f_mod_p_sln}
Let $\overline{f}\in \field_q[x]_{=n}^\mon$ be a polynomial with $\overline{f}(0)=(-1)^n$, and let $\mathcal{G}$ be a set of representatives for the possible lifts of $\overline{f}$ to a monic polynomial in $(\resring{k})[x]$ of degree $n$ with $g(0)=(-1)^n$, modulo $\hayesrelation{k}$.
\begin{enumerate}
    \item If $\deg\rad\overline{f}\ge d$, then $\sum_{g\in\mathcal{G}}|\Delta_{\overline{f},k}^g
    |=0$. 
    \item If $\deg\rad\overline{f}=s<d$, then $\sum_{g\in\mathcal{G}}|\Delta_{\overline{f},k}^g|\le 2\Pr_{B\in SL_n(\field_q)}[\deg\min(B)<d|\charc(B)=\overline{f}]$.
\end{enumerate}
\end{lem}
The proof is omitted since it is essentially the same as the proof of Lemma \ref{lem_discrepancy_bound_interval_starting_from_f_mod_p}.

\subsubsection{Discrepancy bound on distribution in Hayes residue classes}

For $g\in(\resring{k})[x]$ monic of degree $n$ with $g(0)=(-1)^n$, we define the discrepancy $\Delta_g$ by 

$$
\Delta_g:=\Pr[\charc(A)\equiv g\pmod{\hayesrelation{k}}]-q^{-kd}.
$$

\begin{lem}\label{lem_average_discrepancy_bound_sln}
Let $\mathcal{G}$ be a set of representatives for monic polynomials $g$ in $(\resring{k})[x]$ of degree $n$, with $g(0)=(-1)^n$, modulo the relation $\hayesrelation{k}$. Then,

\begin{equation*}
\sum_{g\in\mathcal{G}}|\Delta_g|=O\left(q^{-\frac{n^2}{2d}}\left(1+\frac{1}{q-1}\right)^n\binom{n+d}{n}+q^{d-\frac{n^2}{d}+o(n)}\right).
\end{equation*}
\end{lem}

\begin{proof}
The proof is the same as that of Lemma \ref{lem_average_discrepancy_bound_gln}, with the modifications of using the analogous results for $SL_n$ (proven in the current section) when needed. To bound the mod $\pi$ discrepancy, we use our extension of \cite[Theorem 4.6]{GR21}, namely Corollary \ref{cor_neg_traces_mod_q_sln}.
\end{proof}

\subsection{Joint distribution of traces}

As in Section \ref{section_join_distribution_traces_gln}, we start by a bound on the total variation distance of the mod $\pi^k$ variables, $\TR_{d_1,d_2}^{SL,k},\UTR_{d_1,d_2}^k$.

\begin{theorem}
Let $\TR_{d_1,d_2}^{SL,k}$ be the random variable attaching to a matrix $M\in SL_n(\resring{k})$ its $(d_1,d_2)$-trace datum, and let $\UTR_{d_1,d_2}^{k}$ be a uniform $(d_1,d_2)$-trace datum, as defined in Section \ref{section_trace_data_poly_intervals}. Let $S=S(d_1,k)+S(d_2,k)$. Then,

\begin{multline*}
d_{TV}(\mu_{\TR_{d_1,d_2}^{SL,k}},\mu_{\UTR_{d_1,d_2}^k})=\sum_{\mathfrak{a}\text{ a } (d_1,d_2)\text{-trace data}}
\Bigg|
\Pr_{M\in SL_n(\resring{k})}[M\text{ has }(d_1,d_2)\text{-trace data }\mathfrak{a}]-q^{-(kd-S)}
\Bigg|<\\
<O\left(
 q^{\frac{n}{2}-\frac{n^2}{2d}}\left(1+\frac{1}{q^2-1}\right)^n\binom{n+d}{n}+q^{d-\frac{n^2}{d}+o(n)}
\right).
\end{multline*}
In particular, as long as $d<c_q^{SL}\cdot n$, this total variation distance is $o(1)$ as $n\to\infty$, uniformly in $k$. 
\end{theorem}

The proof is omitted since it is essentially the same as that of Theorem \ref{theorem_mod_pk_joint_traces}, combining the observation that since the free coefficient is determined, many of the summands in Lemma \ref{lem_connection_traces_hayes_classes} are $0$. From here Theorem \ref{main_theorem_sl} follows by taking  $k\to\infty$ in the previous theorem. We also get as a corollary

\begin{theorem}
    Fix $d_1,d_2>0$. Let $B_{-d_2},\ldots,B_{-1},B_1,\ldots,B_{d_1}$ be a collection of balls in $\Ol$, and let $U_{-1},\ldots,U_{-d_2},U_1,\ldots,U_{d_1}$ be independent, uniform random variables on $\Ol$. Set $Z_i=\frac{1}{|i|_p}U_i$. Then,
    $$
    \lim_{n\to\infty} \Pr_{M\in SL_n(\Ol)}\left[
\left(
\tr(M^i)-1_{p|i}\sigma(\tr(M^{i/p}))\in B_i
\right)_{0\neq i=-d_2}^{d_1}
\right]=\prod_{0\neq i=-d_2}^{d_1} \Pr[Z_i\in B_i].
    $$
\end{theorem}

\section{Traces and characteristic polynomial for $Sp_{2n}$}\label{section_sp_main}

\subsection{The image of $\partial\charc_{A_0}$}\label{section_img_derv_charpoly_sp2n}

Recall the map $\partial \charc_{A_0}:\mathfrak{sp}_{2n}(\field_q)\to\field_q[x]_{<2n}$, defined (for $G=Sp_{2n}$) by $A_1\mapsto x\cdot\tr(\Adj(x-A_0) A_1)$. As discussed in Section \ref{subsection_explicit_representatives}, to compute the image of $\partial\charc_{A_0}$ one can move from $A_0$ to any element in its conjugacy class. To calculate the image, we will need the following preparatory lemma:

\begin{lem}\label{lemma_image_of_diagonal_sum_symplectic}
Let $A\in Sp_{2n}(\field_q)$, $B\in Sp_{2m}(\field_q)$, and suppose that $A$ is an upper-triangular (or lower-triangular) block matrix. Let $A\triangle B\in Sp_{2n+2m}(\field_q)$ be their triangular join. Then,
$$
\img(\partial\charc_{A\triangle B})= \mathrm{span}_{\field_q}(\charc(B)\cdot \img (\partial\charc_A),\charc(A)\cdot \img(\partial\charc_B)).
$$
\end{lem}

\begin{proof}
Write 

$$
A=
\left(
\begin{array}{cc}
    A_1 & A_3 \\
    0 & A_2
\end{array}
\right), 
A\triangle B=
\left(
\begin{array}{ccc}
    A_1 & 0 & A_3\\
    0 & B & 0 \\
    0 & 0 & A_2
\end{array}
\right).
$$
We note that 

$$
\Adj(x-A\triangle B)=
\charc(A)\charc(B)\left(
\begin{array}{ccc}
    x-A_1 & 0 & -A_3\\
    0 & x-B & 0 \\
    0 & 0 & x-A_2
\end{array}
\right)^{-1}.
$$
A simple calculation shows now that 

$$
\Adj(x-A\triangle B)=
\charc(A)\charc(B)\left(
\begin{array}{ccc}
    (x-A_1)^{-1} & 0 & (x-A_1)^{-1}A_3(x-A_2)^{-1}\\
    0 & (x-B)^{-1} & 0 \\
    0 & 0 & (x-A_2)^{-1}
\end{array}
\right).
$$

Now let $C\in\mathfrak{sp}_{2n+2n}(\field_q)$. We can write $C$ as a $3\times 3$ block matrix, and note that whenever a block in $\Adj(x-A\triangle B)$ is zero, the corresponding block in $C$ does not contribute to $\tr(\Adj(x-A\triangle B) C)$, thus we may assume for simplicity that it is zero. Thus, we may write

$$
C = 
\left(
\begin{array}{ccc}
    C_1 & 0 & C_3\\
    0 & C_B & 0 \\
    0 & 0 & C_2
\end{array}
\right).
$$
By the structure of $\mathfrak{sp}_{2n+2m}(\field_q)$, we get that $C_B\in \mathfrak{sp}_{2m}(\field_q)$, $
C_A=\left(
\begin{array}{cc}
    C_1 & C_3 \\
    0 & C_2 
\end{array}
\right)\in \mathfrak{sp}_{2n}(\field_q)$. Again, a simple calculation shows that

$$
\tr(\Adj(x-A\triangle B) C)=\charc(B)\tr(\Adj(x-A) C_A)+\charc(A)\tr(\Adj(x-B) C_B),
$$
thus proving the claim.
\end{proof}

\begin{prop}\label{prop_image_is_palindromic}
Let $A_0\in Sp_{2n}(\field_q)$, $A_1\in\mathfrak{sp}_{2n}(\field_q)$. Then $\partial\charc_{A_0}(A_1)$ is a $2n$-palindromic polynomial of degree $<2n$. 
\end{prop}

\begin{proof}
Indeed, lift $A_0$ to a matrix in $Sp_{2n}(\resring{2})$, denoted also $A_0$ by a small abuse of notation. Writing $A=A_0+\pi A_0A_1$ we have that $f=\charc(A)$ is self-reciprocal. Write $f=\charc(A_0)+\pi \partial\charc_{A_0}(A_1)$. Now,

$$
f(x)=x^{2n}f(1/x)=x^{2n}\charc(A_0)(1/x)+\pi x^{2n}\partial\charc_{A_0}(A_1)(1/x).
$$
Since $\charc(A_0)$ is itself self-reciprocal, we get that $\partial\charc_{A_0}(A_1)$ is $2n$-palindromic.
\end{proof}

\begin{theorem}\label{theorem_image_char_derivative_sp2n}
Let $A_0\in Sp_{2n}(\field_q)$ be a symplectic matrix. Then, 
$$
\img(\partial\charc_{A_0})=\left\{a:a=\frac{\charc\overline{A_0}}{\min\overline{A_0}}b,
b\in\field_q[x]_{<\deg\min \overline{A}_0}^\pal\right\}.
$$
\end{theorem}

\begin{proof}
Without loss of generality, we assume that $\charc(A_0)$ splits completely (otherwise we move to work in an algebraic extension $\field_{q^l}$, as in the proof of Theorem \ref{theorem_image_char_derivative_gl_n}). Now using Lemma \ref{lemma_image_of_diagonal_sum_symplectic} we may reduce to the case of a matrix which is of Type I, II or III as defined in Section \ref{subsubsection_explicit_representatives_sp2n} (the proof for this reduction being essentially the same as the argument given in the proof of Theorem \ref{theorem_image_char_derivative_gl_n}).

\maincase{Type I} 
We want to compute $\img(\partial\charc_{A_{\alpha,\alpha^{-1}}^m})$. Let $A_1=\left(
\begin{array}{cc}
    A & B \\
    C & D
\end{array}
\right)\in\mathfrak{sp}_{2m}(\field_q)$, so that $B,C$ are symmetric and $A=-D^t$. Using Lemma \ref{lem_sp_jordan_block_adj}, we see that 

\begin{multline*}
\partial\charc_{A_{\alpha,\alpha^{-1}}^m}(A_1)=\tr\left(
    \begin{array}{cc}
        (x-\alpha)^m Y_{\alpha}^m &  \\
         & (x-\alpha^{-1})^m (X_{\alpha}^m)^t
    \end{array}
    \right)\left(
\begin{array}{cc}
    A & B \\
    C & D
\end{array}
\right)=\\
=\tr(A(x-\alpha)^m Y_{\alpha}^m)+\tr(D(x-\alpha^{-1})^m (X_{\alpha}^m)^t)=\tr\left(A((x-\alpha)^m Y_{\alpha}^m-(x-\alpha^{-1})^m X_{\alpha}^m)
\right).
\end{multline*}

Since $A$ is arbitrary, we get that
\begin{multline*}
\img(\partial\charc_{A_{\alpha,\alpha^{-1}}^m})=
(x-\alpha)^{m}(x-\alpha^{-1})^{m} \mathrm{span}_{\field_q}\bigg(\frac{\alpha}{1-x\alpha}+\frac{1}{x-\alpha},\frac{1}{(1-x\alpha)^2}+\frac{1}{(x-\alpha)^2},\\
\frac{x}{(1-x\alpha)^3}+\frac{1}{(x-\alpha)^3},\cdots,\frac{x^{m-2}}{(1-x\alpha)^{m}}+\frac{1}{(x-\alpha)^{m}}\bigg).    
\end{multline*}

We claim that this is simply $\field_q[x]_{<2m}^{\pal}$. Note that by Proposition \ref{prop_image_is_palindromic}, the image must be contained in this set. On the other hand, the $\field_q$-dimension of $\field_q[x]_{<2m}^\pal$ is $m$, thus it is enough to exhibit $m$ linearly independent functions in the image. But each element in the spanning set above vanishes to a different order at $\alpha$, thus we prove the claim for Type I matrices.

\maincase{Type II} As in the previous case we get that 

\begin{multline*}
\img(\partial\charc_{A_{\alpha,\pm}^m})=
(x-\alpha)^{2m} \mathrm{span}_{\field_q}\Bigg(\frac{\alpha}{1-x\alpha}+\frac{1}{x-\alpha},\frac{1}{(1-x\alpha)^2}+\frac{1}{(x-\alpha)^2},\\
\frac{x}{(1-x\alpha)^3}+\frac{1}{(x-\alpha)^3},\cdots,\frac{x^{m-2}}{(1-x\alpha)^{m}}+\frac{1}{(x-\alpha)^{m}}\Bigg).
\end{multline*}
By Proposition \ref{prop_image_is_palindromic}, all polynomials in the image are $2m$-palindromic. However, since $\alpha=\alpha^{-1}$ in this case, every polynomial in the image is divisible by $(x-\alpha)^m$. The space of polynomials in $\field_q[x]_{<2m}^{\pal}$ which are divisible by $(x-\alpha)^m$ is simply $(x-\alpha)^m\field_q[x]_{<m}^{\pal}$, and since $m$ is odd, its dimension is $\frac{m-1}{2}$. Hence to show the result we need to find $\frac{m-1}{2}$ linearly independent elements in the image. Now note that for $i=1,\cdots, \frac{m-1}{2}$, 

$$
(x-\alpha)^{2m}\left(\frac{x^{2i-2}}{(1-x\alpha)^{2i}}+\frac{1}{(x-\alpha)^{2i}}\right)=(x-\alpha)^{2m-2i}(1-\alpha x^{2i-2})
$$
is of order $2m-2i$ at $\alpha$ (as $1-\alpha x^{2i-2}$ does not vanish at $\alpha$), thus we proved the result for Type II matrices.

\maincase{Type III} Let $A_1=\left(
\begin{array}{cc}
    A & B \\
    C & D
\end{array}
\right)\in\mathfrak{sp}_{2m}(\field_q)$, so that $B,C$ are symmetric and $A=-D^t$. Using Lemma \ref{lem_sp_jordan_block_adj}, we see that 

\begin{multline*}
\partial\charc_{A_{\alpha,\pm}^{2m}}(A_1)=\tr \left(
    \begin{array}{cc}
        (x-\alpha)^mY_\alpha^m & 0 \\
         Z_\alpha^{2m} & (x-\alpha)^m(X_\alpha^m)^t
    \end{array}
    \right)\left(
\begin{array}{cc}
    A & B \\
    C & D
\end{array}
\right)=\\
=\tr((x-\alpha)^mY_\alpha^m A)+\tr(B Z_\alpha^{2m})+\tr((x-\alpha)^m(X_\alpha^m)^t D)=\\
=\tr\left(A((x-\alpha)^m Y_{\alpha}^m-(x-\alpha^{-1})^m X_{\alpha}^m)
\right)+\tr(BZ_\alpha^{2m}).
\end{multline*}
Since $A$ is arbitrary and $B$ is arbitrary symmetric, the image is simply the $\field_q$-span of the polynomials 

\begin{multline*}
(x-\alpha)^{2m}\left(\frac{\alpha}{1-x\alpha}+\frac{1}{x-\alpha},\frac{1}{(1-x\alpha)^2}+\frac{1}{(x-\alpha)^2},\frac{x}{(1-x\alpha)^3}+\frac{1}{(x-\alpha)^3},\cdots,\frac{x^{m-2}}{(1-x\alpha)^{m}}+\frac{1}{(x-\alpha)^{m}}\right),\\
(x-\alpha)^{2m}\left(\frac{(-1)^{r-1}\alpha^s f_r}{(x-\alpha)^s}+\frac{(-1)^{s-1}\alpha^r f_s}{(x-\alpha)^r}\right)_{r,s=1}^m.
\end{multline*}
Here, $f_i$ are defined as in Proposition \ref{prop_matrix_comp_sp2n}. We claim that the span is simply $\field_q[x]_{<2m}^{\pal}$. To prove the claim we need to find $m$ linearly independent functions in the image. We note that $(x-\alpha)^{2m}\frac{(-1)^{j-1}\alpha^j f_j}{(x-\alpha)^j}$ (that is, choosing $r=s=j$) vanishes to order $2m-2j$ at $1$, so we finish the proof.
\end{proof}

\subsection{Distribution of the characteristic polynomial}

In computing the distribution and discrepancies of the characteristic polynomial, we follow the same steps as in Section \ref{section_distribution_charpoly_gln}. For the rest of Section \ref{section_sp_main}, $A$ will be a matrix from $Sp_{2n}(\resring{k})$, unless explicitly stated otherwise.

\subsubsection{One step from $\pi^{k-1}$ to $\pi^k$}

Let $A\in Sp_{2n}(\resring{k})$ be written as $A=A_0+\pi^{k-1}A_0A_1$, where $A_0\in Sp_{2n}(\resring{k})$ is a representative for $A\pmod{\pi^{k-1}}$. When we fix $A_0$, the possibilities for $\charc(A)$ ranges over $\charc(A_0)+\pi^{k-1} \img(\partial\charc_{A_0})$. Recall that from Theorem \ref{theorem_image_char_derivative_sp2n}, 

$$
\img(\partial\charc_{A_0})=\left\{a
\Bigg| 
a=\frac{\charc\overline{A_0}}{\min\overline{A_0}}b, b\in\field_q[x]_{<\deg\min\overline{A_0}}^\pal\right\}.
$$
Since $\frac{\charc\overline{A_0}}{\min\overline{A_0}}$ is palindromic, the condition on $b$ being palindromic is equivalent to $a$ being palindromic. 

\begin{definition}\label{def_delta_number}
For $B\in Sp_{2n}(\field_q)$, define $\delta(B)=\lceil\frac{\deg\min B-1}{2}\rceil$. Slightly more generally, for $r\in \mathbb{N}$ we let $\delta(r)=\lceil\frac{r-1}{2}\rceil$.
\end{definition}

We use the structure of $\img(\partial\charc_{A_0})$ to prove the following lemma, on the distribution of $\charc A$ given $A_0$.

\begin{lem}\label{lem_one_step_sp2n}
Let $g\in (\resring{k})[x]$ be a monic polynomial of degree $2n$. Then for $d<n$,

$$
\Pr[\charc A\in I(g,2n-d)|\charc(\rho_{k-1}(A))\in I(\rho_{k-1}(g),2n-d),\delta(\deg\min\overline{A})\ge d]=q^{-d}.
$$
\end{lem}

\begin{proof}
We show that for any possible choice of $A_0$ with $\delta(\deg\min\overline{A_0})\ge d$, the conditional probability is $q^{-d}$, from which the result follows.

We are given that $\charc(\rho_{k-1}(A_0))\in I(\rho_{k-1}(g),2n-d)$, hence $\charc(A_0)-g=e_0(x)+\pi^k e_1(x)$. Here $\deg e_0(x)<2n-d$ by assumption, and $\deg e_1(x)<2n$ since $g,\charc(A_0)$ are both monic. Thus to have $\charc(A)\in I(g,2n-d)$ we actually need to have 

$$
\deg(\partial\charc_{\overline{A_0}}(A_1)-\overline{e_1})<2n-d.
$$
We assume, WLOG, that $\overline{e_1}(0)=0$ (in reality $\overline{e_1}(0)$ might be non-zero, but the interval defined later will not be different, so we can assume that). Since $\deg(\partial\charc_{\overline{A_0}}(A_1)),\deg(\overline{e_1})<2n$ the last condition is equivalent to 
$$
x^{2n}+\partial\charc_{\overline{A_0}}(A_1)\in I(x^{2n}+\overline{e_1},2n-d).
$$
Since $x^{\deg\min\overline{A_0}}\frac{\charc\overline{A_0}}{\min\overline{A_0}}$ is monic of degree $2n$, this is equivalent to 
$$
x^{\deg\min\overline{A_0}}\frac{\charc\overline{A_0}}{\min\overline{A_0}}+\partial\charc_{\overline{A_0}}(A_1)\in I\left(x^{\deg\min\overline{A_0}}\frac{\charc\overline{A_0}}{\min\overline{A_0}}+\overline{e_1},2n-d\right).
$$
For $r,s$ polynomials in $\field_q[x]_{<2n}$, let us denote by $1_{r,s}$ the indicator that 
$$
x^{\deg\min\overline{A_0}}\frac{\charc\overline{A_0}}{\min\overline{A_0}}+r\in I\left(x^{\deg\min\overline{A_0}}\frac{\charc\overline{A_0}}{\min\overline{A_0}}+s,2n-d\right).
$$

Note that $\dim \mathrm{Im}(\partial\charc_{\overline{A}_0}(A_1))=\delta(A_0)=\delta$. Here, by assumption, $\delta\ge d$. We compute the probability in question using Hayes characters:

\begin{multline*}
\Pr_{A_1\in \mathfrak{sp}_{2n}(\field_q)}[\partial\charc_{\overline{A_0}}(A_1)\in I(\overline{e_1},2n-d)]=q^{-\delta}\sum_{\substack{b\in \field_q[x]\\ \deg(b)<\deg\min\overline{A}_0\\b\text{ is palindromic}}}1_{\frac{\charc\overline{A_0}}{\min\overline{A_0}}b,\overline{e_1}}=
\\
=q^{-\delta}\sum_{\substack{b\in \field_q[x]\\ \deg(b)<\deg\min\overline{A}_0\\b\text{ is palindromic}}}q^{-d}\sum_{\chi\in \shortintervalcharacters}\chi\left(x^{\deg\min\overline{A_0}}\frac{\charc\overline{A_0}}{\min\overline{A_0}}+\frac{\charc\overline{A}_0}{\min\overline{A}_0}b\right)\overline{\chi}\left(x^{\deg\min\overline{A_0}}\frac{\charc\overline{A_0}}{\min\overline{A_0}}+\overline{e_1}\right)=\\
=q^{-d}\sum_{\chi\in \shortintervalcharacters}\chi\left(\frac{\charc\overline{A}_0}{\min\overline{A}_0}\right)\overline{\chi}\left(x^{\deg\min\overline{A_0}}\frac{\charc\overline{A_0}}{\min\overline{A_0}}+\overline{e_1}\right)
q^{-\delta}\sum_{\substack{b\in \field_q[x]\\ \deg(b)< \deg\min\overline{A}_0\\b\text{ is palindromic}}}\chi(x^{\deg\min\overline{A_0}}+b).
\end{multline*}
Using Lemma \ref{lem_hayes_character_sum_over_palindromic_polys} we get that
$$
\sum_{\substack{b\in \field_q[x]\\ \deg(b)< \deg\min\overline{A}_0\\b\text{ is palindromic}}}\chi(x^{\deg\min\overline{A_0}}+b)=q^{\delta}1_{\chi'=\chi_0}.
$$
Substituting this into the previous equation we get the result.
\end{proof}

\subsubsection{Distribution mod $\pi^k$ given mod $\pi$ behavior}

Let $g\in(\resring{k})[x]$ be a monic polynomial of degree $2n$. Define the discrepancy of the conditional distribution given a matrix $B\in Sp_{2n}(\field_q)$ such that $\charc(B)\in I(\overline{g},2n-d)$ to be 

$$
\Delta_{B,k}^g:=\Pr[\charc(A)\in I(g,2n-d)|\overline{A}=B]-q^{-d(k-1)}.
$$

Define $\delta$ as in Definition \ref{def_delta_number}. For $\delta(B)>d$ the discrepancy is zero.

\begin{prop}\label{prop_deltabk_large_minpoly_sp}
Let $g\in (\resring{k})[x]$ be a monic polynomial of degree $2n$, $d<n$. For every matrix $B\in Sp_{2n}(\field_q)$ with $\charc(B)\in I(\overline{g},2n-d)$ and $\delta(B)>d$, we have $\Delta_{B,k}^g=0$.
\end{prop}

\begin{proof}
We write

\begin{multline*}
\Pr[\charc(A)\in I(g,2n-d)|\overline{A}=B]=\\
=\prod_{i=2}^k\Pr[\charc(\rho_{i}(A))\in I(\rho_i(g),2n-d)|\charc(\rho_{i-1}(A))\in I(\rho_{i-1}(g),2n-d),\overline{A}=B],
\end{multline*}
hence using Lemma \ref{lem_one_step_sp2n} we get that indeed $\Delta_{B,k}^g=0$. 
\end{proof}

Now we consider how $\charc(A)$ behaves for lifts $A\in Sp_{2n}(\resring{k})$ of $B$, when $\delta(B)<d$.

\begin{prop}\label{prop_deltabk_small_minpoly_sp}
Let $g\in(\resring{k})[x]$ be a monic polynomial of degree $2n$. Let $B\in Sp_{2n}(\field_q)$ be a matrix with $\delta(B)=l<d$. Let $\mathcal{P}_B^k=\{\charc(A):A\in Sp_{2n}(\resring{k}),\overline{A}=B\}$. Then,

\begin{enumerate}
    \item $|\mathcal{P}_B^k|=q^{(k-1)l}$. Moreover, each element of $\mathcal{P}_B^k$ has different $l$ next-to-leading coefficients modulo $\pi^k$. 
    \item $
\Pr[\charc(A)\in I(g,2n-d)|\overline{A}=B]=1_{\mathcal{P}_B^k\cap I(g,2n-d)\neq \emptyset}q^{-(k-1)l}
$. Equivalently, we have
$$\Delta_{B,k}^g=1_{\mathcal{P}_B^k\cap I(g,2n-d)\neq \emptyset}q^{-(k-1)l}-q^{-(k-1)d}.$$
\end{enumerate}
\end{prop}

\begin{proof}
    We prove the claim by induction on $k$. For $k=1$ this is trivial. Suppose that we proved the claim for $k\ge 1$. We note that the fibers of the map $\mathcal{P}_B^k\to\mathcal{P}_B^{k-1}$ are precisely of the form $f_0+\pi^{k-1}\img(\partial\charc_{B})$. We note that by Theorem \ref{theorem_image_char_derivative_sp2n}, this set is of size $q^l$. Finally, we note that by Lemma \ref{lem_onestep_charpoly_distribution} with $d=l$, each element has different $l$ next-to-leading coefficients mod $\pi^k$, hence the first assertion follows. As for the second assertion, it follows from the previous Proposition when we take $d=l$, together with the first assertion.
\end{proof}

We now let $\mathcal{G}$ be a set of representatives for the lifts of $\charc(B)$ to a monic polynomial $g\in(\resring{k})[x]$ of degree $2n$, modulo equality in the $d$ next-to-leading coefficients. We note that $|\mathcal{G}|=q^{(k-1)d}$.

\begin
{prop}\label{prop_summing_deltabbk_over_all_g_sp}
Let $B\in Sp_{2n}(\field_q)$. Then,

\begin{enumerate}
    \item If $\delta(B)\ge d$, we have $\sum_{g\in G}|\Delta_{B,k}^g|=0$.
    \item If $\delta(B)=l<d$, we have $\sum_{g\in\mathcal{G}}|\Delta_{B,k}^g|\le 2$.
\end{enumerate}
\end{prop}

\begin{proof}
The first claim follows directly from Proposition \ref{prop_deltabk_large_minpoly_sp}. The second follows from the triangle inequality, together with the two parts of Proposition \ref{prop_deltabk_small_minpoly_sp}.
\end{proof}

Let $\overline{f}\in\field_q[x]_{=2n}^\mon$ be a polynomial s.t. $\overline{f}\in I(\overline{g},2n-d)$. We define the discrepancy starting from $\overline{f}$ by 

$$
\Delta_{\overline{f},k}^g=\Pr[\charc(A)\in I(g,2n-d)|\charc(\overline{A})=\overline{f}]-q^{-d(k-1)}.
$$

\begin{lem}\label{lem_discrepancy_bound_interval_starting_from_f_mod_p_sp}
Let $\overline{f}\in \field_q[x]_{=2n}^\mon$ be a polynomial, and let $\mathcal{G}$ be a set of representatives for the possible lifts of $\overline{f}$ to a monic polynomial in $(\resring{k})[x]$ of degree $2n$, modulo equality in the $d$ next-to-leading coefficients.
\begin{enumerate}
    \item If $\delta(\deg\rad\overline{f})\ge d$, then $\sum_{g\in\mathcal{G}}|\Delta_{\overline{f},k}^g
    |=0$. 
    \item If $\delta(\deg\rad\overline{f})=s<d$, then $\sum_{g\in\mathcal{G}}|\Delta_{\overline{f},k}^g|\le 2\Pr_{B\in Sp_{2n}(\field_q)}[\delta(B)<d|\charc(B)=\overline{f}]$.
\end{enumerate}
\end{lem}

\begin{proof}
We have

\begin{equation*}\label{eqn_disc_sum_sp}
\Delta_{\overline{f},k}^g=\sum_{\substack{B\in Sp_{2n}(\field_q)\\ \charc(B)=\overline{f}}}\Delta_{B,k}^g\cdot \Pr[\overline{A}=B|\charc(\overline{A})=\overline{f}].    
\end{equation*}
By Proposition \ref{prop_deltabk_large_minpoly_sp} we conclude that 

$$
\Delta_{\overline{f},k}^g=\sum_{\substack{B\in Sp_{2n}(\field_q)\\ \charc(B)=\overline{f},\delta(B)<d}}\Delta_{B,k}^g\cdot \Pr[\overline{A}=B|\charc(\overline{A})=\overline{f}]
$$
Now using the triangle inequality and Proposition \ref{prop_summing_deltabbk_over_all_g_sp} we get the results. 
\end{proof}

\subsubsection{Discrepancy bound on distribution in intervals}

Similarly to the $GL_n$ case, we define for a monic polynomial $g\in(\resring{k})[x]$ of degree $2n$ the discrepancy of the surrounding interval,

$$
\Delta_{g}=\Pr[\charc A\in I(g,2n-d)]-q^{-kd}.
$$

\begin{lem}\label{lem_discrepancy_bound_single_interval_sp2n}
Let $d<n$, and let $\mathcal{G}$ be a set of representatives for the monic polynomials $g\in(\resring{k})[x]$ of degree $2n$, modulo equivalence in the $d$ next-to-leading coefficients. Then,

\begin{equation*}
    \sum_{g\in\mathcal{G}} |\Delta_g|=O\left(
    q^{d-\frac{n^2}{2d}+\frac{n}{2}}(1+\frac{1}{q^2-1})^n\binom{n+2d-1}{n}
    +q^{2d-\frac{n^2}{2d}+o(n)}
    \right).
\end{equation*}
\end{lem}

\begin{proof}
Write

\begin{multline*}
\Pr[\charc A\in I(g,2n-d)]=\sum_{\overline{f}\in I(\overline{g},2n-d)}\Pr[\charc A\in I(g,2n-d)|\charc\overline{A}=\overline{f}]\Pr[\charc\overline{A}=\overline{f}]=\\
=\sum_{\overline{f}\in I(\overline{g},2n-d)}\left[q^{-(k-1)d}+\Delta_{\overline{f},k}^g\right]\Pr[\charc\overline{A}=\overline{f}]=\\
=q^{-(k-1)d}\Pr[\charc\overline{A}\in I(\overline{g},2n-d)]+
\sum_{\overline{f}\in I(\overline{g},2n-d)}\Delta_{\overline{f},k}^g\cdot \Pr[\charc\overline{A}=\overline{f}].
\end{multline*}

By \cite[Theorem 6.21]{GR21},

$$
\Pr[\charc\overline{A}\in I(\overline{g},2n-d)]=q^{-d}+O\left(q^{-\frac{n^2}{2d}+\frac{n}{2}}\left(1+\frac{1}{q^2-1}\right)^n\binom{n+2d-1}{n}\right).
$$
Hence, 

$$
|\Delta_g|<O\left(q^{-(k-1)d-\frac{n^2}{2d}+\frac{n}{2}}\left(1+\frac{1}{q^2-1}\right)^n\binom{n+2d-1}{n}\right)+\sum_{\overline{f}\in I(\overline{g},2n-d)}|\Delta_{\overline{f},k}^g|\cdot \Pr[\charc\overline{A}=\overline{f}]
$$
Summing the last equation over all $g\in\mathcal{G}$, and using Lemma \ref{lem_discrepancy_bound_interval_starting_from_f_mod_p_sp}, we get that 

$$
\sum_{g\in\mathcal{G}}|\Delta_g|<O\left(q^{d-\frac{n^2}{2d}+\frac{n}{2}}\left(1+\frac{1}{q^2-1}\right)^n\binom{n+2d-1}{n}\right)+\sum_{\substack{\overline{f}\in\field_q[x]\\
\delta(\deg\rad \overline{f})<d}}2\Pr_{B\in Sp_{2n}(\field_q)}[\delta(B)<d|\charc(B)=\overline{f}]\Pr[\charc B=\overline{f}]
$$
Note that if $\delta(B)<d$ then $\deg\min(B)<2d$. Now by Theorem \ref{theorem_probability_small_min_poly_sp2n},

$$\Pr_{B\in Sp_{2n}(\field_q)}[\delta(B)<d|\charc(B)=\overline{f}]\Pr[\charc B=\overline{f}]<\Pr[\deg\min(B)<2d\cap \charc(B)=\overline{f}]<q^{-\frac{n^2}{2d}+o(n)},$$
and by Lemma \ref{lem_small_radical_bound}, the number of polynomials in $\field_{q}[x]_{=2n}$ with a radical of degree $<2d$ is $q^{2d+o(n)}$, hence we get the result.   
\end{proof}

\subsection{Joint distribution of traces}

As in the $GL_n$ case, the total variation distance of $\TR_d^{Sp}$ from $\UTR_d$  may be computed as the limit of the total variation distance of the $\resring{k}$ variables 
 $\TR_d^{Sp,k}$, $\UTR_d^k$. We recall that by the discussion of \cite[Theorem 6.21]{GR21}, There is a constant $c_q^{Sp}$ such that for $d<c_q^{Sp}\cdot n$,

$$
q^{d-\frac{n^2}{2d}+\frac{n}{2}}\left(1+\frac{1}{q^2-1}\right)^n\binom{n+2d-1}{n}=_{n\to\infty} o(1).
$$

\begin{theorem}
Let $d<n$, and let $\{a_i\}_{1\le i\le d}$ be a trace datum. Let $\TR_d^{Sp,k}$ be the random variable attaching to a matrix $M\in Sp_{2n}(\resring{k})$ its trace datum of length $d$, and let $\UTR_d^k$ be the uniform measure on trace data, as defined in Section \ref{section_trace_data_poly_intervals}. Then,

\begin{multline*}
d_{TV}(\mu_{\TR_d^{Sp,k}},\mu_{\UTR_d^k})=\sum_{\{a_i\}_{1\le i\le d}\text{ trace datum}}
\Bigg|
\Pr_{M\in Sp_{2n}(\resring{k})}[M\text{ has trace data }\{a_i\}_{1\le i\le d}]-q^{-(kd-S(d,k))}
\Bigg|
<\\
<O\left(
q^{d-\frac{n^2}{2d}+\frac{n}{2}}\left(1+\frac{1}{q^2-1}\right)^n\binom{n+2d-1}{n}+q^{2d-\frac{n^2}{2d}+o(n)}
\right).
\end{multline*}
\end{theorem}

\begin{proof}
We follow closely the proof of Theorem \ref{theorem_mod_pk_joint_traces}. Lemma \ref{lem_connection_traces_charpoly} gives that there are polynomials $\{g_j\}_{j=1}^{q^S}\subset\resring{k}[x]$, such that

$$
\Pr_{M\in Sp_{2n}(\resring{k})}[M\text{ has trace data }\{a_i\}_{1\le i\le d}]=\sum_{i=1}^{q^S}\Pr_{M\in Sp_{2n}(\resring{k})}[M\in I(g_j,n-d-1)]=q^{S-kd}+\sum_{i=1}^{q^S}\Delta_{g_j}.
$$
Denote $\mathcal{G}=\{g_j\}_{j=1}^{q^S}$. $\mathcal{G}$ is a set of representatives for the monic polynomials modulo equality in the $d$ next-to-leading coefficients (that is, a set of representatives for the Hayes equivalence relation $R_{1,d}$). Using the triangle inequality we get that 

\begin{equation*}
\sum_{\{a_i\}_{1\le i\le d}\text{ trace data}}
\Bigg|\Pr_{M\in Sp_{2n}(\resring{k})}[M\text{ has trace data }\{a_i\}_{1\le i\le d}]-q^{-(kd-S(d,k))}
\Bigg|<\sum_{g\in \mathcal{G}}|\Delta_g|.
\end{equation*}
Now the theorem follows from Lemma \ref{lem_discrepancy_bound_single_interval_sp2n}.
\end{proof}

By taking $k\to\infty$ we get Theorem \ref{main_theorem_sp} as a consequence (using Lemma \ref{lem_total_variation_dist_limit}). We restate it here for convenience. 

\begin{theorem}
Let $d<n$, and let $\TR_d^{Sp}$ be the random variable attaching to a matrix $M\in Sp_{2n}(\Ol)$ its trace datum of length $d$. Let $\UTR_d$ be the random variable $(a_i)_{i=1}^d$, such that $a_i\sim \frac{u_i}{|i|_p}$, where $u_i\sim \Ol$ is sampled uniformly at random w.r.t. the Haar measure. Then, 

$$
d_{TV}(\mu_{\TR_d^{Sp}}, \mu_{\UTR_d})=O\left(
q^{d-\frac{n^2}{2d}+\frac{n}{2}}\left(1+\frac{1}{q^2-1}\right)^n\binom{n+2d-1}{n}+q^{2d-\frac{n^2}{2d}+o(n)}
\right).
$$
In particular, for $d<c_q^{Sp}\cdot n$, this total variation distance is $o(1)$ as $n\to\infty$.

\end{theorem}
We get as a result

\begin{theorem}
Fix $d>0$, and let $B_1,\ldots,B_D$ be balls in $\mathcal{O}$. Let $Z_i\sim \frac{1}{|i|_p}U_i$, where $U_i$ are independent uniform random variables on $\mathcal{O}$. Then, 

$$
\lim_{n\to\infty}\Pr_{M\in Sp_{2n}(\Ol)}\left[
\left(
\tr(M^i)-1_{p|i}\sigma(\tr(M^{i/p}))
\right)_{i=1}^d
\right]=\prod_{i=1}^d \Pr[Z_i\in B_i].
$$
\end{theorem}

\section{Traces and characteristic polynomial for $SO_{n}$}\label{section_so_main}
For the rest of the section, $A$ would be a matrix in $SO_n(\resring{k})$ unless explicitly stated otherwise.

\subsection{The image of $\partial\charc_{A_0}$}

For the rest of the section, fix $\epsilon=\pm 1$ and work in $SO_n^\epsilon(\field_q)$. As we have done for the previous groups, we want to show that for any $A_0\in SO_n^\epsilon(\field_q)$, the map $\partial\charc_{A_0}:\mathfrak{so}_n^\epsilon(\field_q)\to\field_q[x]_{<n}$ behaves well under the orthogonal join. We first work with $O_n^\epsilon(\field_q)$ for convenience.

\begin{lem}\label{lem_derivative_of_ortho_sum}
Let $A\in O_m^\pm (\field_q)$, $B\in O_n^\pm(\field_q)$. Then, 
 
$$
\img(\partial\charc_{A\oplus_{\mathrm{orth}} B})= \mathrm{span}_{\field_q}(\charc(B)\cdot \img (\partial\charc_A),\charc(A)\cdot \img(\partial\charc_B)),
$$
where here $\oplus_{\mathrm{orth}}$ is the orthogonal join, as defined in Section \ref{section_so_generalities}. 
\end{lem}

\begin{proof}

Let $J_A$, $J_B$, $J_{A\oplus_{\mathrm{orth}} B}$ be the forms preserved by $A$, $B$ and $A\oplus_{\mathrm{orth}} B$ (respectively) and let $X$ be the matrix defined in equation \eqref{eqn_K_def}. Also let $Y\in \mathfrak{o}(\field_q,J_{A\oplus_{\mathrm{orth}} B})$. Then, using Claim \ref{claim_lie_algebra_of_sum}, setting $Y'=X^{-1}YX$, we have $Y'\in \mathfrak{o}(\field_q,\mathrm{diag}(J_A,J_B))$. Moreover,

$$
\partial\charc_{A\oplus_{\mathrm{orth}} B}(Y)=\tr(\Adj(x-A\oplus_{\mathrm{orth}} B)Y)=\tr(X^{-1}\Adj(x-\mathrm{diag}(A,B))X\cdot Y').
$$

A simple calculation shows that $\Adj(x-\mathrm{diag}(A,B))=\mathrm{diag}(\charc(B)\Adj(x-A),\charc(A)\Adj(x-B))$. Thus, using the fact that the trace is invariant under conjugation, we get

$$
\partial\charc_{A\oplus_{\mathrm{orth}} B}(Y)=\tr\left(
\left(
\begin{array}{cc}
    \charc(B)\Adj(x-A) & 0 \\
    0 & \charc(A)\Adj(x-B)
\end{array}
\right)
Y'
\right).
$$
Write $Y'=\left(
\begin{array}{cc}
    Y'_{11} & Y'_{12} \\
    Y'_{21} & Y'_{22}
\end{array}
\right)$. Note that without loss of generality, we may assume $Y'_{12},Y'_{21}=0$, since they do not contribute to the result. Now,
For $Y'$ to be in $\mathfrak{o}(\field_q,\mathrm{diag}(J_A,J_B))$ we simply need to have $Y'_{11}\in \mathfrak{o}(\field_q,J_A)$, $Y'_{22}\in\mathfrak{o}(\field_q,J_B)$, hence the result.
\end{proof}

The ideas in the previous Lemma can be used to prove the following

\begin{lem}\label{lem_congruent_forms_same_image}
Let $A\in O(\field_q, J_A)$, and let $J_B$ be a matrix of a form congruent to the form defined by $J_A$ over $\field_q$. Let $X$ be the matrix such that $J_B=X^t J_A X$, and define $B=X^{-1} AX$. Then, 

$$
\img(\partial\charc_A)=\img(\partial\charc_B).
$$
\end{lem}

\begin{proof}
We claim that $\mathfrak{o}(\field_q,J_B)=X^{-1}\mathfrak{o}(\field_q,J_A)X$. Indeed, if $Y\in\mathfrak{o}(\field_q,J_A)$ then

$$
J_B(X^{-1}YX)=X^{t}J_AYX=-X^{t}Y^tJ_AX=-(X^{-1}YX)^t(X^tJ_AX)=-(X^{-1}YX)^tJ_B,
$$
hence $ X^{-1}\mathfrak{o}(\field_q,J_A)X\subset \mathfrak{o}(\field_q,J_B)$, and comparing $\field_q$-dimensions we get equality. Now, let $Y\in \mathfrak{o}(\field_q,J_A)$. Then,

$$
\partial\charc_A(Y)=\tr(\Adj(x-A)Y)=\tr(X^{-1}\Adj(x-B)Y'X)=\tr(\Adj(x-B)Y'),
$$
where $B'\in \mathfrak{o}(\field_q,J_B)$. Hence we get the result. 
\end{proof}

The last lemma will be useful for us since it allows us to work only with Witt types $\langle 0\rangle$, $\langle 1\rangle$. Indeed, by passing to an extension of even degree $\field_{q^{2l}}/\field_q$ the Witt type $\langle 1,\delta\rangle$ becomes congruent to $\langle 0\rangle$ and $\langle \delta\rangle$ becomes congruent to $\langle 1\rangle$. 

\begin{theorem}\label{theorem_derivative_image_so}
Let $A\in SO_n^\epsilon(\field_q)$. Then,

$$
\img(\partial\charc_A)=\left\{
\frac{\charc A}{\min A}a:\deg A<\deg\min A, a \text{ is palindromic}
\right\}.
$$
\end{theorem}

\begin{proof}
By the usual argument, using Lemma \ref{lem_derivative_of_ortho_sum} it is enough to treat the matrices of types I, II, and III as defined in Section \ref{section_explicit_rep_so}. We will only treat the matrices of types II and III for Witt types $\langle 0 \rangle$, $\langle 1 \rangle$ as discussed above, reducing the other Witt types to those cases via Lemma \ref{lem_congruent_forms_same_image}, by passing to a field extension of even dimension $\field_{q^{2l}}$ in which the characteristic polynomial splits. 

\maincase{Type I} Let $A_1=\left(\begin{array}{cc}
    A & B \\
    C & D
\end{array}\right)\in \mathfrak{so}_n^\epsilon(\field_q)$, so that $A=-D^\mathfrak{t}$, and $B$, $C$ are $\mathfrak{t}$ anti-symmetric. Using Lemma \ref{lem_so_jordan_block_adj}, we get that 

\begin{multline*}
\partial\charc_{A_{\alpha,\alpha^{-1}}^m}(A_1)=\tr\left(
    \begin{array}{cc}
       (x-\alpha^{-1})^mX_\alpha^m  &  \\
         & (x-\alpha)^m (Y_\alpha^m)^\mathfrak{t}
    \end{array}\right)\left(\begin{array}{cc}
    A & B \\
    C & D
\end{array}\right)=\\
=\tr(A(x-\alpha^{-1})^mX_\alpha^m)+\tr(D(x-\alpha)^m (Y_\alpha^m)^\mathfrak{t})=\tr(A(x-\alpha^{-1})^mX_\alpha^m)+\tr(-A^\mathfrak{t}(x-\alpha)^m (Y_\alpha^m)^\mathfrak{t})=\\
=(x-\alpha^{-1})^m\tr(AX_\alpha^m)-(x-\alpha)^m\tr(AY_\alpha^m)=\tr(A[(x-\alpha^{-1})^{m}X_{\alpha}^m-(x-\alpha)^mY_\alpha^m]).
\end{multline*}
From here the result follows from the Type I case of the symplectic group, proved in Theorem \ref{theorem_image_char_derivative_sp2n}.

\maincase{Type II} This follows the same lines as the previous case, reducing to the Type II case of the symplectic group.

\maincase{Type III} 
The image is contained in $\field_q[x]_{<2m+1}^{\pal}$, a linear space of dimension $m$. To show equality, it is enough to find $m$ linearly independent functions. Let $1\le i,j\le m$ and write 
$$
X^{ij}=\left(
\begin{array}{ccc}
     0 & & X_{ij}\\
      & 0 & \\
      & & 0
\end{array}
\right)\in \mathfrak{o}^+(\field_q),(X_{ij})^\mathfrak{t}=-X_{ij}.
$$
More specifically we take $X_{ij}$ as the matrix whose $i,j$ entry is $1$ and whose $m+1-j,m+1-i$ entry is $-1$, and is zero elsewhere. Then we get (using the definition of $g_j$ from Proposition \ref{prop_matrix_comp_son})

$$
\partial\charc_{A_\alpha}(X^{ij})=(x-\alpha)^{2m+1}
\left(
\frac{(-1)^{m+1}}{2}\alpha-\frac{1}{x-\alpha}
\right)\left(
(-1)^{j-1}g_j(x)(x-\alpha)^{-m-1+i}-(-1)^{m-i}g_{m+1-i}(x)(x-\alpha)^{-j}
\right).
$$
If $m>i,j>1$ then we can compute the order at $\alpha$ of $\partial\charc_{A_\alpha}(X^{ij})$ by computing the order of

\begin{multline*}
\left(
\frac{(-1)^{m+1}}{2}\alpha+\frac{1}{x-\alpha}
\right)\left(
(-1)^{j-1}\frac{(-1)^{j-1}x^{j-2}}{(1-x\alpha)^j(x-\alpha)^{m+1-i}}
-(-1)^{m-i}\frac{(-1)^{m-i}x^{m-1-i}}{(1-x\alpha)^{m+1-i}(x-\alpha)^j}
\right)=\\
\left(
\frac{(-1)^m}{2}\alpha-\frac{1}{x-\alpha}
\right)\left(
\frac{(-\alpha)^jx^{j-2}-(-\alpha)^{m+1-i}x^{m-1-i}}
{(x-\alpha)^{m+1-i+j}}
\right).
\end{multline*}
We note that since

$$
(-\alpha)^j\alpha^{j-2}-(-\alpha)^{m+1-i}\alpha^{m-1-i}=(-1)^j-(-1)^{m+1-i},
$$
if $m+1-i\neq j\pmod{2}$, the order of the pole this function at $\alpha$ is $m+2-
i+j$. Thus the order of vanishing of $\partial\charc_{A_\alpha}(X^{ij})$ at $\alpha$ is $m-1+i-j$. This way we cover every odd order $1\le j\le 2m-4$.

If we take $A=0$, $B=B_j=\left(
\begin{array}{ccc}
     0_{m\times m} & u_j & 0_{m\times m} \\
     0_{1\times m} & 0  & v_j \\
     0_{m\times m} & 0_{m\times 1}  & 0_{m\times m}
\end{array}
\right)$, where $u_j$ is the $m\times 1$ vector with $1$ in the $j$-th coordinate and $v_j$ is the $1\times m$ vector with $-1$ in the $j$-th coordinate, we get that the following function is in the image:

$$
(x-\alpha)^{2m+1}\left( \frac{(-1)^{j-1}g_j}{x-\alpha}-\frac{x^j}{(1-x\alpha)^{j+1}}
\right).
$$
Similarly to the previous computation, we can get that if $j$ is odd then its order at $\alpha$ is $2m-j$. This way we cover all possible odd order $1\le j\le 2m-1$ so that the claim follows.
\end{proof}

\subsection{Distribution of the characteristic polynomial}\label{sec_chapoly_dist_ortho}

We follow the usual steps.

\subsubsection{One step from $\pi^{k-1}$ to $\pi^k$}

Let $A\in SO_{n}(\resring{k})$ be written as $A=A_0+\pi^{k-1}A_0A_1$, where $A_0 \in H_{G,k}$ is a representative in $SO_n(\resring{k})$ for $A\pmod{\pi^{k-1}}$, $A_1\in\mathfrak{so}(\field_q)$. When we fix $A_0$, the possibilities for $\charc(A)$ range over $\charc(A_0)+\pi^{k-1}\img(\partial\charc_{A_0})$. Recall that from Theorem \ref{theorem_image_char_derivative_sp2n}, 

$$
\img(\partial\charc_{A_0})=\left\{a:a=\frac{\charc\overline{A_0}}{\min\overline{A_0}}b, b\in\field_q[x]_{<\deg\min\overline{A_0}}^{\pal}
\right\}.
$$
Since $\frac{\charc\overline{A_0}}{\min\overline{A_0}}$ is palindromic, the condition on $a$ being palindromic is equivalent to $b$ being palindromic. We use this to prove the following lemma, on the distribution of $\charc A$ given $A_0$.

\begin{lem}\label{lem_one_step_son}
Let $g\in (\resring{k})[x]$ be a monic polynomial of degree $n$. Then for $d<n$,

$$
\Pr[\charc A\in I(g,n-d)|\charc(\rho_{k-1}(A))\in I(\rho_{k-1}(g),n-d),\delta(\deg\min\overline{A})\ge d]=q^{-d}.
$$
\end{lem}

\begin{proof}
The proof is essentially the same as that of Lemma \ref{lem_one_step_sp2n}, making adjustments for using claims for the special orthogonal group wherever necessary.
\end{proof}

\subsubsection{Distribution mod $\pi^k$ given mod $\pi$ behavior}
Define the discrepancy of the conditional distribution given a matrix $B\in SO_{n}(\field_q)$ such that $\charc(B)\in I(\overline{g},n-d)$ to be 

$$
\Delta_{B,k}^g=\Pr[\charc(A)\in I(g,n-d)|\overline{A}=B]-q^{-d(k-1)}.
$$
As in the previous cases, when $\deg\min B$ is large enough, the discrepancy is zero.

\begin{prop}\label{prop_deltabk_large_minpoly_so}
Let $g\in (\resring{k})[x]$ be a monic polynomial of degree $n$, $d<\frac{n}{2}$. For every matrix $B\in SO^\epsilon_{
n}(\field_q)$ with $\charc(B)\in I(\overline{g},n-d)$ and $\delta(\deg\min(B))\ge d$, we have $\Delta_{B,k}^g=0$.
\end{prop}

\begin{proof}
The proof is essentially the same as that of Proposition \ref{prop_deltabk_large_minpoly}, using Theorem \ref{theorem_derivative_image_so}.
\end{proof}

Now we consider how $\charc(A)$ behaves for lifts $A$ of $B$, when $\delta(B)<d$.

\begin{prop}\label{prop_deltabk_small_minpoly_son}
Let $g\in(\resring{k})[x]$ be a monic polynomial of degree $n$. Let $B\in SO_n(\field_q)$ be a matrix with $\deg\min B=l<d$. Let $\mathcal{P}_B^k=\{\charc(A):A\in SO_n(\resring{k}),\overline{A}=B\}$. Then,

\begin{enumerate}
    \item $|\mathcal{P}_B^k|=q^{(k-1)l}$. Moreover, each element of $\mathcal{P}_B^k$ has different $l$ leading coefficients modulo $\pi^k$. 
    \item $
\Pr[\charc(A)\in I(g,n-d)|\overline{A}=B]=1_{\mathcal{P}_B^k\cap I(g,n-d)\neq \emptyset}q^{-(k-1)l}
$. Equivalently, we have
$$\Delta_{B,k}^g=1_{\mathcal{P}_B^k\cap I(g,n-d)\neq \emptyset}q^{-(k-1)l}-q^{-(k-1)d}.$$
\end{enumerate}
\end{prop}

\begin{proof}
    The proof is omitted as it is essentially the proof of Proposition \ref{prop_deltabk_large_minpoly}.
\end{proof}

We now let $\mathcal{G}$ be a set of representatives for the lifts of $\charc(B)$ to a monic polynomial $g\in(\resring{k})[x]$ of degree $n$, modulo equality in the $d$ next-to-leading coefficients. We note that $|\mathcal{G}|=q^{(k-1)d}$.

\begin{prop}\label{prop_summing_deltabbk_over_all_g_so}
Let $B\in SO_{n}(\field_q)$. Then,

\begin{enumerate}
    \item If $\deg\min(B)\ge d$, we have $\sum_{g\in G}|\Delta_{B,k}^g|=0$.
    \item If $\deg\min(B)=l<d$, we have $\sum_{g\in\mathcal{G}}|\Delta_{B,k}^g|\le 2$.
\end{enumerate}
\end{prop}

\begin{proof}
The first claim follows directly from Proposition \ref{prop_deltabk_large_minpoly_so}. The second follows from the triangle inequality using the two parts of Proposition \ref{prop_deltabk_small_minpoly_son}.
\end{proof}

Let $\overline{f}\in\field_q[x]_{=n}^\mon$ be a self-reciprocal polynomial such that $\overline{f}\in I(\overline{g},n-d)$. We define the discrepancy of starting from $f$ by 

$$
\Delta_{\overline{f},k}^g=\Pr[\charc(A)\in I(g,n-d)|\charc(\overline{A})=\overline{f}]-q^{-d(k-1)}.
$$

\begin{lem}\label{lem_discrepancy_bound_interval_starting_from_f_mod_p_so}
Let $\overline{f}\in \field_q[x]_{=n}^\mon$ be a self-reciprocal polynomial, and let $\mathcal{G}$ be a set of representatives for the possible lifts of $\overline{f}$ to a monic polynomial in $(\resring{k})[x]$ of degree $n$, modulo equality in the $d$ next-to-leading coefficients.
\begin{enumerate}
    \item If $\delta(\deg\rad\overline{f})\ge d$, then $\sum_{g\in\mathcal{G}}|\Delta_{\overline{f},k}^g
    |=0$. 
    \item If $\delta(\deg\rad\overline{f})=s<d$, then $\sum_{g\in\mathcal{G}}|\Delta_{\overline{f},k}^g|\le 2\Pr_{B\in SO_{n}(\field_q)}[\delta(\deg\min(B))<d|\charc(B)=\overline{f}]$.
\end{enumerate}
\end{lem}

\begin{proof}
The proof is the same as that of Lemma \ref{lem_discrepancy_bound_interval_starting_from_f_mod_p}, using Propositions \ref{prop_deltabk_large_minpoly_so}, \ref{prop_deltabk_small_minpoly_son}. 
\end{proof}

\subsubsection{Discrepancy bound on distribution in intervals}

As usual we define we define for a monic polynomial $g\in(\resring{k})[x]$ of degree $n$ the discrepancy of the surrounding interval,

$$
\Delta_{g}=\Pr[\charc A\in I(g,n-d)]-q^{-kd}.
$$

\begin{lem}\label{lem_discrepancy_bound_single_interval_son}
Let $d<n$, and let $\mathcal{G}$ be a set of representatives for the monic polynomials $g\in(\resring{k})[x]$ of degree $n$, modulo equivalence in the $d$ next-to-leading coefficients. Then,

\begin{equation*}
    \sum_{g\in\mathcal{G}} |\Delta_g|=O\left(q^{d+\frac{n}{4}-\frac{n^2}{8d}}\left(1+\frac{1}{q^2-1}\right)^{n/2}\binom{n/2+2d-3}{n/2}+q^{2d+\frac{n}{2}-\frac{n^2}{4d}+o(n)}
    \right).
\end{equation*}
\end{lem}

\begin{proof}
Let $g\in\mathcal{G}$. As in the proof of Lemma \ref{lem_average_discrepancy_bound_gln}, write 

\begin{equation*}
\Pr[\charc A\in I(g,n-d)]=q^{-(k-1)d}\Pr[\charc\overline{A}\in I(\overline{g},n-d)]+
\sum_{\overline{f}\in I(\overline{g},n-d)}\Delta_{\overline{f},k}^g\cdot \Pr[\charc\overline{A}=\overline{f}].
\end{equation*}
By \cite[Theorem 7.9]{GR21} we have 

$$
\Pr_{B\in O_n(\field_q)}[\charc B\in I(\overline{g},n-d)]=q^{-d}+O\left(
q^{\frac{n}{4}-\frac{n^2}{8d}}\left(1+\frac{1}{q^2-1}\right)^{n/2}\binom{n/2+2d-3}{n/2}
\right).
$$
Since $[O_n(\field_q):SO_n(\field_q)]=2$ we get the same bound for $SO_n(\field_q)$, therefore 

$$
\Pr[\charc\overline{A}\in I(\overline{g},n-d)]=q^{-d}+O\left(
q^{\frac{n}{4}-\frac{n^2}{8d}}\left(1+\frac{1}{q^2-1}\right)^{n/2}\binom{n/2+2d-3}{n/2}
\right).
$$
Hence, 

$$
|\Delta_g|<O\left(
q^{-(k-1)d+\frac{n}{4}-\frac{n^2}{8d}}\left(1+\frac{1}{q^2-1}\right)^{n/2}\binom{n/2+2d-3}{n/2}
\right)+\sum_{\overline{f}\in I(\overline{g},n-d)}\Delta_{\overline{f},k}^g\cdot \Pr[\charc\overline{A}=\overline{f}].
$$
Summing this over all $g\in\mathcal{G}$, and using Lemma \ref{lem_discrepancy_bound_interval_starting_from_f_mod_p_so}, we get that 

\begin{multline*}
\sum_{g\in\mathcal{G}}|\Delta_g|
<O\left(
q^{d+\frac{n}{4}-\frac{n^2}{8d}}\left(1+\frac{1}{q^2-1}\right)^{n/2}\binom{n/2+2d-3}{n/2}
\right)+\\
+\sum_{\substack{\overline{f}\in\field_q[x]\\
\delta(\deg\rad \overline{f})<d}}2\Pr_{B\in SO_{n}(\field_q)}[\delta(\deg\min(B))<d|\charc(B)=\overline{f}]\Pr[\charc{B}=\overline{f}].
\end{multline*}
Now by Theorem \ref{theorem_probability_small_min_poly_son},

$$\Pr_{B\in SO_{n}(\field_q)}[\delta(B)<d|\charc(B)=\overline{f}]\Pr[\charc{B}=\overline{f}]<\Pr_{B\in SO_n(\field_q)}[\deg\min(B)<2d\cap \charc(B)=\overline{f}]<q^{\frac{n}{2}-\frac{n^2}{4d}+o(n)}.$$
By Lemma \ref{lem_small_radical_bound}, the number of polynomials in $\field_{q}[x]_{=n}^\mon$ with a radical of degree $<2d$ is $q^{2d+o(n)}$, hence we get the result.   
\end{proof}

\subsection{Joint distribution of traces}

As in the $GL_n$ case, the total variation distance of $\TR_d^{SO}$ from $\UTR_d$  may be computed as the limit of the total variation distance of the $\resring{k}$ random variables 
 $\TR_d^{SO,k}$, $\UTR_d^k$. We recall that by the discussion of \cite[Theorem 7.9]{GR21} (and using a correction as in the previous section to move from $O_n(\field_q)$ to $SO_n(\field_q)$), There is a constant $c_q^{SO}$ such that for $d<c_q^{SO}\cdot n$,

$$
q^{d+\frac{n}{4}-\frac{n^2}{8d}}\left(1+\frac{1}{q^2-1}\right)^{n/2}\binom{n/2+2d-3}{n/2}=o(1).
$$

\begin{theorem}
Let $d<\frac{n}{2}$, and let $\{a_i\}_{1\le i\le d}$ be a sequence of trace data. Let $\TR_d^{SO,k}$ be the random variable attaching to a matrix $M\in SO_{n}(\resring{k})$ its trace data up to $d$, and let $\UTR_d^k$ be the uniform measure on trace sequences, as defined in Section \ref{section_trace_data_poly_intervals}. Then,

\begin{multline*}
d_{TV}(\mu_{\TR_d^{SO,k}},\mu_{\UTR_d^k})=\sum_{\{a_i\}_{1\le i\le d}\text{ trace data}}
\Bigg|\Pr_{M\in SO_{n}(\resring{k})}[M\text{ has trace data }\{a_i\}_{1\le i\le d}]-q^{-(kd-S(d,k))}
\Bigg|<\\
<O\left(q^{d+\frac{n}{4}-\frac{n^2}{8d}}\left(1+\frac{1}{q^2-1}\right)^{n/2}\binom{n/2+2d-3}{n/2}+q^{2d+\frac{n}{2}-\frac{n^2}{4d}+o(n)}
    \right).
\end{multline*}
\end{theorem}

\begin{proof}
The proof is omitted, as it is essentially the same as the analogous one for $GL_n(\Ol)$. Note that depending on $q$, $SO_n(\Ol)$ may reduce to either $SO_n^+(\field_q)$ or $SO_n^-(\field_q)$, which is why we treated both groups.
\end{proof}

By taking $k\to\infty$ and using Lemma \ref{lem_total_variation_dist_limit} we get as a consequence Theorem \ref{main_theorem_so}. As a corollary we get

\begin{theorem}
Fix $d>0$, and let $B_1,\ldots,B_d$ be fixed balls in $\mathcal{O}$. Let $Z_i\sim \frac{1}{|i|_p}U_i$, where $U_i$ are independent uniform random variables on $\mathcal{O}$. Then, 

$$
\lim_{n\to\infty}\Pr_{M\in SO_n(\Ol)}\left[
\left(
\tr(M^i)-1_{p|i}\sigma(\tr(M^{i/p}))\in B_i
\right)_{i=1}^d
\right]=\prod_{i=1}^d \Pr[Z_i\in B_i].
$$
\end{theorem}

\section{Traces and characteristic polynomial for $U_n$}\label{section_un_main}

We recall that we fixed an unramified quadratic extension $\mathcal{K}/\mathcal{F}$ with a ring of integers $\mathcal{R}$ and that the unitary group $U_n(\Ol)$ is a subgroup of $GL_n(\mathcal{R})$.

\subsection{The image of $\partial\charc_{A_0}$}

As we saw in Lemma \ref{lem_charpoly_step_in_p_direction}, the map $\partial\charc_{A_0}:\mathfrak{u}_n(\field_q)\to \field_{q^2}[x]$, defined by $\partial\charc_{A_0}(A_1)=\tr(\Adj(x-A_0)A_0A_1)$,
is essential to understanding the distribution of the characteristic polynomial. We now compute its image explicitly.

\begin{lem}\label{lem_image_of_direct_sum_un}
Let $A\in U_n(\field_q)$, $B\in U_m(\field_q)$. Then 
$$\img(\partial\charc_{A\oplus B})=\mathrm{span}_{\field_q}(\charc(B)\cdot \img(\partial\charc_{A}),\charc(A)\cdot \img(\partial\charc_B)).$$
\end{lem}

\begin{proof}
Let $C_1\in \mathfrak{u}_{n+m}(\field_q)$, and write $C_1$ as a block matrix 
$
C_1=\left(
\begin{array}{cc}
    X & Y \\
    Z & W
\end{array}
\right)
$.
We have 

\begin{multline*}
\partial\charc_{A\oplus B}(C_1)=\tr\left(
\Adj\left(
x-\left(
\begin{array}{cc}
    A &  \\
     & B 
\end{array}
\right)
\right)\left(
\begin{array}{cc}
    A &  \\
     & B 
\end{array}
\right)\left(
\begin{array}{cc}
    X & Y \\
    Z & W
\end{array}
\right)
\right)=\\
=\tr\left(
\left(
\begin{array}{cc}
    \charc(B)\Adj(x-A)A &  \\
     & \charc(A)\Adj(x-B)B
\end{array}
\right)\left(
\begin{array}{cc}
    X & Y \\
    Z & W
\end{array}
\right)
\right)=\\
=\charc(B)\tr(\Adj(x-A)AX)+\charc(A)\tr(\Adj(x-B)BW)=\charc(B)\partial\charc_{A}(X)+\charc(A)\partial\charc_{B}(W).
\end{multline*}
\end{proof}

\begin{claim}
Let $A_0\in U_n(\field_q)$ and let $A_1\in \mathfrak{u}_n(\field_q)$. Let $g_0=\charc(A_0)$, and let $j\in \field_{q^2}$ be an element such that $\tr_{\field_{q^2}/\field_q}(j)=0$. Define $J_{A_0}(x)=J(x)=j(g_0-x^n+g_0(0))$. Then, the image of $\partial\charc_{A_0}$ is contained in $\field_q\cdot J_{A_0}(x)\oplus \field_{q^2}[x]_{<n}^{g_0(0)-\skpal}$.
\end{claim}

\begin{proof}
To compute the symmetries of the image, we work over $\resring{2}$. Abusing notation slightly we let $A_0\in U_n(\resring{2})$ be a lift of $A_0$, and write $A=A_0+\pi A_0A_1$. By Lemma \ref{lem_charpoly_step_in_p_direction}, we get
$$\charc(A)=\charc(A_0)+\pi\partial\charc_{A_0}(A_1).$$
Let $g_0=\charc(A_0)$, $g_1=\partial\charc_{A_0}(A_1)$. We have $\widetilde{\charc(A)}=\charc(A)$, that is, 
\begin{equation}\label{eqn_tilde_explicit}
\frac{g_0^*(x)+\pi g_1^*}{\sigma(g_0(0))+\pi \sigma(g_1(0))}=g_0+\pi g_1.    
\end{equation}

We now focus on the LHS of \eqref{eqn_tilde_explicit}. Note that $\frac{1}{\sigma(g_0(0))+\pi\sigma(g_1(0))}=\frac{1}{\sigma(g_0(0))}-\pi\frac{\sigma(g_1)(0)}{\sigma(g_0)(0)^2}
$. Hence, LHS of \eqref{eqn_tilde_explicit} is 
$$
\tilde{g_0}+\pi\left(\frac{g_1^*}{\sigma(g_0(0))}-
\tilde{g_0}\frac{\sigma(g_1(0))}{\sigma(g_0(0))}
\right).
$$
Using the fact that $\frac{g_0^*}{\sigma(g_0(0))}=\tilde{g_0}=g_0$, we get that the LHS of \eqref{eqn_tilde_explicit} is actually 
$$
g_0+\pi\left(
\frac{g_1^*}{\sigma(g_0(0))}-
g_0\frac{\sigma(g_1(0))}{\sigma(g_0(0))}
\right).
$$
Substituting this back into equation \eqref{eqn_tilde_explicit} we get (again by abusing notation and moving to mod $\pi$)
$$
\left(\frac{g_1^*}{\sigma(g_0(0))}-
g_0\frac{\sigma(g_1(0))}{\sigma(g_0(0))}
\right)=g_1.
$$

Here $g_1(x)=\tr(A_1)\charc(A_0)-x\tr(\Adj(x-A_0)A_1)$, hence $g_1(0)/g_0(0)=\tr(A_1)$ and so $\sigma(g_1(0))/\sigma(g_0(0))=-\tr(A_1)$. Note that since $\tr(A_1)$ is an element of $\field_{q^2}$ with $\tr_{\field_{q^2}/\field_{q}}(\tr(A_1))=0$, we can write $\tr(A_1)=\alpha j$ for $\alpha\in\field_q$. Write $g_1=\alpha jg_0+f$.  The last equation above becomes  
$$
\frac{f^*}{\sigma(g_0(0))}=-\alpha jg_0+\frac{f^*}{\sigma(g_0(0))}+\alpha jg_0=\alpha jg_0+f.
$$
Hence, recalling that $\frac{g_0^*}{\sigma(g_0)}=g_0$, we get that
$$
\frac{(f+\frac{1}{2}\alpha jg_0)^*}{\sigma(g_0(0))}=f+\frac{1}{2}\alpha jg_0.
$$
That means that $f+\frac{1}{2}\alpha jg_0$ is $g_0(0)$-skew palindromic (w.r.t. $n$). We note that since $g_1$ has degree $<n$, and $g_0$ is monic, $f+\frac{1}{2}\alpha jg_0$ is of degree precisely $n$ and has leading coefficient $-\frac{1}{2}\alpha j$. Hence, $g_1=h_{\mathrm{temp}}+\frac{1}{2}\alpha jg_0$ where $h_{\mathrm{temp}}$ is in $\field_{q^2}[x]_{=n}^{g_0(0)-\skpal}$ has leading coefficient $-\frac{1}{2}\alpha j$. Since $h_{\mathrm{temp}}$ is $g_0(0)$-skew palindromic, this means that its free coefficient is $\frac{1}{2}\alpha jg_0(0)$. Write $h_{\mathrm{temp}}=-\frac{1}{2}\alpha jx^n+\frac{1}{2}\alpha j g_0(0)+h$, where $h$ is in $\field_{q^2}[x]^{g_0(0)-\skpal}_{<n}$. Then, we eventually get that $g_1=h+\frac{1}{2}\alpha j(g_0-x^n+g_0(0))$. Since $\alpha$ is arbitrary, we get the result. 
\end{proof}

From now on, $A_0$ (and hence $g_0$) are fixed; thus, whenever we speak about a skew palindromic polynomial, it will be $g_0(0)$-skew palindromic. Therefore we use the notations $\field_{q^2}[x]_{<n}^{\skpal}$, etc. instead of the more detailed but longer notation. Now we will compute the image of $\partial\charc_{A_0}$ inside $\field_q\cdot J_{A_0}(x)\oplus \field_{q^2}[x]_{<n}^{\skpal}$. We first do this for matrices $A_0\in U_n(\field_q)\cap GL_n(\field_q)$ and their conjugates (as for the other groups, the image of $\mathcal{L}_{A_0}$ is invariant under conjugation of $A_0$). 

\begin{claim}
    Let $A_0\in U_n(\field_q)\cap GL_n(\field_q)$ be a matrix with $\charc(A_0)=\min(A_0)$. Then, 
    $$\img(\partial\charc_{A_0})=\field_q\cdot J_{A_0}(x)\oplus \field_{q^2}[x]_{<n}^{\skpal}.$$ 
\end{claim}

\begin{proof}
We note that by definition, $\partial\charc_{A_0}(A_1)=\tr(\Adj(x-A_0)A_0A_1)$. We note that if we put $A_1\in M_n(\field_q)$, we get the same result for this map as for the map defined for $GL_n(\field_q)$, which was discussed in Theorem \ref{theorem_image_char_derivative_gl_n}. There, we proved that the image has dimension $n$. We also note that $M_n(\field_q)$ is the direct sum of the symmetric matrices and anti-symmetric matrices over $\field_q$, $M_n(\field_q)=\mathrm{sym}_n(\field_q)\oplus \mathrm{asym}_n(\field_q)$. Note that if $A_1\in \mathrm{asym}_n(\field_q)$ then $A_1\in \mathfrak{u}_n(\field_q)$. Also, if $0\ne j\in \field_{q^2}$ is such that $\tr_{\field_{q^2}/\field_q}(j)=0$, and $A_1'\in \mathrm{sym}_n(\field_q)$, then $A_1=jA_1'\in \mathfrak{u}_n(\field_q)$. We get that $\dim \img(\partial\charc_{A_0})\ge n$, hence the result.  
\end{proof}

Now we prove our main result for this section.

\begin{theorem}\label{theorem_image_char_derivative_un}
Let $A_0\in U_n(\field_q)$. Then, $\img(\partial\charc_{A_0})=\field_q\cdot J_{A_0}(x)\oplus \frac{\charc(A_0)}{\min(A_0)}\cdot \field_{q^2}[x]_{<\deg{\min{A_0}}}^{\skpal}$.
\end{theorem}

\begin{proof}
We first assume that $A_0$, $A_0^*$ do not share eigenvalues over the algebraic closure and that $\charc(A_0)=\min(A_0)$. In this case, let $B_0=A_0\oplus A_0^*$. This matrix is conjugate to a matrix in $U_n(\field_q)\cap GL_n(\field_q)$, and $\charc(B_0)=\min(B_0)$. Hence $\img(\partial\charc_{B_0})$ is of dimension $2n$. However by Lemma \ref{lem_image_of_direct_sum_un}, $\dim \img(\partial\charc_{B_0})=\dim \img(\partial\charc_{A_0})+\dim \img(\partial\charc_{A_0^*})$. It follows that $\dim \img(\partial\charc_{A_0})=n$, hence the result follows by comparing dimensions.

For the more general case, we proceed by induction as in the proof of Theorem \ref{theorem_image_char_derivative_gl_n}, each time removing a different primary part, and using Lemma \ref{lem_image_of_direct_sum_un}.
\end{proof}

\subsection{Distribution of the characteristic polynomial}

We follow the usual steps. In the rest of the section, $A$ will always be a matrix in $U_n(\resring{k})$ unless explicitly stated otherwise.

\subsubsection{One step from $\pi^{k-1}$ to $\pi^k$}

Let $A\in U_{n}(\resring{k})$ be written as $A=A_0+\pi^{k-1}A_0A_1$, where $A_0\in H_{U,k}$ is a representative for $A\pmod{\pi^{k-1}}$. When we fix $A_0$, the possibilities for $\charc(A)$ ranges over $\charc(A_0)+\pi^{k-1} \img(\partial\charc_{\overline{A_0}})$. Recall that from Theorem \ref{theorem_image_char_derivative_un}, 

$$
\img(\partial\charc_{A_0})=\field_q\cdot J_{A_0}(x)\oplus \frac{\charc\overline{A_0}}{\min \overline{A_0}}\cdot \field_{q^2}[x]_{<\deg{\min{A_0}}}^{\skpal}.
$$

\begin{lem}\label{lem_one_step_un}
Let $g\in (\resringr{k})[x]$ be a monic polynomial of degree $n$, and let $\delta=\delta(n)=\lceil\frac{n-1}{2}\rceil$. Then for $d<\frac{n}{2}$,

$$
\Pr[\charc A\in I(g,n-d)|\charc(\rho_{k-1}(A))\in I(\rho_{k-1}(g),n-d),\delta(\deg\min\overline{A})> d]=q^{-2d}.
$$
\end{lem}

\begin{proof}
We show that for any possible choice of $A_0$ with $\delta(\deg\min\overline{A_0})> d$, the conditional probability is $q^{-2d}$, from which the result follows.

We are given that $\charc(\rho_{k-1}(A_0))\in I(\rho_{k-1}(g),n-d)$, hence $\charc(A_0)-g=e_0(x)+\pi^k e_1(x)$, where $\deg e_0(x)<n-d$. Thus to have $\charc(A)\in I(g,n-d)$ we actually need to have 

\begin{equation}\label{eqn_condition_on_derivative_unitary}
\deg\left(
\overline{e_1}-\partial\charc_{\overline{A_0}}(A_1)
\right)<n-d.    
\end{equation}
We note that since $\charc(A_0),g$ are monic we have $\deg\overline{e_1}<n$. By Theorem \ref{theorem_image_char_derivative_un}, there are some $\alpha\in\field_q$, $b\in\field_{q^2}[x]^{\skpal}_{<\deg\min A_0}$, such that $\partial\charc_{\overline{A_0}}(A_1)=\alpha J_{A_0}(x)+\frac{\charc\overline{A_0}}{\min \overline{A_0}}b$. Hence the condition \eqref{eqn_condition_on_derivative_unitary} is equivalent to 
$$
\deg\left(
(\overline{e_1}-\alpha J_{A_0}(x))-\frac{\charc\overline{A_0}}{\min \overline{A_0}}b
\right)
<n-d. 
$$
Since both polynomials are of degree $<n$, and since $x^{\deg\min(A_0)}\frac{\charc\overline{A_0}}{\min \overline{A_0}}$ is monic of degree $n$, we see that \eqref{eqn_condition_on_derivative_unitary} is equivalent to demanding 
$$
x^{\deg\min\overline{A_0}}\frac{\charc\overline{A_0}}{\min\overline{A_0}}+\frac{\charc\overline{A_0}}{\min \overline{A_0}}b\in I\left(
x^{\deg\min\overline{A_0}}\frac{\charc\overline{A_0}}{\min\overline{A_0}}+(\overline{e_1}-\alpha J_{A_0}(x)),n-d
\right).
$$

For polynomials $r,s\in\field_q[x]$, denote by $1_{r,s}$ the event that 

$$
x^{\deg\min\overline{A_0}}\frac{\charc\overline{A_0}}{\min\overline{A_0}}+r\in I\left(
x^{\deg\min\overline{A_0}}\frac{\charc\overline{A_0}}{\min\overline{A_0}}+s,n-d
\right).
$$

We compute the probability for \eqref{eqn_condition_on_derivative_unitary} to hold, using Hayes characters. By the discussion above, it is equal to 

\begin{multline*}
\Pr_{A_1\in \mathfrak{u}_{n}(\field_q)}
\left[
1_{\frac{\charc(A_0)}{\min(A_0)}b,\overline{e_1}-\alpha J_{A_0}(x)}
\right]=
q^{-\delta(\deg\min(\overline{A_0}))-1}
\sum_{\alpha\in\field_q}\sum_{\substack{b\in\field_{q^2}[x]_{<\deg\min(\overline{A_0})}^{\skpal}}}1_{\frac{\charc(A_0)}{\min(A_0)}b,\overline{e_1}-\alpha J_{A_0}(x)}=\\
    =q^{-\delta(\deg\min\overline{A_0})-1}
    \sum_{\alpha\in\field_q}
    \sum_{\substack{ b\in\field_{q^2}[x]_{<\deg\min(\overline{A_0})}^{\skpal}}}q^{-2d}\sum_{\chi\in R_{d,1}}\chi\left(\frac{\charc\overline{A}_0}{\min\overline{A}_0}(x^{\deg\min\overline{A_0}}+b)\right)\overline{\chi}\left(
    \frac{\charc\overline{A}_0}{\min\overline{A}_0}x^{\deg\min\overline{A_0}}+\overline{e_1}-\alpha J_{A_0}(x)
    \right)=\\
    =q^{-2d-1}\sum_{\chi\in R_{d,1}\atop{\alpha\in\field_q}}\chi\left(\frac{\charc\overline{A}_0}{\min\overline{A}_0}\right)\overline{\chi}\left(
    \frac{\charc\overline{A}_0}{\min\overline{A}_0}x^{\deg\min\overline{A_0}}+\overline{e_1}-\alpha J_{A_0}(x)
    \right)
q^{-\delta(\deg\min(\overline{A_0}))}\sum_{\substack{ b\in\field_{q^2}[x]_{<\deg\min(\overline{A_0})}^{\skpal}}}\chi(x^{\deg\min\overline{A_0}}+b).
\end{multline*}
We note that by Lemma \ref{lem_hayes_character_sum_over_skpalindromic_polys}, since $d<\min\overline{A_0}$, we have
$$
\sum_{\substack{ b\in\field_{q^2}[x]_{<\deg\min\overline{A_0}}^{\skpal}}}\chi(x^{\deg\min\overline{A_0}}+b)=
q^{\delta(\deg\min(\overline{A}_0))}1_{\chi=\chi_0}.
$$
Substituting this into the previous equation we get the result.
\end{proof}

\subsubsection{Distribution mod $\pi^k$ given mod $\pi$ behavior}
Let $g\in(\resringr{k})[x]$ be a monic polynomial of degree $n$. Define the discrepancy of the conditional distribution given a matrix $B\in U_n(\field_q)$ such that $\charc(B)\in I(\overline{g},n-d)$ to be 

$$
\Delta_{B,k}^g=\Pr[\charc(A)\in I(g,n-d)|\overline{A}=B]-q^{-2d(k-1)}.
$$
As in the previous cases, when $\deg\min B$ is large enough, the discrepancy is zero.

\begin{prop}\label{prop_deltabk_large_minpoly_un}
Let $g\in (\resringr{k})[x]$ be a monic polynomial of degree $n$, $d<\frac{n}{2}$. For every matrix $B\in U_{n}(\field_q)$ with $\charc(B)\in I(\overline{g},n-d)$ and $\delta(\deg\min(B))\ge d$, we have $\Delta_{B,k}^g=0$.
\end{prop}

\begin{proof}
The proof is essentially the same as that of Proposition \ref{prop_deltabk_large_minpoly}, using Theorem \ref{theorem_image_char_derivative_un}.
\end{proof}

Now we consider how $\charc(A)$ behaves for lifts $A$ of $B$, when $\delta(B)<d$.

\begin{prop}\label{prop_deltabk_small_minpoly_un}
Let $g\in(\resringr{k})[x]$ be a monic polynomial of degree $n$. Let $B\in U_n(\field_q)$ be a matrix with $\delta(\deg\min B)=l<d$. Let $\mathcal{P}_B^k=\{\charc(A):A\in U_n(\mathcal{O}/\pi^k\mathcal{O}),\overline{A}=B\}$. Then,

\begin{enumerate}
    \item $|\mathcal{P}_B^k|=q^{2(k-1)l}$. Moreover, each element of $\mathcal{P}_B^k$ has different $l$ leading coefficients modulo $\pi^k$. 
    \item $
\Pr[\charc(A)\in I(g,n-d)|\overline{A}=B]=1_{\mathcal{P}_B^k\cap I(g,n-d)\neq \emptyset}q^{-2(k-1)l}
$. Equivalently, we have
$$\Delta_{B,k}^g=1_{\mathcal{P}_B^k\cap I(g,n-d)\neq \emptyset}q^{-2(k-1)l}-q^{-2(k-1)d}.$$
\end{enumerate}
\end{prop}

\begin{proof}
    The proof is omitted as it is essentially the proof of Proposition \ref{prop_deltabk_large_minpoly}.
\end{proof}

We now let $\mathcal{G}$ be a set of representatives for the lifts of $\charc(B)$ to a monic polynomial $g\in(\resringr{k})[x]$, modulo equality in the $d$ next-to-leading coefficients. We note that $|\mathcal{G}|=q^{2(k-1)d}$.

\begin{prop}\label{prop_summing_deltabbk_over_all_g_un}
Let $B\in U_{n}(\field_q)$. Then,

\begin{enumerate}
    \item If $\delta(\deg\min(B))\ge d$, we have $\sum_{g\in \mathcal{G}}|\Delta_{B,k}^g|=0$.
    \item If $\delta(\deg\min(B))=l<d$, we have $\sum_{g\in\mathcal{G}}|\Delta_{B,k}^g|\le 2$.
\end{enumerate}
\end{prop}

\begin{proof}
The first claim follows directly from Proposition \ref{prop_deltabk_large_minpoly_un}. The second follows from the triangle inequality, together with the two parts of Proposition \ref{prop_deltabk_small_minpoly_un}.
\end{proof}

Let $g\in(\resringr{k})[x]$ be a monic polynomial of degree $n$. Now we define the discrepancy of a polynomial $\overline{f}\in\field_{q^2}[x]_{=n}^\mon, \overline{f}\in I(\overline{g},n-d)$ by 

$$
\Delta_{\overline{f},k}^g=\Pr[\charc(A)\in I(g,n-d)|\charc(\overline{A})=\overline{f}]-q^{-2d(k-1)}.
$$

\begin{lem}\label{lem_discrepancy_bound_interval_starting_from_f_mod_p_un}
Let $\overline{f}\in \field_{q^2}[x]_{=n}^\mon$ be a $\sim$-symmetric polynomial, and let $\mathcal{G}$ be a set of representatives for the possible lifts of $\overline{f}$ to a monic polynomial of degree $n$ in $(\resringr{k})[x]$, modulo equality in the $d$ next-to-leading coefficients.
\begin{enumerate}
    \item If $\delta(\deg\rad\overline{f})\ge d$, then $\sum_{g\in\mathcal{G}}|\Delta_{\overline{f},k}^g
    |=0$. 
    \item If $\delta(\deg\rad\overline{f})=s<d$, then $\sum_{g\in\mathcal{G}}|\Delta_{\overline{f},k}^g|\le 2\Pr_{B\in U_{n}(\field_q)}[\deg\min(B)<d|\charc(B)=\overline{f}]$.
\end{enumerate}
\end{lem}

\begin{proof}
The proof is the same as that of Lemma \ref{lem_discrepancy_bound_interval_starting_from_f_mod_p}, using Propositions \ref{prop_deltabk_large_minpoly_un}, \ref{prop_deltabk_small_minpoly_un}. 
\end{proof}

\subsubsection{Discrepancy bound on distribution in intervals}

As usual we define for a monic polynomial $g\in(\resringr{k})[x]$ of degree $n$ the discrepancy of the surrounding interval,

$$
\Delta_{g}=\Pr[\charc A\in I(g,n-d)]-q^{-2kd}.
$$

\begin{lem}\label{lem_discrepancy_bound_single_interval_un}
Let $d<n/2$, and let $\mathcal{G}$ be a set of representatives for the monic polynomials $g\in(\resringr{k})[x]$ of degree $n$, modulo equivalence in the $d$ next-to-leading coefficients. Then,

\begin{equation*}
    \sum_{g\in\mathcal{G}} |\Delta_g|=O\left(
q^{2d+\frac{n}{2}-\frac{n^2}{4d}}\left(1+\frac{1}{q-1}\right)^{n}\binom{n+2d-3}{n}+q^{4d-\frac{n^2}{2d}+o(n)}
\right).
\end{equation*}
\end{lem}

\begin{proof}
As in the proof of Lemma \ref{lem_average_discrepancy_bound_gln}, write for $g\in\mathcal{G}$

\begin{equation*}
\Pr[\charc A\in I(g,n-d)]=q^{-2(k-1)d}\Pr[\charc\overline{A}\in I(\overline{g},n-d)]+
\sum_{\overline{f}\in I(\overline{g},n-d)}\Delta_{\overline{f},k}^g\cdot \Pr[\charc\overline{A}=\overline{f}].
\end{equation*}
By \cite[Theorem 5.16]{GR21}, 

$$
\Pr[\charc\overline{A}\in I(\overline{g},n-d)]=q^{-2d}+O\left(
q^{\frac{n}{2}-\frac{n^2}{4d}}\left(1+\frac{1}{q-1}\right)^{n}\binom{n+2d-3}{n}
\right).
$$
Hence, 

$$
|\Delta_g|<O\left(
q^{-(k-1)d+\frac{n}{2}-\frac{n^2}{4d}}\left(1+\frac{1}{q-1}\right)^{n}\binom{n+2d-3}{n}
\right)+\sum_{\overline{f}\in I(\overline{g},n-d)}\Delta_{\overline{f},k}^g\cdot \Pr[\charc\overline{A}=\overline{f}].
$$
Summing this over all $g\in\mathcal{G}$, and using Lemma \ref{lem_discrepancy_bound_interval_starting_from_f_mod_p_un}, we get that 

\begin{multline*}
\sum_{g\in\mathcal{G}}|\Delta_g|
<O\left(
q^{-(k-1)d+\frac{n}{2}-\frac{n^2}{4d}}(1+\frac{1}{q-1})^{n}\binom{n+2d-3}{n}
\right)+\\
+\sum_{\substack{\overline{f}\in\field_{q^2}[x]\\
\delta(\deg\rad \overline{f})<d}}2\Pr_{B\in U_n(\field_q)}[\delta(\deg\min(B))<d|\charc(B)=\overline{f}]\Pr[\charc B=\overline{f}].
\end{multline*}

Now if $\delta(\deg\min(B))<d$ in particular $\deg\min(B)<2d$, hence using Claim \ref{claim_bounding_probability_small_minpoly_un},

$$\Pr_{B\in U_n(\field_q)}[\delta(\deg\min(B))<d|\charc(B)=\overline{f}]\Pr[\charc B=\overline{f}]<\Pr[\deg\min(B)<2d\cap \charc(B)=\overline{f}]<q^{-\frac{n^2}{2d}+o(n)},$$
and by Lemma \ref{lem_small_radical_bound}, the number of polynomials in $\field_{q^2}[x]_{=n}^\mon$ with a radical of degree $<2d$ is $q^{4d+o(n)}$, hence we get the result.   
\end{proof}

\subsection{Joint distribution of traces}

As in the $GL_n$ case, the total variation distance of $\TR_d^{U}$ from $\UTR_d$  may be computed as the limit of the total variation distance of the $\mathcal{O}/\pi^k\mathcal{O}$ random variable 
 $\TR_d^{U,k}$, $\UTR_d^k$. We recall that by the discussion of \cite[Theorem 5.16]{GR21}, There is a constant $c_q^{U}$ such that for $d<c_q^{U}\cdot n$,

$$
q^{\frac{n}{2}-\frac{n^2}{4d}}\left(1+\frac{1}{q-1}\right)^{n}\binom{n+2d-3}{n}=o(1).
$$

\begin{theorem}
Let $d<\frac{n}{2}$, and let $\{a_i\}_{1\le i\le d}$ be a trace datum. Let $\TR_d^{U,k}$ be the random variable attaching to a matrix $M\in U_{n}(\mathcal{O}/\pi^k\mathcal{O})$ its trace datum of length $d$, and let $\UTR_d^k$ be the uniform measure on trace data, as defined in Section \ref{section_trace_data_poly_intervals}. Then,

\begin{multline*}
d_{TV}(\mu_{\TR_d^{U,k}},\mu_{\UTR_d^k})=\sum_{\{a_i\}_{1\le i\le d}\text{ trace data}}
\Bigg|
\Pr_{M\in U_{n}(\mathcal{O}/\pi^k\mathcal{O})}[M\text{ has trace datum }\{a_i\}_{1\le i\le d}]-q^{-(kd-S(d,k))}
\Bigg|<\\
<O\left(
q^{2d+\frac{n}{2}-\frac{n^2}{4d}}(1+\frac{1}{q-1})^{n}\binom{n+2d-3}{n}+q^{4d-\frac{n^2}{2d}+o(n)}
\right).
\end{multline*}
\end{theorem}

\begin{proof}
The proof is omitted, as it is essentially the same as the analogous one for $GL_n(\Ol)$.
\end{proof}

By taking $k\to\infty$ and using Lemma \ref{lem_total_variation_dist_limit} we get as a consequence Theorem \ref{main_theorem_un}.

\bibliography{references}
\bibliographystyle{alpha}

\end{document}